\documentclass[11pt]{article}
\usepackage[abs]{overpic}
\usepackage[utf8]{inputenc}
\usepackage[hmarginratio={1:1},     
  vmarginratio={1:1},     
  textwidth=16.5cm,        
  textheight=24cm,
  heightrounded,]
 {geometry}
 
\usepackage{tocloft}
\usepackage{enumitem}
 
\usepackage{amsmath,accents}
\usepackage{amsthm}
\usepackage{amssymb}
\usepackage{mathrsfs, mathtools}
\usepackage[mathscr]{euscript}
\mathtoolsset{showonlyrefs}

\let\emph\undefined
\newcommand{\emph}[1]{\textsl{#1}}

\usepackage[hidelinks]{hyperref}

\numberwithin{equation}{section}
\newtheoremstyle{style1}
  {13pt}
  {13pt}
  {}
  {}
  {\normalfont\bfseries}
  {.}
  {.5em}
  {}
\theoremstyle{style1}

\newtheorem{definition}[equation]{Definition}
\newtheorem{example}[equation]{Example}
\newtheorem{remark}[equation]{Remark}

\newtheorem{construction}[equation]{Construction}

\newtheoremstyle{style2}
  {13pt}
  {13pt}
  {\slshape}
  {}
  {\normalfont\bfseries}
  {.}
  {.5em}
  {}

\theoremstyle{style2}

\newtheorem{lemma}[equation]{Lemma}
\newtheorem{theorem}[equation]{Theorem}
\newtheorem{proposition}[equation]{Proposition}
\newtheorem{corollary}[equation]{Corollary}

\usepackage{dsfont}
\usepackage{tikz}
\usetikzlibrary{matrix,arrows,decorations.pathmorphing}
\usepackage{tikz-cd}

\usepackage{wasysym}

\usepackage{tikz}

\definecolor{Blue} {rgb} {0.282352,0.239215,0.803921}
\definecolor{Green} {rgb} {0.133333,0.545098,0.133333}
\definecolor{Red}   {rgb} {0.803921,0.000000,0.000000}
\definecolor{Violet}{rgb} {0.580392,0.000000,0.827450}

\newcounter{jfc}

\newcommand{\R}{\mathbb{R}}
\newcommand{\C}{\mathbb{C}}
\renewcommand{\P}{\mathbb{P}}
\newcommand{\Z}{\mathbb{Z}}

\newcommand{\N}{\mathbb{N}}
\newcommand{\bbS}{{\mathbb{S}}}
\newcommand{\bbD}{{\mathbb{D}}}
\newcommand{\bbT}{{\mathbb{T}}}

\newcommand{\Fa}{\mathcal{F}}
\newcommand{\Ga}{\mathcal{G}}
\newcommand{\Ta}{\mathcal{T}}
\newcommand{\Sa}{\mathcal{S}}
\newcommand{\Ha}{\mathcal{H}}
\newcommand{\Ra}{\mathcal{R}}

\newcommand{\Aa}{{\mathcal{A}}}
\newcommand{\Ba}{{\mathcal{B}}}
\newcommand{\Ja}{\mathcal{J}}
\newcommand{\La}{\mathcal{L}}
\newcommand{\Ia}{\mathcal{I}}

\newcommand{\Ua}{{\mathcal{U}}}

\newcommand{\ol}[1]{\overline{#1}}
\newcommand{\ul}[1]{\underline{#1}}
\newcommand{\Ad}{{\mathrm{Ad}}}
\newcommand{\colim}{\operatorname{colim}}

\newcommand{\Set}{\operatorname{\mathscr{S}et}}
\newcommand{\Grpd}{{\mathscr{G}\mathrm{rpd}}}
\newcommand{\Grp}{{\mathscr{G}\mathrm{rp}}}

\newcommand{\midwedge}{\text{\Large$\wedge$}}

\newcommand{\Cscr}{\mathscr{C}}
\newcommand{\Fscr}{\mathscr{F}}

\newcommand{\Hscr}{\mathscr{H}}
\newcommand{\Dscr}{\mathscr{D}}

\newcommand{\Ascr}{{\mathscr{A}}}

\newcommand{\holim}{\operatorname{holim}}
\newcommand{\hocolim}{\operatorname{hocolim}}
\newcommand{\sSet}{{\mathscr{S}\mathrm{et}_{\hspace{-0.04cm}\Delta}}}

\newcommand{\cH}{{\check{\mathrm{H}}}}
\newcommand{\rmB}{{\mathrm{B}}}

\newcommand{\rmH}{\mathrm{H}}

\newcommand{\Sym}{{\mathrm{Sym}}}
\newcommand{\ev}{{\mathrm{ev}}}
\newcommand{\coev}{{\mathrm{coev}}}
\newcommand{\Bun}{{\mathrm{Bun}}}
\newcommand{\Dfg}{\operatorname{\mathscr{D}fg}}
\newcommand{\Gau}{{\mathrm{Gau}}}

\newcommand{\dd}{\mathrm{d}}
\newcommand{\pr}{\operatorname{pr}}
\newcommand{\Cat}{\operatorname{\mathscr{C}at}}

\newcommand{\sfH}{{\mathsf{H}}}
\newcommand{\sfZ}{{\mathsf{Z}}}
\newcommand{\sfE}{{\mathsf{E}}}
\newcommand{\sfX}{{\mathsf{X}}}
\newcommand{\sfP}{{\mathsf{P}}}
\newcommand{\sfu}{{\mathsf{u}}}
\newcommand{\sfa}{{\mathsf{a}}}
\newcommand{\sfr}{{\mathsf{r}}}
\newcommand{\sfl}{{\mathsf{l}}}
\newcommand{\sfR}{{\mathsf{R}}}
\newcommand{\sfc}{{\mathsf{c}}}
\newcommand{\sfs}{{\mathsf{s}}}
\newcommand{\sfB}{{\mathsf{B}}}

\newcommand{\sfL}{{\mathsf{L}}}
\newcommand{\sfG}{{\mathsf{G}}}

\newcommand{\sfC}{{\mathsf{C}}}
\newcommand{\sfD}{{\mathsf{D}}}
\newcommand{\sfA}{{\mathsf{A}}}
\newcommand{\sfg}{{\mathsf{g}}}

\newcommand{\pt}{{\mathtt{pt}}}
\newcommand{\PT}{{\mathtt{PT}}}
\newcommand{\PSh}{{\mathtt{PSh}}}

\newcommand{\hol}{\operatorname{hol}}

\newcommand{\End}{\mathsf{End}}
\newcommand{\Hom}{\mathsf{Hom}}
\newcommand{\Aut}{\mathsf{Aut}}

\newcommand{\id}{\text{id}}

\newcommand{\HLBdl}{\mathrm{HLBdl}}
\newcommand{\BGrb}{\mathrm{BGrb}}

\newcommand{\D}{{D\hspace{-0.25cm}\slash}{}}

\newcommand{\iu}{\mathrm{i}}

\newcommand{\qandq}{\qquad \text{and} \qquad}
\newcommand{\rk}{\mathrm{rk}}
\newcommand{\String}{{\mathrm{String}}}

\newcommand{\DS}{\big/\hspace{-0.15cm}\big/}
\newcommand{\U}{{\operatorname{U}}}

\let\to\undefined
\newcommand{\to}{\longrightarrow}
\let\mapsto\undefined
\newcommand{\mapsto}{\longmapsto}

\newcommand{\opp}{\text{op}}

\newcommand{\Diff}{\mathrm{Diff}}
\newcommand{\PU}{{\mathbb{P}\mathrm{U}}}
\newcommand{\Des}{{\mathfrak{Des}}}
\newcommand{\Cart}{{\operatorname{\mathscr{C}art}}}
\newcommand{\Mfd}{\operatorname{\mathscr{M}fd}}

\newcommand{\qen}{\hfill $\vartriangleleft$}

\DeclareMathSymbol{\Phiit}{\mathalpha}{letters}{"08} 
\DeclareMathSymbol{\Psiit}{\mathalpha}{letters}{"09}
\DeclareMathSymbol{\Sigmait}{\mathalpha}{letters}{"06}
\DeclareMathSymbol{\Xiit}{\mathalpha}{letters}{"04}
\DeclareMathSymbol{\Piit}{\mathalpha}{letters}{"05}\let\Pi\undefined\newcommand{\Pi}{\Piit}
\DeclareMathSymbol{\Gammait}{\mathalpha}{letters}{"00}
\DeclareMathSymbol{\Omegait}{\mathalpha}{letters}{"0A}
\DeclareMathSymbol{\Upsilonit}{\mathalpha}{letters}{"07}
\DeclareMathSymbol{\Thetait}{\mathalpha}{letters}{"02}

\let\Phi\undefined\newcommand{\Phi}{\Phiit}
\let\Sigma\undefined\newcommand{\Sigma}{\Sigmait}
\let\Psi\undefined\newcommand{\Psi}{\Psiit}

\newenvironment{myitemize}{\begin{itemize}[itemsep=-0.05cm, leftmargin=*, topsep=0.1cm]}{\end{itemize}}
\newenvironment{myenumerate}{\begin{enumerate}[itemsep=-0.05cm, leftmargin=*, topsep=0.1cm, label={(\arabic*)}]}{\end{enumerate}}

\begin{document}

\begin{flushright}
\small
\textsf{Hamburger Beiträge zur Mathematik Nr.\,834}\\
{\sf ZMP--HH/20--10} \\
{\sf EMPG--20--08} \\
\end{flushright}

\vspace{10mm}

\begin{center}
	\textbf{\LARGE{Smooth 2-Group Extensions \\[3mm] and
            Symmetries of Bundle Gerbes}}\\
	\vspace{1cm}
	{\large Severin Bunk$^{a}$}, \ \ {\large Lukas Müller$^{b}$} \ \ and \ \ {\large Richard J. Szabo$^{c}$}

\vspace{5mm}

{\em $^a$ Fachbereich Mathematik, Bereich Algebra und Zahlentheorie\\
Universit\"at Hamburg\\
Bundesstra\ss e 55, D\,--\,20146 Hamburg, Germany}\\
Email: {\tt  severin.bunk@uni-hamburg.de\ }
\\[7pt]
{\em $^b$ Max-Planck-Institut f\"ur Mathematik\\
Vivatsgasse 7, D\,--\,53111 Bonn, Germany}\\
Email: {\tt lmueller4@mpim-bonn.mpg.de\ }
\\[7pt]
{\em $^c$ Department of Mathematics\\
Heriot-Watt University\\
Colin Maclaurin Building, Riccarton, Edinburgh EH14 4AS, U.K.}\\
and {\em Maxwell Institute for Mathematical Sciences, Edinburgh, U.K.}\\
and {\em Higgs Centre for Theoretical Physics, Edinburgh, U.K.}\\
Email: {\tt r.j.szabo@hw.ac.uk\ }
\end{center}

\vspace{1cm}

\begin{abstract}
\noindent
We study bundle gerbes on manifolds $M$ that carry an action of a connected Lie group $G$.
We show that these data give rise to a smooth 2-group extension of $G$ by the smooth 2-group of hermitean line bundles on $M$.
This 2-group extension classifies equivariant structures on the bundle gerbe, and its non-triviality poses an obstruction to the existence of equivariant structures.
We present a new global approach to the parallel transport of a bundle
gerbe with connection, and use it to give an alternative construction
of this smooth 2-group extension in terms of a homotopy-coherent version of the associated bundle construction.
We apply our results to give new descriptions of nonassociative magnetic translations in quantum mechanics and the Faddeev-Mickelsson-Shatashvili anomaly in quantum field theory.
We also propose a definition of smooth string 2-group models within our geometric framework.
Starting from a basic gerbe on a compact simply-connected Lie group $G$, we prove that the smooth 2-group extensions of $G$ arising from our construction provide new models for the string group of $G$.
\end{abstract}

\newpage

\tableofcontents

\newpage

\section{Introduction}

This paper is motivated by the following problem from physics: In~\cite{BMS:NA_translations} we showed how a bundle gerbe with connection on
$\R^d$ gives rise to a 3-cocycle on the translation group $\R^d_{\mathtt{t}}$ of $\R^d$.
Even though this 3-cocycle is trivial in group cohomology, it is very interesting from a physical as well as from a mathematical perspective:
it gives a geometric explanation to the presence of nonassociativity
in quantum mechanics with magnetic monopole backgrounds, and it
implements the action of the parallel transport of a bundle gerbe on
its 2-Hilbert space of sections. This appearence of nonassociativity in
quantum mechanics goes back to~\cite{Jackiw:3-cocycles,Gunaydin:1985ur},
but as of yet the more natural extension to realistic scenarios
involving periodically confined motion on configuration spaces such as tori
$\bbT^d$ has not been worked out. The discussion
of~\cite{Jackiw:3-cocycles} was a response to the observed violation
of the Jacobi identity for the algebra of field operators in quantum gauge theories with
chiral fermions~\cite{Jo}, which is a manifestation of the chiral anomaly. Interest in these
models has been recently revived through their conjectural relevance
to non-geometric flux compactifications of string theory,
which is based on backgrounds that are tori or more
generally torus
bundles~\cite{Lust:2010iy,MSS:NonGeo_Fluxes_and_Hopf_twist_Def,Bakas:2013jwa,Mylonas:2013jha}. However,
the original finding~\cite{Blumenhagen:2010hj} of 
nonassociativity in Wess-Zumino-Witten models based on other compact
Lie groups has so far received considerably less attention, and in
particular has not been understood from a geometric perspective.

In the present paper we work out the geometric framework and origin behind these results in complete generality.
Subsequently, we present several applications of our results in both
physics and mathematics, along the lines discussed above.
We consider an action $\Phi \colon G \times M \to M$ of a connected Lie group $G$ on a manifold $M$, where $M$ is endowed with a bundle gerbe $\Ga$.
One can now ask whether $\Ga$ admits a $G$-equivariant structure.
At the very least, such a structure should consist of a choice of 1-isomorphism $\Ga \to \Phi_g^*\Ga$ for every $g \in G$.
Instead of considering possible choices for such 1-isomorphisms individually, we assign to $g$ the groupoid of all such 1-isomorphisms.
This yields an object which can be understood as a bundle $\Sym_G(\Ga) \to G$ of groupoids over $G$.
Considering $g=e$, the identity element of $G$, we see that its typical fibre is the groupoid $\HLBdl(M)$ of hermitean line bundles on $M$.

The definition of $\Sym_G(\Ga)$ so far does not capture the smooth structure of the gerbe $\Ga$.
We thus enhance the construction to take into account smooth families of elements of $G$.
Then one can make sense of $\Sym_G(\Ga)$ as a category fibred in groupoids over a base category $\Cart$ that encodes smooth families of geometric objects.
Categories fibred in groupoids over $\Cart$ assemble into a 2-category $\Hscr$, and there exists a fully faithful inclusion of the category of smooth manifolds into $\Hscr$.
Motivated by~\cite{SP:String_group} we define a smooth 2-group to be a group object in $\Hscr$.
One of the central examples for us is the smooth 2-group $\HLBdl^M$ of hermitean line bundles on $M$.
We introduce a notion of smooth principal 2-bundle in $\Hscr$ that lies between the definitions of higher principal bundles used in~\cite{SP:String_group} and~\cite{NSS:oo-bdls_I} (see in particular Appendix~\ref{app:2-buns are oo-buns}).
We show that our principal 2-bundles are well behaved from a homotopical as well as from a geometric point of view (more precisely, they form effective epimorphisms while also admitting local sections).
With the notion of smooth 2-group and principal 2-bundles, we can make precise what it means to be a (central) extension of smooth 2-groups in analogy to extensions of Lie groups.
Then, our first main results can be summarised as

\begin{theorem}
Let $G$ be a connected Lie group acting on a manifold $M$, and let $\Ga$ be a bundle gerbe on $M$.
Then:
\begin{myenumerate}
\item There is a (non-central) extension of smooth 2-groups
\begin{equation}
\label{eq:Sym SES}
\begin{tikzcd}
	1 \ar[r] & \HLBdl^M \ar[r] & \Sym_G(\Ga) \ar[r] & \ul{G} \ar[r] & 1 \ ,
\end{tikzcd}
\end{equation}
where $\ul{G} \in \Hscr$ denotes the category fibred in groupoids associated to $G$.

\item The smooth 2-group $\Sym_G(\Ga)$ acts on $\Ga$, and the action covers that of $G$ on $M$.

\item The gerbe $\Ga$ admits a $G$-equivariant structure if and only if there exists a morphism $\ul{G} \to \Sym_G(\Ga)$ of smooth 2-groups which splits the extension~\eqref{eq:Sym SES}.
\end{myenumerate}
\end{theorem}

An extension similar to~\eqref{eq:Sym SES} was considered
in~\cite{FRS:Higher_U(1)-connections}, where symmetries of a gerbe with connection were investigated in relation with higher geometric prequantisation.
Infinitesimal versions of the extension~\eqref{eq:Sym SES} were considered in~\cite{Collier:Inft_Symmetries_of_gerbes,FRS:Higher_U(1)-connections}, where it was shown that these give rise to the standard $H$-twisted Courant algebroid on $M$, where $H$ is the 3-form curvature of the connection on $\Ga$.
These considerations have been expanded on and applied to higher versions of Kaluza-Klein reductions of string theory in~\cite{Alfonsi:DFT_and_KK}.

Our point here is that in many applications, such as nonassociativity
in quantum mechanics and string theory, anomalies in quantum field theory, as well as interesting topological constructions, connections on $\Ga$ only play a secondary role: in this context, they can be seen as a tool to compute the extensions~\eqref{eq:Sym SES} and their associated cocycles.
The key to this computability is an alternative presentation of $\Sym_G(\Ga)$ in terms of a categorified descent construction.

In order to work out this construction, we introduce a novel global approach to the parallel transport of a bundle gerbe.
Parallel transport for gerbes has been constructed in~\cite{SW:Fctrs_v_forms,SW:PT_and_Fctrs,SW:Local_2-fctrs}, but for our purposes a global, rather than local, treatment is necessary.
Our construction relies heavily on the transgression-regression
machine for bundle gerbes~\cite{Waldorf:Transgression_II} together with the properties of the fusion product and the connection on the transgression line bundle that were studied in~\cite{Waldorf:Transgression_II,BW:Transgression_of_D-branes}.
Given a connection on $\Ga$, we construct its parallel transport as a quadruple $\pt^\Ga = (\pt^\Ga_1, \pt^\Ga_2, \pt^\Ga_\star, \varepsilon^\Ga)$, consisting of the following data:
first, there is a 1-isomorphism $\pt_1^\Ga \colon \ev_0^*\Ga \to \ev_1^*\Ga$ over the path space $PM$ of $M$, where $\ev_t \colon PM \to M$ is the evaluation of a path at $t \in[0,1]$.
Second, there is a 2-isomorphism \smash{$\pt_2^\Ga \colon \pt_{1|\gamma_0}^\Ga \to \pt_{1|\gamma_1}^\Ga$} for every smooth homotopy with fixed endpoints between paths $\gamma_0$ and $\gamma_1$, which depends smoothly on the paths and the homotopy.
The 2-isomorphisms $\pt_\star^\Ga$ and $\varepsilon^\Ga$ implement the
compatibility of the parallel transport with concatenation of paths
and with constant paths, respectively.
Furthermore, the collection $\pt^\Ga$ is required to be invariant
under thin homotopies in a precise way. We show
\begin{theorem}
Every bundle gerbe with connection has a canonical parallel
transport. 
\end{theorem}

Using the parallel transport, we are able to write down a
$\HLBdl^M$-valued \v{C}ech 1-cocycle on the covering of $G$ by its
space of based paths.
These data are equivalently transition functions for an
$\HLBdl^M$-principal 2-bundle $\Des_\sfL \to \ul{G}\,$.
We construct $\Des_\sfL$ explicitly by a homotopy-coherent version of the associated bundle construction.
Then we prove

\begin{theorem}
The principal 2-bundle $\Des_\sfL \to \ul{G}$ is a smooth 2-group extension of $\ul{G}$ by $\HLBdl^M$.
There is a weakly commutative diagram of smooth 2-groups
\begin{equation}
\begin{tikzcd}[row sep=1cm]
	1 \ar[r] & \HLBdl^M \ar[r] \ar[d, equal] & \Sym_G(\Ga) \ar[r] \ar[d, "\Psi"] & \ul{G} \ar[r] \ar[d, equal] & 1
	\\
	1 \ar[r] & \HLBdl^M \ar[r] & \Des_\sfL \ar[r] & \ul{G} \ar[r] & 1
\end{tikzcd}
\end{equation}
The morphism $\Psi$ is an equivalence.
\end{theorem}

In the case $M = \R^d$, where $G = \R^d_{\mathtt{t}}$ is the translation
group of $\R^d$, and where $\Ga = \Ia_B $ is a trivial gerbe on
$\R^d$ with a connection $B \in\Omega^2(\R^d)$ corresponding to a
magnetic field, we show that the
extension $\Sym_{\R^d_{\mathtt{t}}}(\Ia ) \to \ul{\R}^d_{\,{\mathtt t}}$ reproduces the 3-cocycles we obtained in~\cite{BMS:NA_translations}.
We achieve this by choosing a certain global section of the path
fibration of $\R^d_{\tt t}$ and implicitly pass through $\Des_\sfL$ in the computation.
We show that the parallel transport we defined implements
nonassociative magnetic translations on the sections of the gerbe,
whereas the 2-group extension $\Sym_{\R^d_{\mathtt{t}}}(\Ia) \to
\ul{\R}^d_{\,{\mathtt t}}$ allows us to understand the algebraic
structure of nonassociative magnetic translations even without making any reference to sections.
The latter is particularly useful in cases where there is no good notion of sections, such as when the Dixmier-Douady class of $\Ga$ is non-torsion.
In particular, we study in detail the action of nonassociative
magnetic translations on tori $\bbT^d$ and give an explicit
description of \smash{$\Sym_{\R^d_{\mathtt{t}}}(\Ga)$} for general choice of a gerbe $\Ga$ on $\bbT^d$.

As a further application, we show that if ${\mit\Gamma}$ is a group of gauge transformations, the smooth 2-group extensions $\Sym_{\mit\Gamma}(\Ga) \to \ul{{\mit\Gamma}}$ control the Faddeev-Mickelsson-Shatashvili anomalies in quantum field theory~\cite{Faddeev:Op_Anomaly,FM:Nonabelian_Anomalies,Mickelsson:1983xi}.
The relation between gerbes and these anomalies has been investigated in~\cite{CM94,CM96}, but only as algebraic objects, disregarding the smooth structures.
The relevant bundle gerbe $\Ga$ lives on the space $\Ascr$ of gauge fields
and describes the obstruction to a Fock bundle descending to the orbit
space $\Ascr/{\mit\Gamma}$.
Here the extension $\Sym_{\mit\Gamma}(\Ga) \to \ul{{\mit\Gamma}}$ is split, so that $\Ga$ admits an equivariant structure.
At the same time $\Ga$ is trivialisable as a bundle gerbe, but the anomaly is precisely the obstruction to choosing a ${\mit\Gamma}$-equivariant trivialisation.
This allows us to understand the anomaly in a conceptual way as a higher smooth 1-cocycle on ${\mit\Gamma}$.

Finally, we consider the situation where $M = G$ is a compact simply-connected Lie group, acting on itself by left multiplication, and where $\Ga$ is a bundle gerbe on $G$ whose Dixmier-Douady class generates $\rmH^3(G;\Z) \cong \Z$.
We motivate and propose a new smooth string 2-group model for the string group of $G$.
For this, we first show that with our definition of principal 2-bundle, principal $\sfA$-bundles on a manifold give rise to $\sfA$-valued \v{C}ech 1-cocycles,  for any smooth 2-group $\sfA$.
Then we call a smooth 2-group extension $\sfA \to \sfP \to \ul{G}$ a smooth 2-group model for the string group of $G$ if $\sfA$ is equivalent to an Eilenberg-MacLane space $K(\Z;2)$ in a certain sense and the class in $\cH{}^1(G;\sfB\U(1)) \cong \rmH^3(G;\Z)$ extracted from the 2-bundle $\sfP \to \ul{G}$ is a generator.
Using this definition of smooth string 2-group models, we show

\begin{theorem}
Let $\Sym_G(\Ga)$ and $\Des_\sfL$ be the smooth 2-group extensions of $\ul{G}$ by $\HLBdl^G$ constructed from $\Ga$ with respect to the left action of $G$ on itself via left multiplication.
Then both $\Sym_G(\Ga)$ and $\Des_\sfL$ are smooth 2-group models for the string group of $G$.
\end{theorem}

The remainder of this paper is organised as follows.
In Section~\ref{Sec: Background} we briefly recall some background material on diffeological spaces, bundle gerbes, and transgression.
Section~\ref{Sec: U(1)} provides a motivation of the later
constructions on the level of principal bundles; many concepts become clear already at this level.
In Section~\ref{Sec: parallel transport} we provide our definition and construction of the parallel transport associated to a bundle gerbe with connection.
The construction of $\Sym_G(\Ga)$ and $\Des_\sfL$ takes place in Section~\ref{sect:extensions from gerbes}.
Here we first motivate and then introduce the necessary language of Grothendieck fibrations, smooth 2-groups, and principal 2-bundles, before defining and studying the extensions $\Sym_G(\Ga)$ and $\Des_\sfL$.
We conclude this section by relating these extensions to equivariant structures on $\Ga$.
In the remaining three sections we apply our general results:
in Section~\ref{Sec: Magnetic translations} we study nonassociative magnetic translations using our parallel transport, Section~\ref{Sec: Anomalies} contains the discussion of chiral anomalies and the Faddeev-Mickelsson-Shatashvili anomaly, and in Section~\ref{Sec:String} we show that $\Sym_G(\Ga)$ and $\Des_\sfL$ provide new models for the string group.
We defer some technical results on categories fibred in groupoids and on principal 2-bundles to Appendix~\ref{app:Principal 2-bundles}.

\section{Preliminaries on diffeological spaces and gerbes}
\label{Sec: Background}

In this section we review some of the relevant background material related 
to diffeological spaces and bundle gerbes that will be used throughout
this paper. 

\subsection{Diffeological spaces}

Throughout this paper we will use diffeological spaces (see \cite{book Diffeology} for an extensive
introduction) to describe 
the smooth structure on infinite-dimensional spaces 
such as path and mapping spaces. The idea behind diffeological spaces 
is to describe the smooth 
structure on a space $X$ by specifying the set of smooth maps from Cartesian spaces to $X$. A \emph{Cartesian space} 
$c$ is a smooth manifold diffeomorphic to $\R^n$ for some $n \in
\N_0$. We denote by $\Cart$ the category with 
Cartesian spaces as objects and smooth maps as morphisms.
\begin{definition}
A \emph{diffeological space} is a set $X$ together with a collection of 
maps $c \longrightarrow X$ from Cartesian spaces into $X$, called \emph{plots}, such that 
\begin{myenumerate}
\item the composition of a plot with a smooth map between Cartesian spaces is again a plot,

\item every map $\R^0\longrightarrow X$ is a plot, and

\item 
if $f\colon c \longrightarrow X$ is a map such that there 
exists an open cover $\{ c_i \}_{i\in I}$ of $c$ by Cartesian spaces and $f_{|c_i}$ 
is a plot for all $i\in I$, then $f$ is a plot.
\end{myenumerate} 
A map $f\colon X \longrightarrow Y$ between diffeological spaces
is \emph{smooth} if it maps plots of $X$ to plots of $Y$. 
We denote by $\Dfg$ the category of diffeological spaces and smooth maps.
Isomorphisms in $\Dfg$ are \emph{diffeomorphisms}.  
\end{definition}

\begin{remark}
Usually plots are defined to be maps from open subsets $U$ of Cartesian spaces to $X$.
Since every open subset $U$ of a Cartesian space can be covered by Cartesian spaces, both definitions
are equivalent. 
Diffeological spaces are exactly the concrete sheaves on 
the site of Cartesian spaces~\cite{BH:Diffelogical spaces}. This 
implies that the category of diffeological spaces $\Dfg$ admits all limits 
and colimits, and is Cartesian closed. For more background on this
perspective on diffeological spaces, see also~\cite{Bunk:2020ifw}.
\qen
\end{remark}

Important examples of diffeological spaces include the following.
\begin{example}
Every manifold $M$ (possibly with boundaries or corners) defines a diffeological space by declaring 
a map $f\colon c\longrightarrow M$ to be a plot if and only if $f$ is 
a smooth map of differentiable manifolds. This defines
a fully faithful embedding of the category of smooth manifolds $\Mfd$ into the 
category of diffeological spaces $\Dfg$. \qen
\end{example}

\begin{example}
Let $X$ be a diffeological space and $Y\subset X$ a subset.
We can equip $Y$ with a diffeology by declaring a map 
$c\longrightarrow Y$ to be a plot if and only if the composition 
with the embedding $Y\longrightarrow X$ is a plot. This is called the
subspace diffeology on $Y$. \qen
\end{example}

\begin{example}
Let $X$ and $Y$ be diffeological spaces. The Cartesian product 
$X\times Y$ is a diffeological space by declaring a map 
$f\colon c \longrightarrow X\times Y$ to be a plot if and only if 
$ \pr_X \circ f$ and $\pr_Y \circ f$ are plots, where $\pr_X$ and
$\pr_Y$ are the respective projections of $X\times Y$ onto $X$ and
$Y$. This is called the product diffeology on $X\times Y$. \qen
\end{example}

\begin{example}
\label{eg:Dfg examples}
Let $X$ and $Y$ be diffeological spaces. The set of smooth maps 
$Y^X$ from $X$ to $Y$ becomes a diffeological space by declaring 
a map $f\colon c \longrightarrow Y^X$ to be a plot if and only 
if the map 
\begin{align}
f^{\dashv}\colon c \times X &\longrightarrow Y \\ 
 (u, x) & \longmapsto f(u)(x) 
\end{align} 
is smooth. This is called the mapping space diffeology on $Y^X$. \qen 
\end{example}

A smooth map $f\colon M \longrightarrow M'$ between smooth manifolds is a surjective submersion if and only 
if it admits local sections through every point in $M$, i.e.~for every point $y \in M$ there exists an open neighbourhood $U$ of $f(y)$ in $M'$ and a smooth map $\widehat{s} \colon U \longrightarrow M$ such
that $ f \circ \widehat{s} = 1_U$ is the identity map of $U$.  
Surjective submersions define a Grothendieck topology on the category of manifolds, and many (higher) geometric objects on manifolds can be constructed via sheafification with respect to this topology (see, for instance,~\cite{NS:Equivar}).
On the category of diffeological spaces, a useful Grothendieck topology is induced by the subductions:

\begin{definition}
A smooth map $f\colon X \longrightarrow Y$ of diffeological 
spaces is a \emph{subduction} if for all plots $\varphi \colon c \longrightarrow Y$
and $x\in c$ there exists an open neighbourhood $U_x \subset c$ of 
$x$ and a plot $\widehat{\varphi}_x\colon U_x \longrightarrow X$ such 
that $\varphi_{|U_x}=f\circ \widehat{\varphi}_x$.
\end{definition}

\begin{example}
Let $M$ be a connected manifold. The space of paths in $M$ with 
sitting instants $PM$ is the subspace of $M^{[0,1]}$ of maps which 
are constant in an open neighbourhood of $0$ and $1$, equipped with
the subspace diffeology. The evaluation maps $\ev_0 \colon PM \longrightarrow M$
and $\ev_1 \colon PM \longrightarrow M$ at $0$ and $1$, respectively, are 
subductions.
\qen   
\end{example}

Another source for subductions are quotient maps. Let $X$ be a diffeological 
space and $\sim$ an equivalence relation on $X$. Then the space $X/{\sim}$ 
becomes a diffeological space in a canonical way making the map $\pi \colon X\longrightarrow X/{\sim}$ into 
a subduction: a map $\varphi\colon c \longrightarrow X/{\sim} $ is a plot if and only
if for all $x\in c$ there exists an open neighbourhood $U_x \subset c$ of 
$x$ and a plot $\widehat{\varphi}_x\colon U_x \longrightarrow X$ such 
that $\varphi_{|U_x}=\pi \circ \widehat{\varphi}_x$. Clearly all subductions are of this type for 
appropriate equivalence relations. 
Diffeological quotients behave nicely with respect to quotients of manifolds when they exist.

\begin{proposition}\label{Prop: quotients G and diffeology}
Let $M$ be a manifold with a free and proper action of a Lie group $G$. 
Define an equivalence relation $\sim_G$ on $M$ by $m_1\sim_G m_2$ if and only if 
there exists $g\in G$ such that $g\cdot m_1=m_2$.
Then the manifold $M/G$ and the diffeological space $M/{\sim}_G$ agree. 
\end{proposition}
\begin{proof}
From \cite[Theorem 21.10]{Lee} it follows that $\pi\colon M \longrightarrow M/G$ is 
a surjective submersion. Since every surjective submersion is a subduction, the 
statement follows.
\end{proof}

\begin{definition}
Let $X$ be a diffeological space and $k\geq 0$. A \emph{$k$-form $\omega$} on $X$ consists 
of a family of differential forms $\omega_\varphi \in \Omega^k(c)$ indexed by the plots
$\varphi \colon c \longrightarrow X$ of $X$ such that $\omega_{\varphi_1}= f^*\omega_{\varphi_2}$
for all commuting triangles
\begin{equation}
\begin{tikzcd}
c_1 \ar[rr,"f"] \ar[dr,"\varphi_1",swap]& & c_2 \ar[ld, "\varphi_2"] \\
 & X &
\end{tikzcd} 
\end{equation}
\end{definition} 
\begin{definition}[{\cite[Section 3]{WaldorfI}}]
Let $G$ be a Lie group and $X$ a diffeological space. A \emph{principal $G$-bundle} on
$X$ consists of a subduction $\pi\colon P\longrightarrow X$ together with a fibre-preserving 
right action $P\times G \longrightarrow P$ such that the map
\begin{align}\label{Eq: Definition Bundle}
P\times G & \longrightarrow P \times_X P \\
(p, g) & \longmapsto (p,p\cdot g)
\end{align}
is a diffeomorphism. 
A \emph{connection} on a principal $G$-bundle $P$ is a 1-form 
$A\in \Omega^1(P; \mathfrak{g})$ satisfying 
\begin{align}
\rho^* A = \Ad^{-1}_{\pr_G}(\pr_P^* A)+ \pr_G^* \theta 
\end{align}
on $P\times G$, where $\rho \colon P\times G \longrightarrow P$ is the right $G$-action,
$\theta$ is the left-invariant Maurer-Cartan 1-form on $G$, and
$\pr_P\colon P\times G \longrightarrow P$ and $\pr_G \colon P\times G \longrightarrow G$ are
the projections onto $P$ and $G$, respectively.
\end{definition}

\subsection{Bundle gerbes and transgression}
\label{Sec: Bundle gerbes}

Bundle gerbes are higher categorical analogues of line bundles. 
They provide a geometric realisation for the third cohomology group 
with integer coefficients. Similarly to line bundles, bundle gerbes can be equipped with connections.
We briefly recall the definition of the 2-groupoid of bundle gerbes and their transgression to loop space.
For details we refer to~\cite{Waldorf--More_morphisms,Waldorf:Transgression_II,Bunk--Thesis,Murray--Bundle_gerbes}.

Let $X$ be a diffeological space.
We denote by $\HLBdl(X)$ (resp. $\HLBdl^\nabla(X)$) the category of
hermitean line bundles (resp. with connection) on $X$. 
Before defining bundle gerbes we need to introduce some notation:
for a subduction $\pi\colon Y \longrightarrow X $ of diffeological spaces
we denote by \begin{align}
 Y^{[n]}= \big\{ (y_0,y_1,\dots, y_{n-1})\in Y^{\times n} \ \big| \ \pi(y_0)=\pi(y_1)=\dots =\pi(y_{n-1}) \big\}\subset Y^{\times n}
\end{align} 
the $n$-fold iterated fibre product $Y^{[n]}=Y\times_X\cdots\times_X
Y$ over $X$ equipped with the subspace diffeology.
Then $Y^{[\bullet]}$ is a simplicial diffeological space corresponding
to the subduction groupoid $Y\times_X Y \rightrightarrows Y$, and for
$k<n$ and 
$0\leq i_1 <\dots < i_k < n$ we define the smooth face maps
\begin{align}
\pi_{i_1,\dots, i_k} \colon Y^{[n]} & \longrightarrow Y^{[k]} \\ 
(y_0,y_1,\dots, y_{n-1}) & \longmapsto (y_{i_1}, \dots , y_{i_k}) \ .
\end{align}
\begin{definition}[{\cite{Waldorf:Transgression_II}}]
Let $X$ be a diffeological space. A \emph{hermitean bundle gerbe} on $X$ consists of a subduction $\pi \colon 
Y\longrightarrow X$, a hermitean line bundle $L\longrightarrow Y^{[2]}$, and a unitary isomorphism 
$\mu \colon \pi_{1,2}^*L\otimes \pi_{0,1}^*L \longrightarrow \pi_{0,2}^*L$ of line bundles
over $Y^{[3]}$, called the \emph{bundle gerbe multiplication}, which 
is associative over $Y^{[4]}$,
i.e. $\pi_{0,2,3}^*\mu\circ(\pi_{0,1,2}^*\mu\otimes1) = \pi_{0,1,3}^*\mu\circ(1\otimes\pi_{1,2,3}^*\mu)$.

A \emph{connection} on a hermitean bundle gerbe $\mathcal{G}=(\pi \colon Y\longrightarrow X, L, \mu)$
consists of a hermitean connection $\nabla^L$ on $L$ and a 2-form $B\in \Omega^2(Y)$ such that
\begin{myenumerate}
\item the isomorphism $\mu  \colon \pi_{1,2}^*L\otimes \pi_{0,1}^*L \longrightarrow \pi_{0,2}^*L$
is parallel with respect to $\nabla^L$, and 
\item the curvature of $\nabla^L$ is equal to $\iu\,(\pi_1^*B-\pi_0^*B)$.
\end{myenumerate}
The 2-form $B$ is called a \emph{curving}. The second condition
implies that the closed 3-form $\dd B=\pi^*H$ 
descends to a unique closed 3-form $H$ on $X$ with integer periods, which is called the \emph{curvature} of the bundle gerbe 
connection~$(\nabla^L, B)$ .
\end{definition}
Schematically, the data corresponding to a bundle gerbe can be
visualised by the
diagram
\begin{equation}
\begin{tikzcd}[column sep=1em, row sep=1cm]
\pi_{1,2}^*L\otimes \pi_{0,1}^*L \ar[rd] \ar[rr,"\mu"]& &  \pi_{0,2}^*L \ar[ld] & L \ar[d] &  & \\
& Y^{[3]} \ar[rr, shift left] \ar[rr, shift right] \ar[rr] & &
Y^{[2]}\ar[rr, shift left, "\pi_1"] \ar[rr, shift right, swap,"\pi_0"] & & Y \ar[d,"\pi"] \\
& & & & &  X
\end{tikzcd}
\end{equation} 
illustrating that hermitean bundle gerbes are equivalent to
$\U(1)$-central extensions of subduction groupoids.

\begin{example}
Let $X$ be a diffeological space. The \emph{trivial hermitean bundle gerbe $\mathcal{I}$ on $X$} consists of the identity
subduction $1_X \colon X\longrightarrow X$ together with the trivial hermitean line bundle 
$I \coloneqq X\times \C$ over $X^{[2]}=X$ and bundle gerbe multiplication 
\begin{align}
X\times (\C \otimes \C) & \longrightarrow X\times \C \\ 
\big(x, (z_1\otimes z_2)\big) & \longmapsto (x, z_1\, z_2) \ .
\end{align}   
For every 2-form $B \in \Omega^2(X)$ we can define a connection on
$\mathcal{I}$ by setting $\nabla^{I}=\dd$
and taking $B$ as the curving. We denote the resulting hermitean bundle gerbe with connection 
by $\mathcal{I}_B $. The curvature of $\Ia_B $ is given by $H=\dd B $. 
\qen
\end{example}
Hermitean bundle gerbes (resp. with connection) on a diffeological space $X$ are the objects of a 
symmetric monoidal bicategory 
which we denote by $\BGrb(X)$ (resp.~$\BGrb^\nabla (X)$)~\cite{Waldorf--More_morphisms}.  

\begin{definition}
\label{def:mps of BGrbs}
Let $\mathcal{G}=(\pi \colon Y\longrightarrow X, L, \mu, \nabla^L, B)$ and $\mathcal{G}'=(\pi' \colon Y'\longrightarrow X, L', \mu', \nabla^{L'}, B')$ be hermitean bundle gerbes with connection on 
a diffeological space $X$. A \emph{1-isomorphism} $\mathcal{G}\longrightarrow \mathcal{G}'$ of hermitean bundle 
gerbes (with connection) consists
of a subduction $\xi \colon Z\longrightarrow Y\times_X Y'$, a hermitean line bundle $E$ (with hermitean connection
$\nabla^E$) on $Z$ and (parallel) unitary isomorphisms 
$$  
\alpha \colon  \big((\pr_Y \circ\, \xi)^{[2]}\big) ^* L \otimes \xi_1^*E
\longrightarrow \xi_0^*E  \otimes \big((\pr_{Y'} \circ\, \xi)^{[2]}\big)
^* L' 
$$ 
over $Z^{[2]}$ satisfying a natural set of compatibility 
conditions, see~\cite{Waldorf:Transgression_II} for details. We will
denote such a 1-isomorphism by $(E,\xi)$ (resp. $(E,\xi,\nabla^E)$),
or sometimes simply by $E$.
\end{definition}

\begin{remark}
One can also define non-invertible 1-morphisms of bundle gerbes by
using higher rank hermitean vector bundles $E$ in Definition~\ref{def:mps of BGrbs}~\cite{Waldorf--More_morphisms}.
In that case, a 1-morphism is weakly invertible if and only if the underlying hermitean
vector bundle $E$ is of rank
$1$~\cite[Proposition~2.3.4]{waldorf}. However, with the exception of
Section~\ref{Sec: Magnetic translations}, we will only consider
invertible 1-morphisms of bundle gerbes in the
present paper.
\qen
\end{remark}

\begin{definition}
\label{def:2mps of BGrbs}
Let $(\xi_a \colon Z_a \longrightarrow Y\times_X Y',E_a, \nabla^{E_a}, \alpha_a)$ and
$(\xi_b \colon Z_b \longrightarrow Y\times_X Y',E_b, \nabla^{E_b}, \alpha_b)$ be 
1-isomorphisms $\mathcal{G}\longrightarrow \mathcal{G}'$ of hermitean
bundle gerbes with connection. 
A \emph{2-isomorphism} of bundle gerbes is an equivalence class of a subduction  
$\omega \colon W\longrightarrow Z_a \times_{Y\times_X Y'} Z_b$ and a parallel unitary 
isomorphism $ (\pr_{Z_a} \circ\, \omega)^*E_a \longrightarrow (\pr_{Z_b} \circ\, \omega)^*E_b$
satisfying a natural compatibility condition, see e.g.~\cite{Waldorf--More_morphisms} for details and the equivalence
relation. 
\end{definition}

Bundle gerbes on a diffeological space $X$ are classified by their
Dixmier-Douady class in $\rmH^3(X;\Z)$, analogously to the Chern-Weil
classification of line bundles by their Chern class in
$\rmH^2(X;\Z)$. For a bundle gerbe with connection, the Dixmier-Douady
class maps to the de~Rham cohomology class of the curvature under
the homomorphism $\rmH^3(X;\Z)\to\rmH^3(X;\R)$ induced by
the inclusion of coefficient groups $\Z\hookrightarrow\R$.

Let $\Ga$ be a hermitean bundle gerbe defined over a subduction $\pi:Y \to X$, with underlying 
hermitean line bundle $L \to Y^{[2]}$.
Let $\Aa \colon \Ga \to \Ga$ be an endomorphism of $\Ga$, with underlying hermitean vector bundle $A$ over some subduction $\xi:Z \to Y^{[2]}$.
Consider the hermitean vector bundle $L^\vee \otimes A$ on $Z$, where
we denote the dual line bundle by $L^\vee$.
This comes with a canonical descent isomorphism defined by the diagram~\cite{Waldorf--More_morphisms,Bunk--Thesis}
\begin{equation}
\label{eq:construction in Section 2}
\begin{tikzcd}[row sep=1cm, column sep=1.5cm]
	\xi_1^*(L^\vee \otimes A) \ar[r] \ar[dd, dashed]
	& \pi_{2,3}^*L^\vee \otimes \xi_1^*A \ar[d, "\pi_{0,2,3}^*\mu^{-1}"]
	\\
	& \pi_{0,2}^*L \otimes \pi_{0,3}^*L^\vee \otimes \xi_1^*A \ar[d]
	\\
	\xi_0^*(A \otimes L^\vee) & \pi_{0,3}^*L^\vee \otimes \xi_0^*A \otimes \pi_{1,3}^*L \ar[l, "\pi_{0,1,3}^*\mu"]
\end{tikzcd}
\end{equation}
In fact, this construction establishes an equivalence of categories $
\mathsf{R}\colon \BGrb(X)(\Ga, \Ga) \longrightarrow \HLBdl(X)$.

From a hermitean bundle gerbe with connection $\Ga$ on a 
diffeological space $X$ we can 
construct the \emph{transgression line bundle} $\Ta \Ga$ over the loop space $LX$
of $X$. The fibre $\Ta \Ga_{\gamma}$ over a loop $\gamma\colon \mathbb{S}^1 \longrightarrow X$ consists of
equivalence classes $[[\Sa],z]$ of a 2-isomorphism class of a trivialisation $\Sa\colon \gamma^* \Ga \longrightarrow 
\mathcal{I}_0$ in $\BGrb^\nabla(\mathbb{S}^1)$ over the unit circle $\mathbb{S}^1$ and an element $z\in \C$. Two pairs 
$([\Sa],z)$ and $([\Sa'],z')$ are equivalent if and only if $z'=\hol(\mathbb{S}^1,\mathsf{R}(\Sa'\circ \Sa^{-1}))\, z$.
For the construction of a diffeological structure on $\Ta \Ga \coloneqq \coprod_{\gamma\in LX}\, \Ta \Ga_{\gamma}$ we refer to~\cite{Waldorf:Transgression_II}. 
A connection on a line bundle over the loop space $LX$ is 
\emph{superficial} if the holonomy around every thin
loop\footnote{A loop $\Gamma \in LLX$ is \emph{thin} if the adjoint map
$\Gamma^\dashv\colon \mathbb{S}^1\times \mathbb{S}^1 \longrightarrow  X$ has
at most rank 1.} is equal to~$1$ and thin homotopic loops\footnote{Two loops 
$\Gamma, \Gamma'\in LLX$ are \emph{thin homotopic} if there exists
a homotopy $h\in PLLX$ such that the adjoint map $h^\dashv: [0,1] \times \mathbb{S}^1\times \mathbb{S}^1 \longrightarrow X$ has at most rank 2.} have the same holonomy.
In the situation where $X = M$ is a manifold, a superficial connection on $\Ta \Ga$ has been
constructed from the connection on $\Ga$ in~\cite[Prop.~3.3.1]{Waldorf:Transgression_II}; note that in our later constructions, we will always work with bundle gerbes over manifolds.
The bundle gerbe multiplication
induces, for all triples of paths $(\gamma_1,\gamma_2, \gamma_3)$ with
sitting instants and the same start and end points, a \emph{fusion product}
\begin{align}
\lambda\colon\Ta \Ga_{\overline{\gamma_2} \star \gamma_1} \otimes 
\Ta \Ga_{\overline{\gamma_3} \star \gamma_2} \longrightarrow 
\Ta \Ga_{ \overline{\gamma_3} \star \gamma_1} \ ,
\end{align}
where $\star$ denotes the concatenation of paths and $\overline{\gamma}$ is
the path $t\longmapsto \gamma(1-t)$.
The fusion product depends smoothly on the paths, is parallel with respect to 
the superficial connection, and is associative. The connection and fusion product
satisfy one further compatibility condition, related to the rotation of all loops
involved by $180^\circ$ (see~\cite[Definition 2.1.5]{Waldorf:Transgression_II}).
A line bundle over $LX$ admitting all the structures discussed above 
is a \emph{fusion line bundle with superficial connection}.
 
For $X = M$ a manifold, transgression extends to a functor $\Ta$ from $\mathtt{h}\BGrb^\nabla(M)$, the 1-category
obtained from $\BGrb^\nabla(M)$ by identifying isomorphic 1-morphisms, to 
the category of fusion line bundles with superficial connection over $LM$.
The central result of~\cite{Waldorf:Transgression_II} is that $\Ta$ defines
an equivalence of categories. An explicit inverse 
functor $\mathcal{R}$ is constructed
in~\cite{Waldorf:Transgression_II} and is
called \emph{regression}.  

\section{Group extensions from principal bundles}\label{Sec: U(1)}

In this section we construct group extensions from group actions on manifolds with principal bundles. 
We generalise this extension to
higher geometry in Section~\ref{sect:extensions from gerbes}. 
We present two perspectives on this group extension. The first one is
global. The second one is local 
and can be formulated in terms of the parallel transport of an
auxiliary connection on a principal bundle.

\subsection{Global description}
\label{sect:gauge bundles on groups}

Let $H$ be a Lie group and $P \to M$ a principal $H$-bundle on a
manifold $M$; principal $H$-bundles on $M$ and isomorphisms form a
groupoid which we denote by $\Bun_H(M)$.
We consider a Lie group action
\begin{align}
	\Phi \colon G \times M &\to M \\
 (g,x)&\longmapsto \Phi_g(x)=\Phi(g,x)
\end{align}
on the base manifold $M$, and ask whether and how this action lifts to $P$.
An action of a Lie group $G$ on $M$ can equivalently be written as a smooth homomorphism of groups $\Phi \colon G \to \Diff(M)$, where $\Diff(M)$ is the diffeological group of 
diffeomorphisms $M \to M$. In general, the action of $G$ does not lift to $P$.
Instead, we will construct a group extension 
\begin{align}
1\longrightarrow \Gau(P) \longrightarrow \Sym_G(P)\longrightarrow G \longrightarrow 1
\end{align} 
of $G$ by the gauge group $\Gau(P)$ of $P$.   The group $\Sym_G(P)$ acts on the total 
space $P$ in a way compatible with the action of $G$ on $M$. 
We show that it is the universal extension of $G$ having this 
property. 
\begin{remark}
The extension can be constructed as the pullback of
the short exact sequence 
\begin{align}
1 \longrightarrow \Gau(P) \longrightarrow \Aut_G(P) \longrightarrow \Diff_P(M) \longrightarrow 1
\end{align} 
of diffeological groups along $\Phi$, where $\Aut_G(P)$ is the group of $G$-equivariant diffeomorphisms 
of $P$ and $\Diff_P(M)$ is the subgroup of diffeomorphisms of $M$ which admit an equivariant lift to $P$. In the following we 
present a different construction which generalises directly to bundle gerbes.
\end{remark}

We can pull back the bundle $P$ along the source and target maps of
the action groupoid
\begin{equation}
\begin{tikzcd}
	G \times M \ar[rr, shift left,"\Phi"] \ar[rr, shift
        right,swap,"\pr_M"] & & M \ .
\end{tikzcd} 
\end{equation}
We define a bundle
\begin{equation}
\label{eq:def P_G}
	\Sym_G(P) \overset{\pi}{\to} G
	\qquad \mbox{with} \quad
	\Sym_G(P)_{|g} \coloneqq \Bun_{H}(M)(P, \Phi_g^*P)
\end{equation}
for all $g\in G$,
where $\Bun_{H}(M)(P, \Phi_g^*P)$ is the collection of gauge transformations from $P$ to $\Phi_g^* P$. 
In order for $\Sym_G(P)$ to be a bundle over $G$, we must ensure that
the fibres of $\Sym_G(P) $ are actually pairwise 
diffeomorphic.
It might happen that a pullback bundle $\Phi_g^*P$ is no longer isomorphic to $P$ and hence the fibre over $g$ 
is empty. As an example,
consider the action of the group $G=\Z$ on the 2-torus $M=\bbT^2$
generated by an orientation-reversing diffeomorphism $f$, and let $P
\to \bbT^2$ be a $\U(1)$-bundle with non-trivial Chern class.
Then $[f^*P] = -[P]$, and thus $\Sym_\Z(P)_{|1}=\Bun_{\U(1)}(P, f^*P) = \varnothing$.
Hence in \eqref{eq:def P_G} we have to ensure that the fibres of
$\Sym_G(P) $ are actually all non-trivial.

We restrict our attention to \emph{connected} Lie groups $G$; otherwise,
if $G$ is not connected,
we consider only the connected component of the identity $e\in G$. We show that in this case 
the fibres are always non-trivial:
we need to show that for any $g \in G$ the fibre of $\Sym_G(P) \to G$ over $g$ is non-empty.
That is, we need to show that there exists an isomorphism $ P \to \Phi_g^*P $ of $H$-bundles over $M$.
Let $f_P \colon M \to \sfB H$ be a map that classifies the bundle $P \to M$.
Then $\Phi_g^*P$ is classified by the map $f_P \circ \Phi_g \colon M \to \sfB H$.
Since $G$ is connected, we can find a smooth path $\gamma \colon [0,1] \to G$ with $\gamma(0) = e$ and $\gamma(1) = g$.
Consider the smooth map
\begin{align}
	\Phi_\gamma \colon [0,1] \times M &\to M \\
	(t,x) &\mapsto \Phi_{\gamma(t)}(x) \ .
\end{align}
We can postcompose this map by $f_P$ to obtain a homotopy
\begin{equation}
	f_P \circ \Phi_\gamma \colon [0,1] \times M \to \sfB H
\end{equation}
from $f_P$ to $f_P \circ \Phi_g$.
This shows that there exists a bundle isomorphism $P \to \Phi_g^*P $.
We note for later use that this argument generalises to $n$-gerbes $\Ga$, as these are classified by maps $f_\Ga \colon M \to \sfB^{n+1}\U(1)$.

In order to equip the set $\Sym_G(P)$ with a diffeology, we note that 
$\Sym_G(P)$ can be identified with the subspace of the Cartesian product of the
space of $H$-equivariant diffeomorphisms
$P\longrightarrow P$ which cover the action of an arbitrary element $g\in G$ on $M$ 
with $G$,
and equip $\Sym_G(P)$ with the subspace diffeology. Concretely, for $c\in \Cart$, a map $f\colon c \longrightarrow 
\Sym_G(P)$ is a plot if and only if the composition $\pi\circ f \colon c\longrightarrow G$ is smooth 
and the induced map $ \pr_M^*P \longrightarrow  \Phi_{f}^*P$ is an isomorphism in $\Bun_H(c\times M)$, where 
$\pr_M \colon c\times M \longrightarrow M $ is the projection onto $M$
and $\Phi_f=\Phi\circ(f\times 1_M)$. 
The automorphism group or group of gauge transformations
\begin{equation}
	\Gau(P) \coloneqq \Bun_{H}(M)(P,P)
\end{equation}
acts simply and transitively on each fibre $\Sym_G(P)_{|g}$ from the right via precomposition.
The set $\Gau(P)$ forms a diffeological group with respect to the composition of 
automorphisms and the smooth structure induced from the mapping space diffeology on $P^P$.

\begin{proposition}
$\pi \colon \Sym_G(P) \to G$ is a principal $\Gau(P)$-bundle on $G$.
\end{proposition}

\begin{proof}
We verify
that the map $\pi \colon \Sym_G(P) \to G$ is a subduction. Let $f\colon c\longrightarrow G$ be 
a plot. 
We can pick an isomorphism $\varphi_f \colon  \pr_M^*P   \longrightarrow \Phi_f^*P$ (since $c$ is contractible) and define the map 
\begin{align}
\widehat{f} \colon c &\longrightarrow \Sym_G(P) \\
 x & \longmapsto \varphi_{f|\{x \} \times M} \ .
\end{align}
The map $\widehat{f}$ is a smooth lift of the plot $f$, showing that $\Sym_G(P) \to G$
is a subduction. 

The map 
\begin{align}
\Sym_G(P)\times_G \Sym_G(P) &\longrightarrow \Sym_G(P)\times \Gau(P)  \\ 
(\varphi\colon P \longrightarrow \Phi_g^* P, \varphi' \colon P \longrightarrow \Phi_g^* P)& \longmapsto (\varphi,  \varphi^{-1} \circ \varphi')
\end{align} 
provides a smooth inverse to the map $\Sym_G(P)\times \Gau(P)
\longrightarrow \Sym_G(P) \times_G \Sym_G(P) $ from \eqref{Eq:
  Definition Bundle}, and the result follows.
\end{proof}
  
\begin{proposition}\label{Prop: multiplicative structure bundle}
$\Sym_G(P)$ is a diffeological group. The principal bundle $\Sym_G(P)
\to G$ is part of an extension of diffeological groups
\begin{equation}
1\longrightarrow \Gau(P)\longrightarrow \Sym_G(P)\longrightarrow G
\longrightarrow 1 \ .
\end{equation}
\end{proposition}

\begin{proof}
To complete the proof we need to equip $\Sym_G(P)$ with a
diffeological group structure such that the map $\Sym_G(P) \to G$ 
becomes a morphism of diffeological groups. 
Consider isomorphisms $\psi \colon P \to \Phi_g^*P$ and $\phi \colon P \to \Phi_{g'}^*P$ for $g,g' \in G$.
We set
\begin{equation}
	\mu(\psi, \phi) \coloneqq \Phi_{g'}^*\psi \circ \phi \ \colon
        P \to \Phi_{g'}^*P \to \Phi_{g'}^*\Phi_{g}^*P =
        \Phi_{g\,g'}^*P \ .
\end{equation}
This is associative by the associativity of pullbacks, the multiplication in $G$, and composition of morphisms. The inverse
of an element $\psi \colon P \to \Phi_g^* P$ with respect to $\mu$ is the isomorphism \[ 
P= \Phi_{g^{-1}}^*\Phi_{g}^*P \xrightarrow{\Phi_{g^{-1}}^*\psi^{-1}}
\Phi_{g^{-1}}^*P \ , \] 
and the result follows from the observation that these maps are
smooth. 
\end{proof}

\begin{proposition}\label{Prop: Action line bundle}
The group $\Sym_G(P)$ acts smoothly on $P$, lifting the action of $G$ on $M$.
It is universal in the following sense: 
let $\widehat{G}$ be a Lie group, $\varphi \colon \widehat{G} \longrightarrow G$
a Lie group homomorphism and $\widehat{\psi} \colon \widehat{G} \times P \longrightarrow P$
an action of $\widehat{G}$ on $P$ making the diagram
\begin{equation}
\begin{tikzcd}[row sep=1cm]
\widehat{G} \times P \ar[r, "\widehat{\psi}"] \ar[d, "\varphi\times \varpi", swap] & P \ar[d,"\varpi"]\\ 
G \times M \ar[r, swap ,"\Phi"] & M 
\end{tikzcd}
\end{equation} 
commute, where $\varpi \colon P \to M$ is the bundle projection. Then there exists a unique smooth group homomorphism $\widehat{G}\longrightarrow 
\Sym_G(P)$ such that the diagram
\begin{equation}
\begin{tikzcd}[row sep=1cm]
\widehat{G} \times P \ar[rd] \ar[drr, bend left=18] \ar[ddr, bend right] &  &  \\
& \Sym_G(P) \times P \ar[r] \ar[d, swap] & P \ar[d,]\\ 
 & G \times M \ar[r ] & M 
\end{tikzcd}
\end{equation}
commutes. 
\end{proposition}

\begin{proof}
The action is via the evaluation
\begin{align}
	\widehat{\Phi} \colon \Sym_G(P) \times P &\to P \\
	(\phi,p) &\mapsto \phi(p) = \phi_{|\varpi(p)}(p) \ .
\end{align}
The unique smooth group homomorphism in the universality statement is
\begin{align}
\widehat{G}&\longrightarrow \Sym_G(P) \\ 
\widehat{g}&\longmapsto \big(\widehat{\psi}_{\widehat{g}} \colon P \longrightarrow \Phi_{\varphi(\widehat{g})}^*P\big) \ ,
\end{align}                         
and the result follows.
\end{proof}

The construction of the group $\Sym_G(P)$ is functorial in $P$, i.e. 
an
isomorphism of bundles $\psi \colon P \longrightarrow P'$
induces an isomorphism of group extensions 
\begin{align}\label{Eq: Functoriality}
\widehat{\psi}\colon\Sym_G(P) & \longrightarrow\Sym_G(P') \\
( f\colon P\longrightarrow g^* P) &\longmapsto \Big( P' \xrightarrow{\psi^{-1}} P\overset{f}{\longrightarrow} g^* P \xrightarrow{g^*\psi} g^*P' \Big) \ .
\end{align}

\subsection{Equivariant bundles}\label{Sec: Eq bundles}

Let $G$ be a connected Lie group, $M$ a manifold with $G$-action $\Phi \colon G\times M \longrightarrow M$, and
$P$ a principal $H$-bundle over $M$. A \emph{$G$-equivariant structure} on 
$P$ consists of an isomorphism $\chi \colon \pr_M^* P \longrightarrow
\Phi^* P$ of principal bundles over $G\times M$ such that the diagram
\begin{equation}
\begin{tikzcd}
P_x \ar[rr,"\chi_{(g\,g',x)}"] \ar[rd,"\chi_{(g,x)}",swap]  & & P_{\Phi_{g\,g'}(x)} \\ 
 & P_{\Phi_{g}(x)} \ar[ru, " \Phi_{g}^* \chi_{(g',x)}",swap] &
\end{tikzcd}
\end{equation} 
commutes for all $g,g'\in G$ and $x\in M$. We denote by $\mathcal{E}(P)$
the set of equivariant structures on $P$.
A \emph{splitting} $s$ of $\pi \colon \Sym_G(P)\longrightarrow G$ is a smooth
group homomorphism $s\colon G \longrightarrow \Sym_G(P)$ such that
$\pi\circ s = 1_G$. We denote the set of splittings of $\pi \colon \Sym_G(P)\longrightarrow G$ by $\mathcal{S}(G;\Sym_G(P))$.
\begin{proposition}
There is a natural bijection of sets $\Xi \colon \mathcal{E}(P)\longrightarrow
\mathcal{S}(G;\Sym_G(P))$. In particular, the bundle $P$ admits 
an equivariant structure if and only if the extension
$$
1 \longrightarrow \Gau(P) \longrightarrow \Sym_G(P)\longrightarrow G
\longrightarrow 1
$$ 
is trivial as an extension of diffeological groups. 
\end{proposition}

\begin{proof}
Let $(P,\chi)$ be an equivariant bundle. We define 
$
\Xi (P,\chi) (g) \colon P \longrightarrow \Phi_g^*P 
$ to be $\chi_{|\{g\}\times M}$. The inverse $\Xi^{-1}\colon \mathcal{S}(G;\Sym_G(P))\longrightarrow \mathcal{E}(P)$ can be constructed by
sending a splitting $s\colon G \longrightarrow \Sym_G(P)$ to the
isomorphism $\Xi^{-1}(s) \colon \pr_M^* P \longrightarrow \Phi^* P $
which is given by
$s(g)(x)\colon P_x \longrightarrow P_{\Phi_g(x)}$ at $(g,x)\in G\times M$.
\end{proof}

Let $(P, \chi)$ and $(P', \chi')$ be $G$-equivariant $H$-bundles on $M$. An isomorphism $\psi \colon P \longrightarrow 
P'$ is \emph{equivariant} if the diagram
\begin{equation}
\begin{tikzcd}[row sep=1cm]
P \ar[r,"\psi"] \ar[d, "\chi_g ",swap]& P' \ar[d, "\chi_g' "] \\
\Phi_{g}^*P \ar[r,"\Phi_{g}^*\psi",swap] & \Phi_{g}^*P' 
\end{tikzcd}
\end{equation}
commutes for all $g\in G$.
The equivariant structures on $P$ and $P'$ can be described by smooth group homomorphisms
$s_P \colon G \longrightarrow \Sym_G(P)$ and    
$s_{P'} \colon G \longrightarrow \Sym_G(P')$. 
Since the isomorphism $\widehat{\psi} $ defined in 
\eqref{Eq: Functoriality}
intertwines the action of $\Sym_G(P)$ and $\Sym_G(P')$ on $P$ 
and $P'$, respectively, it follows that $\psi$ is equivariant if and
only if $s_{P'}= \widehat{\psi} \circ s_P$. Hence the smooth group
extension $\Sym_G$ contains all information
on equivariance. 

\subsection{Description via parallel transport}
\label{sect:non-triv fibs via connections}

The extension $\Sym_G(P)$ can be described more explicitly using the parallel transport of a
connection on $P$, as we will now explain. 
In Section~\ref{Sec: Magnetic translations} we apply this to the
description of magnetic translations in quantum mechanics. 
We consider a principal $H$-bundle $P \to M$.
Let $P_0 G$ denote the diffeological space of smooth paths in $G$ with sitting instants based at $e\in G$, $\ev_1\colon P_0G \longrightarrow G$ the evaluation at the end point, $(P_0G)^{[2]}$ the fibre product $P_0G\times_{G} P_0G$ with respect
to $\ev_1$, and $LM$ the space
of smooth loops in $M$. 
We denote by $\star$ the concatenation of paths. 
For a path $\gamma \colon [0,1] \longrightarrow G $ we denote by $\overline{\gamma}$ the precomposition 
of $\gamma$ with 
\begin{align}
[0,1]  &\longrightarrow [0,1] \\
t  &\longmapsto 1-t \ .
\end{align}                         
For a path $\gamma \in P_0 G$ and a point $x \in M$, set
\begin{align}
	\gamma_x \colon [0,1] &\to M \\
	t &\mapsto \Phi_{\gamma(t)} (x) \ .
\end{align}

Endow $P$ with an arbitrary connection $A$.
The $H$-bundle $P$ with connection then induces a principal $\Gau(P)$-bundle on $G$ as follows:
we set
\begin{equation}
	\La_G \coloneqq \big( P_0 G \times \Gau(P) \big) \big/{\sim} \
        ,
\end{equation}
where we define the equivalence relation
\begin{equation}
	(\gamma, \phi) \sim \big( \alpha, \hol(P, \alpha, \gamma)\circ \phi \big)
	\qquad \text{with} \quad
	\hol(P, \alpha, \gamma) (x) \coloneqq \hol \big( P, (
        \overline{\alpha}\star \gamma)_x  \big) \ \in \ \End(P_x)
\label{Eq: Decent description bundles}
\end{equation}
for all $(\gamma,\alpha)\in(P_0G)^{[2]}$ and $x\in M$, 
and we interpret the holonomy of $P$ along a loop starting and ending
at $x$ as an endomorphism of the fibre $P_x$.
Note that, with this notation, we have defined a smooth map $\hol(P,-) \colon (P_0 G)^{[2]} \to \Gau(P)$.
We endow $\La_G$ with
the quotient diffeology.

Then the $\Gau(P)$-bundle $\La_G \to G$ can be defined in terms of descent data 
as follows:
the action $\Phi$ of $G$ on $M$ induces a smooth map
\begin{align}\label{EQ: Descent description line bundle}
	L\Phi \colon (P_0 G)^{[2]} \times M &\to LM \\
	(\gamma, \alpha, x) &\mapsto (\overline{\alpha} \star \gamma )_x \ .
\end{align}
Explicitly, 
\begin{equation}
	( \overline{\alpha } \star \gamma)_x(t) = \Phi_{(\overline{\alpha }
          \star \gamma) (t)}(x) \ \in \ M
\end{equation}
for all $t \in [0,1]$ and $x \in M$.
The descent data for the bundle $\La_G$ consists of the subduction
$P_0G\to G$, the trivial bundle $P_0G\times\Gau(P)\to P_0G$, and the smooth map
\begin{align}
	\sfg \colon (P_0 G)^{[2]} &\to \Gau(P) \\
	(\gamma,\alpha)&\mapsto \sfg(\gamma,\alpha) = \hol ( P,
                         \overline{\alpha } \star \gamma) \ .
\end{align}

\begin{proposition}\label{Prop: multiplicative}
The total space $\La_G$ is a smooth group extension
\begin{equation}
\begin{tikzcd}
1 \ar[r] & \Gau(P) \ar[r] & \La_G \ar[r] & G \ar[r] & 1 \ .
\end{tikzcd}
\end{equation}
\end{proposition}

\begin{proof}
Let $\gamma$ and $\gamma'$ be two paths in $G$.
The evaluation $\ev_1 \colon P_0 G \to G$ is a group homomorphism with
respect to the pointwise product of paths.

Let $x \in M$ be an arbitrary point.
To any triple $(x, \gamma, \gamma')$, we can associate a map
\begin{align}
	\Delta^2(x,\gamma,\gamma') \colon |\Delta^2| &\to M \\
	(t_1,t_2) &\mapsto \Phi_{\gamma(t_1)}\big(\gamma'_x(t_2)\big)
                    \ ,
\end{align}
where $|\Delta^2|$ is the standard topological 2-simplex with $|\Delta^2| \cong \{ (t_1, t_2) \in \R^2\, | \, 0 \leq t_2 \leq t_1 \leq 1 \}$.
Diagrammatically, this is a homotopy
\begin{equation}
\begin{tikzcd}[row sep=1.25cm, column sep=1cm]
	& \gamma(1) \cdot \gamma'(1) \cdot x
	\\
	x \ar[r, "\gamma'(t) \cdot x"'] \ar[ur, "\gamma'(t) \cdot \gamma(t) \cdot x"] & \gamma'(1) \cdot x \ar[u, "\gamma(t) \cdot \gamma'(1) \cdot x"']
\end{tikzcd}
\end{equation}
between the product path $\gamma \, \gamma' \in P_0 G$ and the concatenated path $(\gamma \, \gamma'(1)) \star \gamma' \in P_0 G$.

For $\gamma , \gamma' \in P_0G$ and $\phi, \phi' \in \Gau(P)$, we define
\begin{equation}\label{Eq:Multiplication}
	\mu \big( (\gamma, \phi), (\gamma', \phi') \big) \coloneqq \big( \gamma \, \gamma',\, \pt_{\gamma\, \gamma'}^{-1} \circ \pt_{\gamma\, \left( \gamma'(1)\right)} \circ (\Phi_{\gamma'(1)}^*\phi)\, \circ \pt_{\gamma'} \circ \phi' \big)\, ,
\end{equation}
where we denote by $\pt_\gamma $ the isomorphism $P\longrightarrow \Phi_{\gamma(1)}^*P$
defined at a point $x\in M$ by the parallel transport along the path $\gamma_x$.
This is well-defined: let $\alpha, \alpha' \in P_0 G$ with $\gamma(1) = \alpha(1)$ and $\gamma'(1) = \alpha'(1)$.
Then
\begin{align}
	\mu \big( (\alpha, \hol(P, \alpha, \gamma) \circ \phi)&, (\alpha', \hol(P, \alpha', \gamma')\circ \phi') \big)
	\\[4pt]
	&= \big( \alpha \, \alpha',\, \pt_{\alpha\, \alpha'}^{-1}   \circ \Phi_{\alpha'(1)}^*(\pt_{\alpha} \circ \hol(P, \alpha, \gamma) \circ \phi')\, \circ \pt_{\alpha'} \circ \hol(P, \alpha', \gamma')\circ \phi'  \big)
	\\[4pt]
	&= \big( \alpha \, \alpha',\, \hol(P, \alpha \, \alpha',
          \gamma \, \gamma') \circ \pt_{\gamma\, \gamma'}^{-1}   \circ \Phi_{\gamma'(1)}^*(\pt_{\gamma} \circ \phi')\, \circ \pt_{\gamma'}\circ \phi'  \big)
	\\[4pt]
	&= \big( \gamma \, \gamma',\,   \pt_{\gamma\, \gamma'}^{-1}   \circ \Phi_{\gamma'(1)}^*(\pt_{\gamma} \circ \phi')\, \circ \pt_{\gamma'}\circ \phi'  \big)
	\\[4pt]
	&= 	\mu \big( (\gamma, \phi), (\gamma', \phi') \big) \ ,
\end{align}
where we used $\Phi_{\gamma'(1)}^*\pt_{\gamma}= \pt_{\gamma\, (\gamma'(1))}$.
Associativity then follows immediately from the associativity of the
products in $P_0 G$ and $\Gau(P)$, together with associativity of
taking pullbacks. Smoothness follows from the definition of the
quotient diffeology and the smooth dependence of parallel transport on
the path. 
\end{proof}

\begin{remark}
For abelian structure group $H$, we can use the fact that parallel 
transport commutes with gauge transformations to get the simplified expression
\begin{align}
\mu \big( (\gamma, \phi), (\gamma', \phi') \big) = \big( \gamma \, \gamma',\, \hol(P, \partial |\Delta^2|) \, (\Phi_{\gamma'(1)}^*\phi) \circ \phi' \big)
\end{align}
for the multiplication \eqref{Eq:Multiplication}. 
\qen
\end{remark}

\begin{remark}
\label{rmk:commutativity}
If $G$ is abelian, then the multiplicative structure yields isomorphisms
\begin{equation}
	\La_{G|g} \times \La_{G|g'} \to \La_{G|g\,g'} = \La_{G|g'\,g} \to \La_{G|g'} \times \La_{G|g}
\end{equation}
for all $g,g' \in G$.
That is, the group extension $\La_G$ spoils the commutativity of $G$, since its fibres multiply commutatively only up to coherent isomorphism.
\qen
\end{remark}

We summarise the connection to the construction from Section~\ref{sect:gauge bundles on groups} in

\begin{proposition}
Let $G$ be a connected Lie group, and let $P\longrightarrow M$ be 
a principal $H$-bundle on a manifold $M$ with smooth $G$-action. 
The map 
\begin{align}
\Gamma \colon \La_{G} &\longrightarrow \Sym_G(P) \\
\big[(\gamma,\phi)\big]  &\longmapsto \big(\pt_{\gamma} \circ \phi \colon 
P \longrightarrow \Phi_{\gamma(1)}^*P\big)
\end{align}
is an isomorphism of diffeological group extensions of $G$. 
\end{proposition}
\begin{proof}
The map is well-defined: consider two representatives
$(\gamma, \phi)$ and $(\alpha, \hol(P,\alpha,\gamma)\circ \phi)$ of the same equivalence class in $\La_G$, and calculate \[ 
\pt_\alpha \circ \hol(P,\alpha,\gamma)\circ \phi= \pt_\alpha \circ \pt_{\overline{\alpha}}\circ \pt_\gamma \circ \phi =\pt_\gamma \circ \phi \ .
\] The map is bijective, because two gauge transformations
$P\longrightarrow \Phi_{g}^*P$ differ by exactly one gauge transformation of $P$. 
It also follows directly from the definition that $\Gamma$ is a morphism of extensions.
We check that $\Gamma$ is a group homomorphism: for $[(\gamma,\phi)], \, [(\gamma',\phi')]\in \La_G$ we compute 
\begin{align}
\mu\big(\Gamma(\gamma,\phi), \Gamma (\gamma',\phi')\big) & = \mu(\pt_{\gamma} \circ \phi ,\, \pt_{\gamma'} \circ \phi') \\[4pt]
&=  \Phi_{\gamma'(1)}^*(\pt_{\gamma} \circ \phi) \circ \pt_{\gamma'} \circ \phi' \\[4pt]
&= \Gamma\big(\gamma \, \gamma',\, \pt_{\gamma\, \gamma'}^{-1} \circ \pt_{\gamma\, (\gamma'(1))} \circ (\Phi_{\gamma'(1)}^*\phi)\, \circ \pt_{\gamma'} \circ \phi' \big) \\[4pt]
&= \Gamma\Big(\mu\big((\gamma,\phi), (\gamma',\phi)\big)\Big) \ .
\end{align}

Finally, we verify that $\Gamma$ is smooth. Let $f\colon c \longrightarrow \mathcal{L}_G$
be a plot admitting a lift $\widehat{f}\colon c \longrightarrow 
P_0 G\times \Gau(P)$. We denote the components of $\widehat{f}$ by $\widehat{f}_\gamma$
and $\widehat{f}_{\Gau(P)}$. It is enough to show that  
\begin{align*}
 \pr_M^* P &\longrightarrow \Phi_{\widehat{f}_\gamma}^*P  \\
  P_x \ni p &\longmapsto \pt_{\widehat{f}_\gamma(u)} \big(\widehat{f}_{\Gau(P)}(u)(p)\big)  \in P_{\Phi_{\widehat{f}_\gamma(u)(1)}(x)}
\end{align*} 
is a gauge transformation. This follows from the smoothness of parallel transport 
(recalled in Section~\ref{sec: path space for PT} below).   
\end{proof}

\begin{corollary}
The action $\Phi \colon G \times M \to M$ lifts to an action
\begin{equation}
	\widehat{\Phi} \colon \La_G \times P \to P\,
\end{equation}
which covers the action of $G$ on $M$. 
\end{corollary}

\section{A global approach to parallel transport for bundle gerbes}
\label{Sec: parallel transport}

In Section~\ref{Sec: U(1)} we have constructed two diffeological
groups, $\Sym_G(P)$ and $\La_G$, which extend $G$ and control the
existence of $G$-equivariant structures on a principal bundle $P$ over
$M$.
The key to constructing $\La_G$, as well as to comparing the groups $\Sym_G(P)$ and $\La_G$ (see Section~\ref{sect:non-triv fibs via connections}), was the parallel transport on the principal bundle $P$.

If one replaces the principal bundle $P$ by a bundle gerbe $\Ga$ on $M$, there exist categorified versions of both these constructions which will be given in Section~\ref{sect:extensions from gerbes}.
However, in order to write down the categorification of $\La_G$ we need a notion of parallel transport for $\Ga$.
In this section we give a definition of parallel transport for $\Ga$ suited for our purposes and explicitly construct such a parallel transport from any connection on $\Ga$.
Our construction relies heavily on Waldorf's transgression-regression machine~\cite{Waldorf:Transgression_II}.

There is a different approach to the parallel transport on a bundle gerbe developed by Schreiber and Waldorf~\cite{SW:PT_and_Fctrs,SW:Fctrs_v_forms,SW:Local_2-fctrs}.
It relies on their technology of transport functors and is based on local constructions, which are then glued to global objects.
In~\cite{Waldorf:PT_in_2Bundles}, this has been extended to a canonical assignment of a parallel transport (in terms of a transport 2-functor) to any principal 2-bundle with connection whose structure group is a Lie 2-group.

Here, in contrast, we directly define and construct a global version of parallel transport suitable for our purposes.
As our main goal in this paper is the construction of categorified smooth group extensions, we leave it for future work to prove in detail that our notion of parallel transport for $\Ga$ agrees with that of Schreiber and Waldorf,  and instead focus on building the necessary input for the constructions in Section~\ref{sect:extensions from gerbes}.

\subsection{A path space approach to parallel transport on line bundles}
\label{sec: path space for PT}

Before we give our definition and construction of the parallel transport for bundle gerbes, we recast the parallel transport on line bundles from a global perspective.
Our notion of parallel transport for bundle gerbes will then be a categorification of this picture.
Let $M$ be a connected smooth manifold, and fix a base point $x\in M$;
if $M$ is not connected, we restrict to its connected components individually.
We denote by $PM$ the diffeological space of smooth paths with sitting instants in $M$ 
and by $P_0M$ the subspace of paths starting at $x$. 
Let $L$ be a line bundle on $M$ with connection. The smoothness of the parallel 
transport on $L$ can be encoded as follows: for $t \in [0,1]$, denote by $\ev_t \colon 
PM \longrightarrow M$, $\gamma \mapsto \gamma(t)$, the evaluation at $t$.
Parallel transport on $L$ is in particular an isomorphism
\begin{align*}
\pt^L\colon \ev_0^*L
&\longrightarrow \ev_1^*L
\\ 
L_{\gamma(0)} = \ev_0^*L_\gamma \ni \ell
&\longmapsto \pt^L_\gamma (\ell) \in L_{\gamma(1)} = \ev_1^*L_\gamma
\end{align*} 
of line bundles over $PM$. 

There is a different way to construct this isomorphism using descent.
Via transgression and regression
\cite{WaldorfI} we can construct a bundle $\Ra \Ta (L)$, which is isomorphic to $L$, from the descent data 
\begin{equation}
\begin{tikzcd}[row sep=1cm]
  \U(1) &  \\
(P_0{M})^{[2]} \ar[u,"f"] \ar[r, shift right=0.70]\ar[r, shift left=0.70] & {P}_0M  \ar[d,"\ev_1"] 
\\
 & M 
\end{tikzcd} \ \  
\end{equation}
with respect to the path fibration.
Here $f$ is constructed as in Section~\ref{sect:non-triv fibs via connections} from the holonomy of $L$. 
The total space of the line bundle $\Ra\Ta(L)$ consists of 
equivalence classes of pairs $(\gamma,\zeta)\in P_0 M\times \C$, where the equivalence
relation reads as $(\gamma_1,\zeta)\sim (\gamma_2,f(\gamma_1,\gamma_2)\,\zeta)$ for 
$(\gamma_1,\gamma_2)\in(P_0{M})^{[2]}$ and $\zeta \in \C$.
An isomorphism $ g_\chi \colon \Ra \Ta (L) \longrightarrow L$ 
can be constructed by picking a 
trivialisation $\chi \colon\C \longrightarrow  L_{x}$ of the fibre of $L$ 
over the base point $x \in M$ and defining
\[
g_\chi \big([\gamma,\zeta]\big) \coloneqq \pt^L_\gamma \big(\chi(\zeta)\big) \ .
\] 
 
The pullbacks $\ev_0^* \Ra\Ta(L)$ and $\ev_1^* \Ra\Ta(L)$ are thus described in terms of descent data with respect to the covers 
$ \ev_0^*P_0 M\cong P_0M \times_M PM \to M$ and $ \ev_1^*P_0 M \cong PM\times_M P_0M \to M$, respectively. 
In order to construct the isomorphism $\pt^{\Ra\Ta(L)}$ explicitly we use the space (see Figure~\ref{Fig:P Delta})
\begin{equation}
	P_{\partial \Delta^2} M \coloneqq \ev_0^*P_0 M \times_{PM} \ev_1^*P_0M
	\cong P_0 M \times_M PM \times_M P_0 M \ ,
\end{equation}
which fits into the diagram 
 \begin{equation}
 \begin{tikzcd}
 & P_{\partial \Delta^2} M \ar[rd] \ar[ld] & \\ 
\ev_0^*P_0M \ar[rd]  & & \ev_1^*P_0M  \ar[ld]\\ 
 & PM &  
 \end{tikzcd}
\end{equation}   

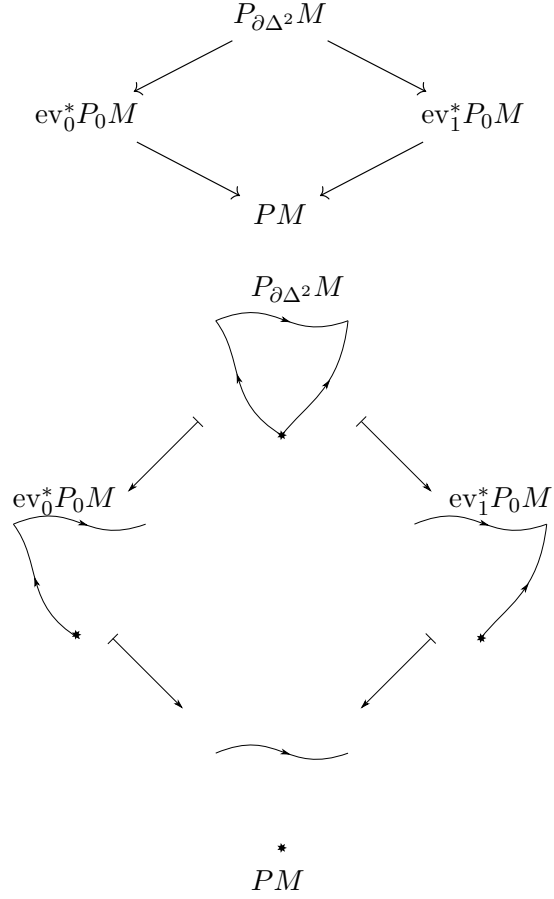
\begin{figure}[h]
\begin{center}
\begin{overpic}[scale=0.5]
{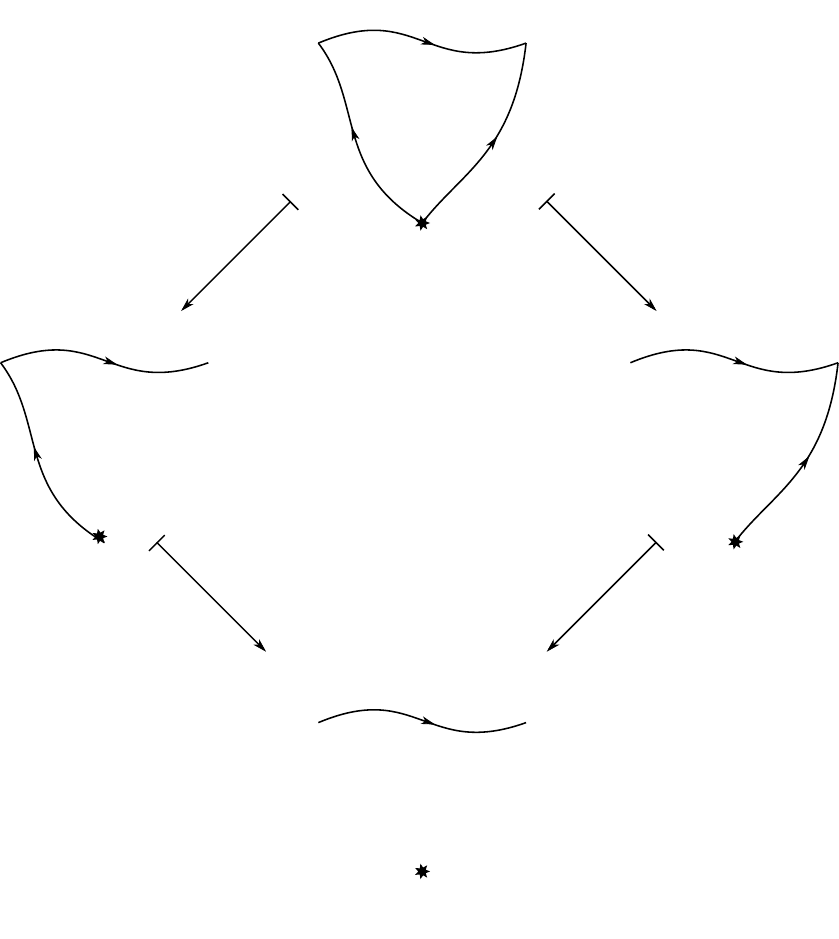}
\put(90,225){$P_{\partial \Delta^2}M$}
\put(0,145){$\ev_0^*P_0M$}
\put(165,145){$\ev_1^*P_0M$}
\put(90,0){$PM$}
\end{overpic}
\end{center}
\caption{\small Elements in the spaces $P_{\partial \Delta^2}M$, $\ev_0^*P_0M$ and $\ev_1^*P_0M$.}
\label{Fig:P Delta}
\end{figure}     
An isomorphism from $\ev_0^*\Ra \Ta (L)$ to $\ev_1^*\Ra \Ta (L)$ can be described by 
a function $P_{\partial \Delta^2} M \longrightarrow \U(1)$ which is compatible with the descent data. There is a 
canonical choice for such a function given by
\begin{align}
 P_{\partial \Delta^2} M \longrightarrow LM
  \xrightarrow{ \ \hol \ } \U(1) \ .
\end{align} 
Concretely, the induced map is 
\begin{align*}
\pt^{\Ra \Ta (L)} \colon \ev_0^*\Ra \Ta (L)
&\longrightarrow \ev_1^*\Ra \Ta (L)
\\ 
[\gamma_{xy},\zeta]&\mapsto\pt^{\Ra \Ta (L)}_{(\gamma_{xy}, \gamma_{yz}, \gamma_{xz})} [\gamma_{xy},\zeta]
= \big[ \gamma_{xz},\hol\big(L,\overline{\gamma_{xz}}\star( \gamma_{yz}\star
                     \gamma_{xy})\big)\, \zeta \big] \ ,
\end{align*} 
where $x$ is the fixed base point of $M$ while $y,z \in M$ are arbitrary points, and $\gamma_{ab}$ denotes a path from $a$ to $b$ for $a,b \in \{x,y,z\}$.
The holonomy appearing 
here agrees with $ \hol(L,(\overline{\gamma_{xz}}\star \gamma_{yz})\star \gamma_{xy})$.  The construction is independent
of all choices involved.    
Now a straightforward computation shows that the diagram
\begin{equation}
\begin{tikzcd}[column sep=2cm, row sep=1.25cm]
\ev_0^*L \ar[r,"\pt^L"] \ar[d, "\ev_0^* g_\chi", swap] & \ev_1^*L \ar[d, "\ev_1^* g_\chi"] \\ 
\ev_0^*\Ra \Ta (L) \ar[r,"\pt^{\Ra \Ta (L)}",swap] & \ev_1^*\Ra \Ta (L)
\end{tikzcd}
\end{equation}
commutes. This shows that we can construct the parallel transport on $L$ completely in terms
of the descent data with respect to the path fibration. For bundle gerbes, the analogue of
$\pt^L$ is difficult to define directly, but an analogous approach via descent data on $P_0M$ allows us to solve this problem.

\subsection{Global definition of parallel transport on bundle gerbes}

As before, let $M$ be a manifold, and let $\Ga \in \BGrb(M)$ be a bundle gerbe on $M$.
A parallel transport on $\Ga$ should in particular be a 1-isomorphism
\begin{equation}
	\pt^\Ga_1 \colon \ev_0^*\Ga \to \ev_1^*\Ga
\end{equation}
in $\BGrb(PM)$ with a 2-isomorphism
\begin{equation}
	\sfc^* \pt^\Ga_1 \cong 1_\Ga \ ,
\end{equation}
where $\sfc \colon M \to PM$ is the embedding of $M$ into $PM$ as constant paths.
Note that the parallel transport is, in general, an isomorphism of
gerbes \emph{without} connections. The same is true for bundles: the parallel transport on a vector
bundle with connection respects the connection if and only if the connection is flat.

To proceed further, we need some definitions. Let $i,n \in \N$ with $1 \leq i \leq n$.
For each $s = (s_1, \ldots, s_{n-1}) \in [0,1]^{n-1}$, define a smooth map
\begin{align}
	\iota^n_{i;s} \colon [0,1] &\to [0,1]^n
	\\
	t &\mapsto (s_1, \ldots, s_{i-1}, t, s_i, \ldots, s_{n-1}) \ .
\end{align}
Consider the diffeological spaces $P^nM$ which are defined by the sets of
all maps 
$$
\Sigma \colon [0,1]^n \to M
$$ 
satisfying the following property: for all $i = 1,
\ldots, n$, there exists $\epsilon_i > 0$ such the map $\Sigma
\circ \iota^n_{i;s}:[0,1]\to M$ is locally constant on $[0,\epsilon_i) \sqcup (1-\epsilon_i,1]$.
Note that in a plot of $P^nM$, the $\epsilon_i$ do not have to be
constant over the domain of the plot. The space $P^n M$ describes $n$-cubes in $M$ with sitting instants in all directions perpendicular to the faces of $[0,1]^n$; that is, $P^nM$ describes iterated smooth homotopies of paths with sitting instants in $M$.

We also consider the subspaces $P^n_*M$ of the diffeological spaces $P^nM$
consisting of maps $\Sigma\in P^nM$ satisfying the following property: for
all $s \in [0,1]^{n-1}$, and for each $ j = 1, \ldots, n-1$ such that
$ s_j \in \{0,1\}$, the map $\Sigma \circ \iota^n_{i;s}$ is constant
for all $i>j$. The space $P^n_*M$ describes iterated smooth homotopies with fixed endpoints in $M$.
For example, $P_*M = PM$ is the space of paths with sitting instants, $P^2_*M$ consists of maps $\Sigma \in
P^2M$ such that 
$$
\Sigma(0,t) = \Sigma(0,0) \qquad \mbox{and} \qquad \Sigma(1,t) =
\Sigma(1,0)
$$ 
for all $t \in [0,1]$ and so is the space of homotopies
of paths with fixed endpoints in $M$, and an element in $P^3_*M$ is a
family of fixed-ends homotopies between two fixed paths in $M$.
We say that an element $\Sigma$ in $P^n_*M$ or in $P^n M$ is
\emph{thin} if its differential $\Sigma_*$ has non-maximal ranks $\rk(\Sigma_{*|s}) < n$ for all $s \in [0,1]^n$.

Let $s = (s_1, \ldots, s_k) \in [0,1]^k$ and $n = k+l$.
For $0 \leq i_1 < \cdots < i_l \leq n$, we define a map
\begin{equation}
	\iota^n_{i_1, \ldots, i_l; s} \colon [0,1]^l \to [0,1]^n
\end{equation}
which inserts the coordinates of $t = (t_1, \ldots, t_l) \in [0,1]^l$ into the $k$-tuple $s$ such that
\begin{equation}
	\big( \iota^n_{i_1, \ldots, i_l; s}(t) \big)_j = t_j
\end{equation}
for every $j \in \{i_1, \ldots, i_l\}$.
The maps $\iota^n_{i_1, \ldots, i_l; s} \colon [0,1]^l \to [0,1]^n$ induce maps
\begin{equation}
	\iota^{n*}_{i_1, \ldots, i_l; s} \colon P^n M \to P^l M
\end{equation}
which map $P_*^n M$ to $P^l_* M$.

For the parallel transport of a bundle gerbe, there should also be a 2-isomorphism
\begin{equation}
\label{eq:pt^G thin 1}
	\pt^\Ga_2 \colon (\iota^{2*}_{1;0})^* \pt^\Ga_1 \to (\iota^{2*}_{1;1})^* \pt^\Ga_1
\end{equation}
in $\BGrb(P^2_*M)$.
In other words, any map $\Sigma \in P^2_*M$ is in particular a smooth map $[0,1]^2 \to M$ from the square to $M$.
This map is constant on the vertical edges of the square.
Pulling back the isomorphism $\pt^\Ga_1$ to the horizontal edges of the square gives two 1-morphisms $\Ga_{\Sigma(0,0)} \to \Ga_{\Sigma(1,0)}$, and the 2-morphism $\pt^\Ga_2$ relates these.
The data $(\pt^\Ga_1, \pt^\Ga_2)$ are required to satisfy the following two properties, which are motivated by~\cite{BW:OCFFTs, BW:Transgression_of_D-branes, Waldorf:Transgression_II}:
\begin{myenumerate}
\item For any two thin maps $\Sigma, \Sigma' \in P^2_*M$ with $\Sigma
  \circ \iota^2_{1;s} = \Sigma' \circ \iota^2_{1;s}$ for $s = 0,1$,
  there is an equality
\begin{equation}
\label{eq:pt^G thin 2}
	\pt^\Ga_{2|\Sigma} = \pt^\Ga_{2|\Sigma'} \ .
\end{equation}
That is, the 2-morphism $\pt^\Ga_2$ evaluated on any pair of fixed-ends thin homotopies between any two given paths in $M$ gives the same result.

\item We further demand that for any thin map $h \in P^3_*M$, there is
  an equality
\begin{equation}
\label{eq:pt^G thin 3}
	(\iota^{3*}_{1,2;0})^*\pt^\Ga_2{}_{|h} =
        (\iota^{3*}_{1,2;1})^*\pt^\Ga_2{}_{|h} \ .
\end{equation}
\end{myenumerate}

As we will be using $P^n_* M$ mostly for $n = 0,1,2$, we adopt the convention to write $\gamma_2 \star \gamma_1$ for the concatenation of smooth paths in $M$, and if $\Sigma, \Sigma' \in P_*^2 M$ are homotopies $\Sigma \colon \gamma \to \gamma'$ and $\Sigma' \colon \gamma' \to \gamma''$, we write $\Sigma' \star_2 \Sigma \colon \gamma \to \gamma''$ for their vertical concatenation.
If $\Xi \colon \alpha \to \alpha'$ is a further homotopy in $ P_*^2 M$ such that the starting point of $\alpha$ is the endpoint of $\gamma$, then we write $\Xi \star \Sigma \colon \alpha \star \gamma \to \alpha' \star \gamma'$ for the horizontal concatenation of the homotopies.
We will also often use the term `composition' instead of `concatenation'.

\begin{definition}
\label{def:pt for gerbe}
Let $M$ be a smooth manifold.
A \emph{parallel transport} on a bundle gerbe $\Ga \in \BGrb(M)$ is a quadruple $\pt^\Ga = (\pt^\Ga_1, \pt^\Ga_2, \pt^\Ga_\star, \varepsilon^\Ga)$ of
\begin{myenumerate}
\item a 1-isomorphism
\begin{equation}
	\pt^\Ga_1 \colon \ev_0^*\Ga \to \ev_1^*\Ga
\end{equation}
of bundle gerbes over $PM$,

\item a 2-isomorphism
\begin{equation}
\pt^\Ga_2 \colon (\iota^{2*}_{1;0})^* \pt^\Ga_1 \to (\iota^{2*}_{1;1})^* \pt^\Ga_1
\end{equation}
in $\BGrb(P^2_*M)$,

\item a 2-isomorphism
\begin{equation}
	\pt^\Ga_\star \colon \pr_1^* \pt^\Ga_1 \circ \pr_2^* \pt^\Ga_1 \to (\,\boldsymbol\cdot\, \star \,\boldsymbol\cdot\,)^* \pt^\Ga_1
\end{equation}
over $PM \times_M PM$, where $\pr_1$ and $\pr_2$ are the respective
projections of $PM\times_MPM$ to the first and second factors, and

\item a 2-isomorphism
\begin{equation}
	\varepsilon^\Ga \colon \sfc^* \pt^\Ga_1 \to 1_\Ga
\end{equation}
over $M$, where $\sfc \colon M \to PM$ is the inclusion of $M$ as the space of constant paths.
\end{myenumerate}

These data are required to satisfy properties~\eqref{eq:pt^G thin 2} and~\eqref{eq:pt^G thin 3}.
Due to property~\eqref{eq:pt^G thin 2}, there is a \emph{canonical} 2-isomorphism
\begin{equation}
	\pt^\Ga_{1|(\gamma_3 \star \gamma_2) \star \gamma_1} \cong \pt^\Ga_{1|\gamma_3 \star (\gamma_2 \star \gamma_1)}
\end{equation}
for every $(\gamma_1, \gamma_2, \gamma_3) \in PM \times_M PM \times_M PM$, and we demand that $\pt^\Ga_\star$ is coherently associative with respect to this isomorphism.
The morphism $\pt^\Ga_\star$ also needs to be compatible with the
unitors in $\BGrb(PM)$ and sit in a commutative diagram
\begin{equation}
\label{eq:compat pt_2 and pt_star}
\begin{tikzcd}[column sep=2cm, row sep=1.25cm]
	\pt^\Ga_{1|\gamma_2} \circ \pt^\Ga_{1|\gamma_1} \ar[r, "\pt^\Ga_{\star|\gamma_2, \gamma_1}"] \ar[d, "\pt^\Ga_{2|\Sigma_2} \circ \pt^\Ga_{2|\Sigma_1}"']
	& \pt^\Ga_{1|\gamma_2 \star \gamma_1} \ar[d, "\pt^\Ga_{2|\Sigma_2 \star \Sigma_1}"]
	\\
	\pt^\Ga_{1|\alpha_2} \circ \pt^\Ga_{1|\alpha_1} \ar[r, "\pt^\Ga_{\star|\alpha_2, \alpha_1}"']
	& \pt^\Ga_{1|\alpha_2 \star \alpha_1}
\end{tikzcd}
\end{equation}
for all $x,y,z \in M$, all paths $\gamma_1, \alpha_1$ from $x$ to $y$, all paths $\gamma_2, \alpha_2$ from $y$ to $z$ in $M$, and for all fixed-ends homotopies $\Sigma_i \colon \gamma_i \to \alpha_i$.
Furthermore, $\pt^\Ga_2$ has to respect vertical composition and satisfy the interchange law
\begin{equation}
\label{eq:interchange law for pt^Ga}
	\pt^\Ga_{2|\Sigma'_1 \star \Sigma'_0} \circ \pt^\Ga_{2|\Sigma'_1 \star \Sigma'_0}
	= \pt^\Ga_{2|\Sigma'_1 \star_2 \Sigma_1} \circ_2 \pt^\Ga_{2|\Sigma'_0 \star_2 \Sigma_0}
\end{equation}
for all points $x_0,x_1,x_2 \in M$, all paths $\alpha_i, \beta_i,
\gamma_i $ from $x_i$ to $x_{i+1}$, and for all fixed-ends homotopies $\Sigma_i \colon \alpha_i \to \beta_i$ and $\Sigma'_i \colon \beta_i \to \gamma_i$ with $i = 0,1$.
\end{definition}

\begin{remark}
The associativity condition in detail reads as follows:
for every concatenable triple $(\gamma_1, \gamma_2, \gamma_3)$ of
paths in $M$ there is a commutative diagram
\begin{equation}
\begin{tikzcd}[column sep=1.75cm, row sep=1.25cm]
	\pt^\Ga_{1|\gamma_3} \circ \pt^\Ga_{1|\gamma_2 \star \gamma_1} \ar[d, "\pt^\Ga_{\star|\gamma_3, \gamma_2 \star \gamma_1}"']
	& \pt^\Ga_{1|\gamma_3} \circ \pt^\Ga_{1|\gamma_2} \circ \pt^\Ga_{1|\gamma_1} \ar[r, "\pt^\Ga_{\star|\gamma_3, \gamma_2} \circ 1"]
	\ar[l, "1 \circ \pt^\Ga_{\star|\gamma_2, \gamma_1}"']
	& \pt^\Ga_{1|\gamma_3 \star \gamma_2} \circ \pt^\Ga_{1|\gamma_1} \ar[d, "\pt^\Ga_{\star|\gamma_3 \star \gamma_2, \gamma_1}"]
	\\
	\pt^\Ga_{1|\gamma_3 \star (\gamma_2 \star \gamma_1)} \ar[rr]
	& & \pt^\Ga_{1|(\gamma_3 \star \gamma_2) \star \gamma_1}
\end{tikzcd}
\end{equation}
in $\BGrb(PM\times_MPM)$, where the bottom arrow is the canonical 2-isomorphism obtained via~\eqref{eq:pt^G thin 2} from any reparameterisation of $[0,1]$ that yields a homotopy $\gamma_3 \star (\gamma_2 \star \gamma_1) \sim (\gamma_3 \star \gamma_2) \star \gamma_1$.
\qen
\end{remark}

\begin{remark}
By property~\eqref{eq:pt^G thin 3}, our definition factors through the path 2-groupoid of $M$ as defined by Schreiber and Waldorf~\cite{SW:Fctrs_v_forms,SW:Local_2-fctrs}.
Given a manifold $M$, they construct a 2-groupoid internal to diffeological spaces, whose level sets are essentially $M $, $PM$ and the quotient $P^2_* M/{\sim}$ of $P_*^2 M$ by thin homotopies.
Note that in this quotient they also implement condition~(4.1).
\qen
\end{remark}

In contrast to the case of parallel transport on vector bundles, we can define morphisms between parallel transports on a given bundle gerbe.

\begin{definition}
\label{def:PT(G)}
Let $\Ga \in \BGrb(M)$ be a bundle gerbe on $M$.
Let $\pt^\Ga = (\pt^\Ga_1, \pt^\Ga_2, \pt^\Ga_\star, \varepsilon^\Ga)$ and $\pt'{}^\Ga = (\pt'_1{}^\Ga, \pt'_2{}^\Ga, \pt'_\star{}^\Ga, \varepsilon'{}^\Ga)$ be two choices of parallel transport on $\Ga$.
A \emph{morphism} $\pt^\Ga \to \pt'{}^\Ga$ of parallel transports on $\Ga$ is a 2-isomorphism $\psi \colon \pt^\Ga_1 \to \pt'_1{}^\Ga$ in $\BGrb(PM)$ that intertwines the 2-isomorphism $\pt^\Ga_2$ with $\pt'_2{}^\Ga$, the 2-isomorphism $\pt^\Ga_\star$ with $\pt'_\star{}^\Ga$, and the 2-isomorphism $\varepsilon^\Ga$ with $\varepsilon'{}^\Ga$.
This defines a groupoid $\PT(\Ga)$ of parallel transports on $\Ga$.
\end{definition}

This notion of morphism of parallel transports is not an analogue of a
gauge transformation, since it does not necessarily come from an
automorphism of the bundle gerbe $\Ga$.

\subsection{Construction of the parallel transport}

We now proceed to show that every bundle gerbe with connection on a manifold $M$ has a canonical parallel transport.
Let $M$ be a connected manifold, and fix a base point $x \in M$;
otherwise, if $M$ is not connected, we treat the connected components
of $M$ separately.
By results of Waldorf~\cite{Waldorf:Transgression_II}, any bundle gerbe $\Ga \in \BGrb^\nabla(M)$ is isomorphic to a bundle gerbe $\Ga' \in \BGrb^\nabla(M)$ that is defined over the diffeological path fibration $P_0 M \to M$.
Given a choice of base point $x \in M$, Waldorf constructs a bundle gerbe $\Ga' = \Ra \Ta (\Ga)$ as the regression of the transgression line bundle of $\Ga$, together with a natural 1-isomorphism $\Aa_\Ga \colon \Ga \to \Ga'$ in the homotopy category of $\BGrb^\nabla(M)$; that is, $\Aa_\Ga$ is determined only up to 2-isomorphism.
(We remark, however, that the natural 1-isomorphism $\Aa_\Ga$ from~\cite{Waldorf:Transgression_II} is determined \emph{canonically} once we fix a preimage of the base point $x \in M$ under the surjective submersion $\pi \colon Y \to M$ underlying the bundle gerbe $\Ga$.)

Consider the bundle gerbe $\Ga' = \Ra \Ta (\Ga) \in \BGrb^\nabla(M)$ with connection on $M$, defined with respect to the path fibration $\pi:P_0 M \to M$.
Its line bundle $L$ is the pullback of the transgression line bundle $\Ta \Ga \to LM$ along the map
\begin{align}
	(P_0 M)^{[2]} &\to LM
	\\
	(\alpha, \alpha') &\mapsto \ol{\alpha'} \star \alpha \ .
\end{align}
By a slight abuse of notation, we also denote this pullback by $\Ta \Ga \to (P_0 M)^{[2]}$.

\subsubsection*{Construction of $\pt^{\Ga'}_1$}

We would like to construct a 1-isomorphism
\begin{equation}\label{eq:pt1Ga'}
	\pt^{\Ga'}_1 \colon \ev_0^*\Ga' \to \ev_1^*\Ga'
\end{equation}
in $\BGrb(PM)$.
For $t = 0,1$, the bundle gerbe $\ev_t^*\Ga'$ is defined over the subduction $\ev_t^*P_0 M \to PM$.
There are canonical isomorphisms of diffeological spaces
\begin{equation} 
	\ev_0^*P_0 M \cong P_0 M \times_M PM
	\qandq
	\ev_1^*P_0 M \cong PM \times_M P_0 M \ .
\end{equation}
Recall from Section~\ref{sec: path space for PT} the space
\begin{equation}
	P_{\partial \Delta^2} M \coloneqq \ev_0^*P_0 M \times_{PM} \ev_1^*P_0M
	\cong P_0 M \times_M PM \times_M P_0 M \ .
\end{equation}
A point in the total space $P_{\partial \Delta^2} M$ is a triple $(\alpha_0, \gamma, \alpha_1)$ of a path $\gamma \in PM$ and based paths $\alpha_t \in P_0 M$ such that $\gamma(t) = \alpha_t(1)$ for $t = 0,1$.
Any 1-morphism $\ev_0^*\Ga' \to \ev_1^*\Ga'$ is defined over
(possibly a refinement of) the subduction $\xi:P_{\partial \Delta^2} M \to
PM$.

There is a smooth map, i.e.~a morphism of diffeological spaces
\begin{align}
\label{eq:sfs}
	\sfs \colon P_{\partial \Delta^2} M &\to LM
	\\
	(\alpha_0, \gamma, \alpha_1) &\mapsto \ol{\alpha_1} \star
                                       (\gamma \star \alpha_0) \ .
\end{align}
There is also the smooth map
\begin{align}
	\tilde{\sfs} \colon P_{\partial \Delta^2} M &\to LM
	\\
	(\alpha_0, \gamma, \alpha_1) &\mapsto (\ol{\alpha_1} \star
                                       \gamma) \star \alpha_0 \ .
\end{align}
The maps $\sfs$ and $\tilde{\sfs}$ are smoothly homotopic via precomposition by a homotopy $h$ of piecewise smooth homeomorphisms $[0,1] \to [0,1]$; these fail to be smooth exactly at those points of the interval where the concatenations happen, but at these points all three paths have sitting instants, so that at each time the homotopy maps to $LM$, as desired.
For each triple of paths $(\alpha_0, \gamma, \alpha_1)$, this results in a thin homotopy in $LM$ from $\ol{\alpha_1} \star (\gamma \star \alpha_0)$ to $(\ol{\alpha_1} \star \gamma) \star \alpha_0$.
By the superficiality of the parallel transport $\pt^{\Ta \Ga}$ on the
transgression line
bundle~\cite[Definition~2.2.1]{Waldorf:Transgression_II} (see also the
end of Section~\ref{Sec: Bundle gerbes}), we thus obtain a canonical isomorphism
\begin{equation}
	r \colon \sfs^*\Ta\Ga \to \tilde{\sfs}^*\Ta\Ga
\end{equation}
in $\HLBdl^\nabla(P_{\partial \Delta^2} M)$.
The fact that this isomorphism preserves connections is a direct consequence of~\cite[Lemma~2.3.3]{Waldorf:Transgression_II}.
Since $\pt^{\Ta\Ga}$ is thin-invariant, it follows that the morphism $r$ is defined independently of the choice of homotopy $h$.

We define a morphism $\pt^{\Ga'}_1\colon \ev_0^* \Ga' \longrightarrow \ev_1^*\Ga'$ as follows:
its underlying line bundle is the line bundle $\sfs^* \Ta \Ga \to P_{\partial \Delta^2} M$.
To turn this into a morphism of bundle gerbes, we need to provide an isomorphism of line bundles
\begin{equation}
	\beta \colon \pr_0^*\Ta \Ga \otimes \xi_1^* \sfs^*\Ta \Ga \to \xi_0^*\sfs^*\Ta \Ga \otimes \pr_1^*\Ta \Ga
\end{equation}
over $(P_{\partial \Delta^2} M)^{[2]}$.
Let us unravel this: the fibre product $(P_{\partial \Delta^2} M)^{[2]} = P_{\partial \Delta^2} M \times _{PM} P_{\partial \Delta^2} M$ consists of pairs $((\alpha_0, \gamma, \alpha_1),(\alpha'_0, \gamma, \alpha'_1))$ where $(\alpha_0, \gamma, \alpha_1)$ and $(\alpha'_0, \gamma, \alpha'_1)$ are elements of $P_{\partial \Delta^2} M$.
For $t = 0,1$, there are the projection maps
\begin{align}
	\pr_t \colon (P_{\partial \Delta^2} M)^{[2]} &\to (P_0 M)^{[2]}
	\\
	\big((\alpha_0, \gamma, \alpha_1),(\alpha'_0, \gamma,
  \alpha'_1)\big) &\mapsto (\alpha_t, \alpha'_t) \ .
\end{align}
Thus
\begin{align}
	(\pr_0^*\Ta \Ga \otimes \xi_1^* \sfs^*\Ta \Ga)_{((\alpha_0, \gamma, \alpha_1),(\alpha'_0, \gamma, \alpha'_1))}
	&= \Ta \Ga_{\ol{\alpha'_0} \star \alpha_0} \otimes \Ta
          \Ga_{\ol{\alpha'_1} \star (\gamma \star \alpha'_0)} \ ,
	\\[4pt]
	(\xi_0^* \sfs^*\Ta \Ga \otimes \pr_1^*\Ta \Ga)_{((\alpha_0, \gamma, \alpha_1),(\alpha'_0, \gamma, \alpha'_1))}
	&= \Ta \Ga_{\ol{\alpha_1} \star (\gamma \star \alpha_0)}
          \otimes \Ta \Ga_{\ol{\alpha'_1} \star \alpha_1} \ .
\end{align}
Let $\lambda \colon \pi_{0,1}^*\Ta \Ga \otimes \pi_{1,2}^*\Ta \Ga \to
\pi_{0,2}^*\Ta \Ga$ denote the fusion product of the transgression
line bundle $\Ta \Ga$ over $(P_0 M)^{[3]}$
(see~\cite[Section~4.2]{Waldorf:Transgression_II}), which provides the
bundle gerbe multiplication on $\Ga'$. At a point $(\alpha_0,\alpha_1,\alpha_2)\in 
(P_0M)^{[3]}$ the fusion product consists of unitary isomorphisms
\begin{align}
\lambda_{\alpha_0,\alpha_1,\alpha_2} \colon \Ta \Ga_{\ol{\alpha_1}\star \alpha_0}\otimes \Ta \Ga_{ \ol{\alpha_2} \star \alpha_1}
\longrightarrow \Ta \Ga_{\ol{\alpha_2} \star \alpha_0} \ .
\end{align} 
The diffeological space $(P_{\partial \Delta^2} M)^{[2]}$ comes with smooth maps
\begin{align}
	p_0 \colon (P_{\partial \Delta^2} M)^{[2]} &\to (P_0 M)^{[3]}
	\\
	\big((\alpha_0, \gamma, \alpha_1),(\alpha'_0, \gamma,
  \alpha'_1)\big) &\mapsto (\alpha_0, \alpha'_0, \ol{\gamma} \star
                    \alpha'_1)
\end{align}
and
\begin{align}
p_1 \colon (P_{\partial \Delta^2} M)^{[2]} &\to (P_0 M)^{[3]}
	\\
	\big((\alpha_0, \gamma, \alpha_1),(\alpha'_0, \gamma,
  \alpha'_1)\big) &\mapsto (\gamma \star \alpha_0, \alpha'_1,
                    \alpha_1) \ .
\end{align}

We set
\begin{equation}
	\beta \coloneqq p_1^*\lambda^{-1} \circ r^{-1} \circ
        p_0^*\lambda \circ (1 \otimes r) \ .
\end{equation}
Explicitly, at a point $((\alpha_0, \gamma, \alpha_1),(\alpha'_0,
\gamma, \alpha'_1)) \in (P_{\partial \Delta^2} M)^{[2]}$, this is the
isomorphism defined by the diagram
\begin{equation}
\label{eq:def beta}
\begin{tikzcd}[row sep=1cm, column sep=2.5cm]
	\Ta \Ga_{\ol{\alpha'_0} \star \alpha_0} \otimes \Ta \Ga_{\ol{\alpha'_1} \star (\gamma \star \alpha'_0)}
	\ar[r, "1 \otimes r"] \ar[dd, dashed, "\beta"']
	& \Ta \Ga_{\ol{\alpha'_0} \star \alpha_0} \otimes \Ta \Ga_{(\ol{\alpha'_1} \star \gamma) \star \alpha'_0}
	\ar[d, "\lambda"]
	\\
	& \Ta \Ga_{(\ol{\alpha'_1} \star \gamma) \star \alpha_0}
	\ar[d, "r^{-1}"]
	\\
	\Ta \Ga_{\ol{\alpha_1} \star (\gamma \star \alpha_0)} \otimes \Ta \Ga_{\ol{\alpha'_1} \star \alpha_1}
	& \Ta \Ga_{\ol{\alpha'_1} \star (\gamma \star \alpha_0)}
	\ar[l, "\lambda^{-1}"]
\end{tikzcd}
\end{equation}
This morphism is compatible with the bundle gerbe multiplication on $\Ga'$:
consider an arbitrary point
\begin{equation}
	\big( (\alpha_0, \gamma, \alpha_1),(\alpha'_0, \gamma, \alpha'_1), (\alpha''_0, \gamma, \alpha''_1) \big)
	\ \in \ (P_{\partial \Delta^2} M)^{[3]} \ .
\end{equation}
Then there is a commutative diagram
\begin{equation}
\begin{tikzcd}[row sep=1cm, column sep=2.5cm]
	\Ta \Ga_{\ol{\alpha'_0} \star \alpha_0} \otimes \Ta \Ga_{\ol{\alpha''_0} \star \alpha'_0} \otimes \Ta \Ga_{\ol{\alpha''_1} \star (\gamma \star \alpha''_0)}
	\ar[r, "\lambda_{(\alpha_0, \alpha'_0, \alpha''_0)} \otimes 1"] \ar[d, "1 \otimes \beta"']
	& \Ta \Ga_{\ol{\alpha''_0} \star \alpha_0} \otimes \Ta \Ga_{\ol{\alpha''_1} \star (\gamma \star \alpha''_0)}
	\ar[dd, "\beta"]
	\\
	\Ta \Ga_{\ol{\alpha'_0} \star \alpha_0} \otimes \Ta \Ga_{\ol{\alpha'_1} \star (\gamma \star \alpha'_0)} \otimes \Ta \Ga_{\ol{\alpha''_1} \star \alpha'_1}
	\ar[d, "\beta \otimes 1"']
	&
	\\
	\Ta \Ga_{\ol{\alpha_1} \star (\gamma \star \alpha_0)} \otimes \Ta \Ga_{\ol{\alpha'_1} \star \alpha_1} \otimes \Ta \Ga_{\ol{\alpha''_1} \star \alpha'_1}
	\ar[r, "1 \otimes \lambda_{(\alpha_1, \alpha'_1, \alpha''_1)}"']
	& \Ta \Ga_{\ol{\alpha_1} \star (\gamma \star \alpha_0)} \otimes \Ta \Ga_{\ol{\alpha''_1} \star \alpha_1}
\end{tikzcd}
\end{equation}
The commutativity follows from the associativity of the fusion product $\lambda$ and the fact that it respects the connection on $\Ta\Ga$~\cite{Waldorf:Transgression_II} so that, in particular, $\lambda$ is compatible with the morphism~$r$.

\subsubsection*{Construction of $\pt^{\Ga'}_2$}

Next we construct the 2-isomorphism 
$$
\pt^{\Ga'}_2\colon (\iota^{2*}_{1;0})^* \pt_1^{\Ga'} \longrightarrow
(\iota^{2*}_{1,1})^* \pt_1^{\Ga'} 
$$ 
in $\BGrb(P_*^2M)$ that is part of the parallel transport data for $\Ga'$.
For this, we recall that the fibre of the hermitean line bundle $\Ta\Ga$ at a loop $\gamma$ is constructed from pairs $([\Sa], z)$ of a 2-isomorphism class $[\Sa]$ of trivialisations $\Sa \colon \gamma^*\Ga \to \Ia_0$ in $\BGrb^\nabla(\bbS^1)$ and a complex number $z \in \C$.
The complex line $\Ta\Ga_{\gamma}$ is the set of equivalence classes of such pairs under the equivalence relation
\begin{equation}
	\big( [\Sa], z \big) \sim \big( [\Sa'], \hol(\bbS^1, \sfR (
        \Sa' \circ \Sa^{-1}))\, z \big) \ ,
\end{equation}
where, for a manifold $M$, the functor $\sfR \colon
\BGrb^\nabla(M)(\Ia_B , \Ia_{B'}) \to \HLBdl^\nabla(M)$ for
$B,B'\in\Omega^2(M)$ is essentially descent for line bundles; for
details, see~\cite{Bunk--Thesis,BSS--HGeoQuan,Waldorf--More_morphisms}
(see also Section~\ref{Sec: Bundle gerbes}).

Let $M^{\bbD^2}$ be the diffeological space of smooth maps from the unit disk $\bbD^2$ to $M$.
Let
\begin{align}
	\partial \colon M^{\bbD^2} &\to LM \\ f &\mapsto f_{|\bbS^1}
\end{align}
denote the smooth map induced by restriction to the boundary of the unit disk.
The hermitean line bundle $\partial^*\Ta\Ga$ on $M^{\bbD^2}$ has a canonical trivialisation which is defined as follows:
for a smooth map $f \colon \bbD^2 \to M$, choose a trivialisation $\Sa \colon f^*\Ga \to \Ia_B $ for some $B  \in \Omega^2(\bbD^2)$.
Define a unitary isomorphism of hermitean complex lines
\begin{align}
\label{eq:can sect over filled loops}
	\sigma_{f} \colon \C &\to (\partial^*\Ta\Ga)_{f}
	\\
	z &\mapsto \sigma_{f}(z) \coloneqq \Big[ [\Sa_{|\bbS^1}],
            \exp\Big( - \iu\, \int_{\bbD^2}\, B \Big)\, z \Big] \ .
\end{align}
This isomorphism is defined independently of the choice of $\Sa$:
let $\Sa' \colon f^*\Ga \to \Ia_{B '}$ be another trivialisation.
Then the line bundle $\sfR(\Sa' \circ \Sa^{-1})$ has curvature $B ' - B $, which implies that
\begin{align}
	\sigma_{f}(z) &\coloneqq \Big[ [\Sa_{|\bbS^1}], \exp\Big( - \iu\, \int_{\bbD^2}\, B \Big)\, z \Big]
	\\[4pt]
	&= \Big[ [\Sa_{|\bbS^1}], \exp\Big( - \iu\, \int_{\bbD^2}\, B '\Big)\, \hol\big(\bbS^1, \sfR(\Sa' \circ \Sa^{-1})\big)\,  z \Big]
	\\[4pt]
	&= \Big[ [\Sa'_{|\bbS^1}], \exp\Big( - \iu\, \int_{\bbD^2}\,
          B '\Big)\, z \Big] \ .
\end{align}
This construction works equally well if we replace the `round' unit disk $\bbD^2$ by the unit square $[0,1]^2$, as long as we consider maps $f \colon [0,1]^2 \to M$ whose restrictions to $\partial [0,1]^2$ have sitting instants at the corners.
By the construction of the fusion product $\lambda$ on $\Ta\Ga$, the section $\sigma$ is compatible with fusion,
\begin{equation}
	\sigma_{f' \star_2 f} = \lambda(\sigma_{f'}, \sigma_f)
\end{equation}
for all disks $f,f' \colon [0,1]^2 \to M$ that can be concatenated vertically.
(This is merely the statement that the integral over $[0,1]^2$ decomposes as the sum $\int_{[0,1]^2} = \int_{[0,1] \times [0,\frac{1}{2}]} +\int_{[0,1] \times [\frac{1}{2},1]}$.)

Now consider the following setup:
let $\Sigma \colon [0,1]^2 \to M$ be an element in $P_*^2M$, presenting a fixed-end homotopy from a path $\gamma$ to a path $\gamma'$ in $M$.
We want to compare the 1-isomorphisms $(\iota^{2*}_{1;0})^* \pt^{\Ga'}_1$ and $(\iota^{2*}_{1;1})^* \pt^{\Ga'}_1$ of bundle gerbes over $P_*^2M$.
The source bundle gerbes of both these morphisms have subductions
\begin{equation}
	Y_0 \coloneqq (\iota^{2*}_{1;0})^* \ev_0^* P_0 M =
        \ev_{(0,0)}^* P_0 M = \ev_{(1,0)}^* P_0 M \to P_*^2 M \ ,
\end{equation}
while the target bundle gerbes live over
\begin{equation}
	Y_1 \coloneqq (\iota^{2*}_{1;0})^* \ev_1^* P_0 M =
        \ev_{(0,1)}^* P_0 M = \ev_{(1,1)}^* P_0 M \to P_*^2 M \ .
\end{equation}
The fibre product $\widehat{Y} \coloneqq Y_0 \times_{P_*^2 M} Y_1$ is the space of triples $(\alpha_0, \Sigma, \alpha_1)$ of based paths $\alpha_0, \alpha_1 \in P_0 M$ and fixed-ends homotopies $\Sigma \in P_*^2M$ between arbitrary paths in $M$ such that $\alpha_t(1) = \Sigma(0,t)$ for $t = 0,1$ (see Figure~\ref{Fig:Y hat}).
\begin{figure}[h]
\begin{center}
\begin{overpic}[scale=0.8]
{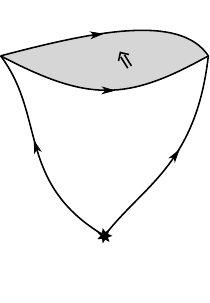}
\put(8,40){$\alpha_0$}
\put(60,40){$\alpha_1$}
\put(30,90){$\Sigma$}
\end{overpic}
\end{center}
\caption{\small An element of the space $\widehat{Y}$.}
\label{Fig:Y hat}
\end{figure}  

The 1-isomorphism $(\iota^{2*}_{1;i})^* \pt^{\Ga'}_1$, for $i = 0,1$, is defined over the subduction
\begin{equation}
	Z_i \coloneqq (\iota^{2*}_{1;i})^* P_{\partial \Delta^2}M \to
        \widehat{Y} \ ,
\end{equation}
which is actually an isomorphism.
Consequently, the 2-isomorphism $\pt^{\Ga'}_2$ should be defined with respect to the subduction
\begin{equation}
	\widehat{Z} \coloneqq Z_0 \times_{\widehat{Y}} Z_1 \to
        \widehat{Y} \ ,
\end{equation}
which again is an isomorphism.
Its elements are triples $(\alpha_0, \Sigma, \alpha_1)$ as above. 
Set $\gamma_t \coloneqq \Sigma \circ \iota^2_{1;t}$ for $t = 0,1$, and
let $x = \gamma_t(0)$ and $y = \gamma_s(1)$ for $t,s = 0,1$.

At a point $(\alpha_0, \Sigma, \alpha_1)$, the morphism of hermitean line bundles over $\widehat{Z}$ that defines $\pt^{\Ga'}_2$ is given by the morphism
\begin{equation}
	\pt^{\Ga'}_{2|(\alpha_0, \Sigma, \alpha_1)} \colon \Ta\Ga_{\overline{\alpha_1} \star (\gamma_0 \star \alpha_0)}
	\to \Ta\Ga_{\overline{\alpha_1} \star (\gamma_1 \star \alpha_0)}
\end{equation}
of complex lines obtained as follows:
\begin{myenumerate}
\item Using a smooth family of rotations of $\bbS^1$, apply parallel transport on $\Ta\Ga$ to obtain an isomorphism
\begin{equation}
	\psi_1 \colon \Ta\Ga_{\overline{\alpha_1} \star (\gamma_0 \star \alpha_0)}
	\to \Ta\Ga_{(\gamma_0 \star \alpha_0) \star \overline{\alpha_1}}
	\to \Ta\Ga_{\gamma_0 \star (\alpha_0 \star
          \overline{\alpha_1})} \ .
\end{equation}
This is achieved by parallel transport along a thin path in $LM$.
Hence, since the parallel transport on $\Ta\Ga$ is superficial, this isomorphism is independent of the choice of a smooth family of rotations.

\item Use the canonical section $\sigma_{\Sigma}(1) \in \Ta\Ga_{\partial \Sigma}$ from (\ref{eq:can sect over filled loops}) to obtain an isomorphism
\begin{equation}
	\psi_2 \colon \Ta\Ga_{\gamma_0 \star (\alpha_0 \star \overline{\alpha_1})}
	\to \Ta\Ga_{\gamma_0 \star (\alpha_0 \star
          \overline{\alpha_1})} \otimes \Ta\Ga_{\partial \Sigma} \ .
\end{equation}

\item 
The boundary loop $\partial \Sigma$ is smoothly and thinly homotopic (via reparameterisations) to $((\id_y \star \gamma_1) \star \id_x) \star \overline{\gamma_0}$, where $\id_x$ is the constant path at the point $x \in M$.
This loop is, in turn, thinly homotopic to $\gamma_1 \star \overline{\gamma_0}$.
We thus obtain a canonical isomorphism
\begin{equation}
	\psi_3 \colon \Ta\Ga_{\gamma_0 \star (\alpha_0 \star \overline{\alpha_1})} \otimes \Ta\Ga_{\partial \Sigma}
	\to \Ta\Ga_{\gamma_0 \star (\alpha_0 \star
          \overline{\alpha_1})} \otimes \Ta\Ga_{\gamma_1 \star
          \overline{\gamma_0}} \ .
\end{equation}

\item The fusion product on $\Ta\Ga$ yields an isomorphism
\begin{equation}
	\psi_4 \colon \Ta\Ga_{\gamma_0 \star (\alpha_0 \star \overline{\alpha_1})} \otimes \Ta\Ga_{\gamma_1 \star \overline{\gamma_0}}
	\to \Ta\Ga_{\gamma_1 \star (\alpha_0 \star
          \overline{\alpha_1})} \ .
\end{equation}

\item Finally, we again use parallel transport along a path in $LM$ that arises from a smooth family of rotations to obtain a canonical isomorphism
\begin{equation}
	\psi_5 \colon \Ta\Ga_{\gamma_1 \star (\alpha_0 \star \overline{\alpha_1})}
	\to \Ta\Ga_{\overline{\alpha_1} \star (\gamma_1 \star
          \alpha_0)} \ .
\end{equation}
\end{myenumerate}
We then define
\begin{equation}
	\pt^{\Ga'}_{2|(\alpha_0,\Sigma,\alpha_1)} \coloneqq \psi_5
        \circ \psi_4\circ \psi_3\circ \psi_2 \circ \psi_1 \ .
\end{equation}

This is compatible with vertical composition in $P_*^2M$:
let $\Sigma, \Sigma' \in P_*^2M$ be two maps $[0,1]^2 \to M$ that can be concatenated vertically.
Since the connection on $\Ta\Ga$ is superficial and compatible with the fusion product, we can replace the morphism $\psi_1$ by
\begin{equation}
	\tilde{\psi}_1 \colon \Ta\Ga_{\overline{\alpha_1} \star (\iota^{2*}_{1;0}\Sigma \star \alpha_0)}
\to \Ta\Ga_{\iota^{2*}_{1;0}\Sigma \star (\alpha_0 \star
  \overline{\alpha_1})} \ .
\end{equation}
Applying the fusion product with $\partial (\Sigma' \star_2 \Sigma)$ yields an isomorphism
\begin{equation}
	\varphi \colon \Ta\Ga_{\iota^{2*}_{1;0}\Sigma \star (\alpha_0 \star \overline{\alpha_1})}
	\to \Ta\Ga_{\iota^{2*}_{1;1}\Sigma' \star (\alpha_0 \star
          \overline{\alpha_1})} \ .
\end{equation}
Combining the fact that the fusion product $\lambda$ is associative
and compatible with the parallel transport on $\Ta\Ga$, that the
parallel transport on $\Ta\Ga$ is superficial (in particular, parallel
transport along thin paths is independent of the choice of thin path),
and that the section $\sigma$ from \eqref{eq:can sect over filled
  loops} is compatible with $\lambda$, it follows that $\pt^{\Ga'}_2$ respects vertical concatenation.

Since all morphisms involved in the construction of $\pt^{\Ga'}_2$ are
smooth, it follows that $\pt^{\Ga'}_2$ is in fact a smooth morphism of bundle gerbes as desired.

\subsubsection*{Construction of $\pt^{\Ga'}_\star$}

The 2-isomorphism 
$$
\pt^{\Ga'}_\star\colon \pr_1^* \pt_1^{\Ga'} \circ \pr_2^* \pt_1^{\Ga'}
\longrightarrow (\,\boldsymbol\cdot\, \star \,\boldsymbol\cdot\,)^*\pt_1^{\Ga'}
$$ 
is directly constructed from the fusion product $\lambda$ on the transgression line bundle $\Ta\Ga$.
Define $q \colon PM \times_M PM \to M$ by $(\gamma, \gamma') \mapsto \gamma(0) = \gamma'(1)$.
The morphism $\pt^{\Ga'}_\star$ is defined over the subduction
\begin{equation}
	Q_1 \coloneqq \pr_1^* P_{\partial \Delta^2}M \times^{\phantom{\dag}}_{q^* P_0M} \pr_2^* P_{\partial \Delta^2}M
	\to PM \times_M PM \ .
\end{equation}
Given a point $\big((\alpha_0, \gamma, \alpha_1), (\alpha_1, \gamma',
\alpha_2)\big) \in Q_1$, the morphism $\pt^{\Ga'}_\star$ is given by
the diagram
\begin{equation}
\begin{tikzcd}[row sep=1.25cm, column sep=1.75cm]
	\Ta\Ga_{\overline{\alpha_1} \star (\gamma \star \alpha_0)} \otimes \Ta\Ga_{\overline{\alpha_2} \star (\gamma' \star \alpha_1)}
	\ar[r] \ar[d, dashed, "\pt^{\Ga'}_\star"']
	& \Ta\Ga_{\overline{\alpha_1} \star (\gamma \star \alpha_0)} \otimes \Ta\Ga_{(\overline{\alpha_2} \star \gamma') \star \alpha_1} \ar[d, "\lambda"]
	\\
	\Ta\Ga_{\overline{\alpha_2} \star (\gamma' \star \gamma) \star \alpha_0}
	& \Ta\Ga_{(\overline{\alpha_2} \star \gamma') \star (\gamma \star \alpha_0)} \ar[l]
\end{tikzcd}
\end{equation}
where the horizontal morphisms are induced by smooth families of reparameterisations.
The compatibility of this morphism with the morphism $\beta$ from~\eqref{eq:def beta} follows again from the superficiality of the connection on $\Ta\Ga$ and the associativity of the fusion product $\lambda$.

The compatibility of $\pt^{\Ga'}_\star$ with $\pt^{\Ga'}_2$ as in~\eqref{eq:compat pt_2 and pt_star} is seen analogously to how we proved the compatibility of $\pt^{\Ga'}_2$ with vertical concatenation of homotopies.
The interchange law~\eqref{eq:interchange law for pt^Ga} is satisfied
by the associativity of $\lambda$, its compatibility with the parallel
transport on $\Ta\Ga$ and with the section $\sigma$ from \eqref{eq:can sect over filled loops}, as well as the superficiality of the connection on $\Ta\Ga$.

\subsubsection*{Construction of $\varepsilon^{\Ga'}$}

Finally, the 2-isomorphism 
$$
\varepsilon^{\Ga'}\colon \sfc^*\pt_1^{\Ga'} \longrightarrow 1_{\Ga'}
$$ 
is obtained directly from the superficial connection on $\Ta\Ga$:
it is defined over the space of triples $(\alpha, \id_x, \alpha) \in
P_{\partial \Delta^2}M$, and all paths
of the form $\ol{\alpha}\star \alpha$ are canonically contractible by thin homotopies.

All necessary coherences in Definition~\ref{def:pt for gerbe} then follow from the superficiality of the parallel transport on $\Ta\Ga$, the associativity of the fusion product $\lambda$ and its compatibility with the section $\sigma$, and the fact that the parallel transport on $\Ta\Ga$ is compatible with the fusion product.
Thus we have

\begin{theorem}
Let $\Ga \in \BGrb^\nabla(M)$ be a bundle gerbe with connection on $M$, and let $\Ga' \coloneqq \Ra\Ta(\Ga) \in \BGrb^\nabla(M)$ be the regression of the transgression of $\Ga$.
Then the quadruple $\pt^{\Ga'} = (\pt^{\Ga'}_1, \pt^{\Ga'}_2, \pt^{\Ga'}_\star, \varepsilon^{\Ga'})$ defines a parallel transport on the bundle gerbe $\Ga'$.
\end{theorem}

\subsubsection*{Transfer to arbitrary bundle gerbes}

In~\cite{Waldorf:Transgression_II}, Waldorf shows that the functors
$\Ta$ and $\Ra$ come with a canonical natural isomorphism
\begin{equation}
	\Aa \colon 1 \to \Ra \circ \Ta
\end{equation}
as endofunctors of the homotopy 1-category $\mathtt{h}\BGrb^\nabla(M)$.
Given a bundle gerbe $\Ga \in \BGrb^\nabla(M)$, we thus get a 2-isomorphism class of 1-isomorphisms $\Ga \to \Ga' = \Ra\Ta(\Ga)$.
Let $\Aa_\Ga \colon \Ga \to \Ga'$ be a representative for this class.

Let $\Ba \colon \Ga \to \Ga'$ be a 1-isomorphism with \emph{adjoint
  inverse} $\Ba^{-1}$, i.e.~a weak inverse $\Ba^{-1}$ together with 2-isomorphisms $\epsilon_\Ba \colon 1_\Ga \longrightarrow \Ba^{-1} \circ \Ba$ and $\delta_\Ba \colon \Ba \circ \Ba^{-1} \longrightarrow 1_{\Ga'}$ that satisfy the triangle identities.
We can use it to define a 1-isomorphism $\pt^{\Ga,\Ba}_1$ as the composition
\begin{equation}
\begin{tikzcd}[row sep=1.25cm, column sep=1.5cm]
	\ev_0^*\Ga' \ar[r, "\pt^{\Ga'}_1"] & \ev_1^*\Ga' \ar[d, "\ev_1^* \Ba^{-1}"]
	\\
	\ev_0^*\Ga \ar[u, "\ev_0^*\Ba"] \ar[r, dashed, "\pt^{\Ga,\Ba}_1"'] & \ev_1^*\Ga
\end{tikzcd}
\end{equation}
We define 2-isomorphisms
\begin{align}
	\pt^{\Ga,\Ba}_2 &\coloneqq 1_{\Ba^{-1}} \circ \pt^{\Ga'}_2
                          \circ 1_\Ba \ ,
	\\[4pt]
	\pt^{\Ga,\Ba}_\star &\coloneqq 1_{\Ba^{-1}} \circ \big(
                              \pt^{\Ga'}_\star \circ_2
                              (1_{\pt^{\Ga'}_1} \circ \delta_\Ba \circ
                              1_{\pt^{\Ga'}_1}) \big) \circ 1_\Ba \ ,
	\\[4pt]
	\varepsilon^{\Ga,\Ba} &\coloneqq \epsilon_\Ba \circ_2
                                (1_{\Ba^{-1}} \circ \varepsilon^{\Ga'}
                                \circ 1_\Ba) \ ,
\end{align} 
where we have omitted pullbacks.
From these definitions we readily see

\begin{proposition}
The quadruple $\pt^{\Ga,\Ba} = (\pt^{\Ga,\Ba}_1, \pt^{\Ga,\Ba}_2, \pt^{\Ga,\Ba}_\star, \varepsilon^{\Ga,\Ba})$ defines an object in $\PT(\Ga)$.
\end{proposition}

For $\psi \colon \Ba \to \Ba'$ a 2-isomorphism of 1-isomorphisms $\Ba, \Ba' \colon \Ga' \to \Ga$, we obtain a 2-isomorphism
\begin{equation}\label{eq:psi2iso}
	\psi^{(-1)} \circ 1_{\pt^{\Ga'}_1} \circ \psi \colon
        \pt_1^{\Ga,\Ba} \to \pt_1^{\Ga,\Ba'} \ .
\end{equation}
Here $\psi^{(-1)}$ denotes the 2-isomorphism obtained from $\psi$ by taking the inverse with respect to horizontal composition.
Again it follows from the definitions that this defines an isomorphism
\begin{equation}
	\widehat{\psi} \colon \pt^{\Ga,\Ba} \to \pt^{\Ga,\Ba'}
\end{equation}
in the category $\PT(\Ga)$.
If $\psi' \colon \Ba \to \Ba'$ is another (parallel unitary)
2-isomorphism, then $\psi$ and $\psi'$ differ by multiplication with a locally constant $\U(1)$-valued function $f_{\psi, \psi'}$ on $M$.
Since horizontal inverses of 2-isomorphisms have dual underlying line bundles~\cite{Waldorf--More_morphisms}, the morphisms $\psi^{(-1)}$ and $\psi'{}^{(-1)}$ differ by the locally constant $\U(1)$-valued function $f_{\psi^{(-1)}, \psi'{}^{(-1)}} = (f_{\psi, \psi'})^{-1}$.
Consequently, we deduce from \eqref{eq:psi2iso} that
\begin{equation}
	\widehat{\psi'} = \widehat{\psi} \ .
\end{equation}
That is, for any pair of parallel unitary 1-isomorphisms $\Ba, \Ba'
\colon \Ga \to \Ga'$ for which there exists some parallel unitary 2-isomorphism $\Ba \to \Ba'$, we obtain a \emph{unique} isomorphism $\pt^{\Ga, \Ba} \to \pt^{\Ga,\Ba'}$.

Let $[[\Aa_\Ga]]$ denote the full subgroupoid of $\BGrb^\nabla(M)(\Ga,\Ga')$ on those 1-isomorphisms $\Ga \to \Ga'$ that are isomorphic to Waldorf's 1-isomorphism $\Aa_\Ga$.
Our constructions define a functor
\begin{equation}
	\pt^{\Ga, (\,\boldsymbol\cdot\,)} \colon [[\Aa_\Ga]] \to \PT(\Ga) \ .
\end{equation}
This functor factors through a groupoid $[[\Aa_\Ga]]_*$ with the same
objects as $[[\Aa_\Ga]]$ and a unique isomorphism between any two
objects. In particular, every object in $[[\Aa_\Ga]]_*$ is final and
the canonical functor $[[\Aa_\Ga]] \to [[\Aa_\Ga]]_*$ is a final
functor. It follows that, for any category $\Cscr$, any
functor $\Fscr \colon [[\Aa_\Ga]] \to \Cscr$ that factors through $[[\Aa_\Ga]]_*$ has a colimit, which is represented by $\Fscr(\Ba)$ for any object $\Ba \in [[\Aa_\Ga]]$.

\begin{definition}\label{Def: parallel transport on Gerbe}
Let $\Ga \in \BGrb^\nabla(M)$ be a bundle gerbe with connection on $M$.
The \emph{parallel transport} of $\Ga$ is
\begin{equation}
	\pt^\Ga \coloneqq \colim \big( \pt^{\Ga, (\,\boldsymbol\cdot\,)} \colon
        [[\Aa_\Ga]] \to \PT(\Ga) \big) \ .
\end{equation}
\end{definition}

\subsection{The transgression line bundle as a holonomy}
\label{sec: Transgression and PT}

Let $\Ga \in \BGrb^\nabla(M)$ be a bundle gerbe with connection on $M$ and write $\Ga' = \Ra\Ta(\Ga)$.
We will now determine the holonomy of the parallel transport on $\Ga'$.
For this, consider the diffeological space $L_*M$ of smooth maps $\bbS^1 \to M$ that have a sitting instant at $1 \in \bbS^1$.
In other words, $L_*M$ is the pullback
\begin{equation}
\begin{tikzcd}[row sep=1.25cm]
	L_*M \ar[r, "\iota"] \ar[d, "\ev_1"'] & PM \ar[d, "\ev_0 \times \ev_1"]
	\\
	M \ar[r, "{\mit\Delta}"'] & M \times M
\end{tikzcd}
\end{equation}
in $\Dfg$, where $\iota$ denotes the inclusion map and ${\mit\Delta}$ is the
diagonal embedding.
The pullback $\iota^* \pt^{\Ga'}_1$ is an automorphism of $\ev_1^*\Ga'$, which we understand as the \emph{holonomy} of $\pt^{\Ga'}$.
It is defined over the subduction $\iota^* P_{\partial \Delta^2} M \to L_*M$.
Recall from Section~\ref{Sec: Bundle gerbes} that a 1-automorphism of a bundle gerbe defines a line bundle via descent.
Thus the holonomy $\iota^*\pt^{\Ga'}_1$ gives rise to a descended line bundle $\hol(\Ga) \in \HLBdl(L_*M)$.
Our goal is to understand this descended line bundle more explicitly.

Let $\widehat{\ev_1} \colon \ev_1^*P_0M \to P_0M$ be the morphism induced by the pullback
\begin{equation}
\begin{tikzcd}[row sep=1cm]
	\ev_1^*P_0M \ar[r, "\widehat{\ev_1}"] \ar[d] & P_0M \ar[d]
	\\
	L_*M \ar[r, "\ev_1"'] & M
\end{tikzcd}
\end{equation}
in $\Dfg$. The hermitean line bundle (with connection) underlying the bundle gerbe $\ev_1^*\Ga'$ is the pullback bundle
\begin{equation}
	L\coloneqq \widehat{\ev_1}^{[2]*}\Ta\Ga \to (\ev_1^*P_0M)^{[2]}
	\cong \iota^* P_{\partial \Delta^2} M \ .
\end{equation}
We now apply the construction from the diagram~\eqref{eq:construction in Section 2} that produces a line bundle $\sfR(\Aa)$ from an automorphism $\Aa$ of a bundle gerbe:
the tensor product bundle $L^\vee \otimes \pt^{\Ga'}_1$ on $\iota^* P_{\partial \Delta^2}M$ has fibres
\begin{equation}
	\big( L^\vee \otimes \pt^{\Ga'}_1 \big)_{(\alpha_0, \gamma, \alpha_1)}
	= \Ta\Ga^\vee_{\overline{\alpha_1} \star \alpha_0} \otimes \Ta\Ga_{\overline{\alpha_1} \star (\gamma \star \alpha_0)}
	\cong \Ta\Ga_{\overline{\alpha_0} \star \alpha_1} \otimes
        \Ta\Ga_{\overline{\alpha_1} \star (\gamma \star \alpha_0)} \ .
\end{equation}
Now the thin-invariant parallel transport and the fusion product on $\Ta\Ga$ yield an isomorphism
\begin{equation}
\label{eq:iso L cong pt_loop}
\begin{tikzcd}[column sep=1.75cm, row sep=1cm]
	\Ta\Ga_{\overline{\alpha_0} \star \alpha_1} \otimes \Ta\Ga_{\overline{\alpha_1} \star (\gamma \star \alpha_0)}
	\ar[r] \ar[d, dashed]
	& \Ta\Ga_{\alpha_1 \star \overline{\alpha_0}} \otimes \Ta\Ga_{\alpha_0 \star (\overline{\alpha_1} \star \gamma)} \ar[d, "\lambda"]
	\\
	\Ta\Ga_{\gamma} & \Ta\Ga_{\alpha_1 \star (\overline{\alpha_1} \star \gamma)} \ar[l]
\end{tikzcd}
\end{equation}
By the associativity of $\lambda$, this is an isomorphism of descent data for line bundles on $L_*M$.
This shows

\begin{proposition}
\label{st:iso L cong pt_loop}
The morphism~\eqref{eq:iso L cong pt_loop} yields an isomorphism of hermitean line bundles
\begin{equation}
	\hol(\Ga) \coloneqq \sfR \big( \iota^*\pt^{\Ga'}_1 \big)
	\overset{\cong}{\to} \Ta\Ga_{|L_*M}
\end{equation}
over $L_*M$.
\end{proposition}

Thus the parallel transport $\pt^{\Ga'}$ reproduces the transgression line bundle $\Ta\Ga$ as its holonomy.

\begin{remark}
For a generic bundle gerbe $\Ga$ with parallel transport $\pt^\Ga$, the morphism $\pt^\Ga_2$ induces a parallel transport on $\hol(\Ga)$.
It should be possible to construct from this a fusion line bundle with connection on $LM$ in the sense of~\cite{Waldorf:Transgression_II}, which then regresses to a bundle gerbe with connection on $M$.
Its underlying bundle gerbe should be canonically isomorphic to $\Ga$
(up to 2-isomorphism), and that should allow the reconstruction of the connection on $\Ga$ from its parallel transport in our sense.
However, this would go beyond the scope of this paper, and since for our applications in Sections~\ref{sect:extensions from gerbes} and~\ref{Sec: Magnetic translations} having an explicit construction for $\pt^{\Ga'}$ is sufficient, we leave this reconstruction of the connection on $\Ga$ for future work.
\qen
\end{remark}

\section{2-group extensions from bundle gerbes}
\label{sect:extensions from gerbes}

Let $G$ be a connected Lie group with a smooth group action on a manifold $M$.
In Section~\ref{Sec: U(1)} we saw how a principal bundle $P \to M$ gives rise to a group extension $\Sym_G(P) \to G$ which encodes all information about equivariant structures on $P$.
We were able to give two equivalent constructions for $\Sym_G(P)$, one as a subgroup of $\Diff(P)$, and one as descent data associated to the path fibration $P_0 G \to G$ and a parallel transport on $P$.

In this section we study the analogous situation for a bundle gerbe $\Ga \in \BGrb(M)$ instead of a principal bundle $P\in\Bun_H(M)$.
There are two main differences to the situation in Section~\ref{Sec: U(1)}:
equivariant structures on $\Ga$ form a groupoid rather than a set, and they do not assemble into a topological or smooth space.
We thus cannot expect a universal extension $\Sym_G(\Ga) \to G$ as diffeological groups.
A good framework to describe this extension is that of group objects in categories fibred in groupoids over $\Cart$, where the fibration encodes the smooth structure.
After carefully setting up this framework, we give two constructions of $\Sym_G(\Ga)$, in analogy to the two constructions of $\Sym_G(P)$ in Section~\ref{Sec: U(1)}.
We conclude this section by showing that, again, the extension $\Sym_G(\Ga) \to G$ encodes all information about equivariant structures on $\Ga$.

\subsection{Smooth groupoids and symmetries of gerbes}
\label{sect:smoothsymgerbes}

Let $\Ga \in \BGrb(M)$ be a bundle gerbe on $M$.
Let $\Phi \colon G \times M \to M$ be an action of a connected Lie (or
diffeological) group $G$ on $M$.
Let $\Cart$ denote the category of smooth manifolds that are diffeomorphic to $\R^n$ for some $n \in \N_0$.
The morphisms in $\Cart$ are the smooth maps between these manifolds.

We can view $M$ and $G$ as presheaves on $\Cart$ by setting
\begin{equation}\label{Def: G as sheave}
	M(c) = C^\infty(c,M)
	\qandq
	G(c) = C^\infty(c,G) \ .
\end{equation}
By adding identity morphisms, we can canonically enhance the presheaf
$G$ to a (pre)stack on $\Cart$, i.e.~a (pre)sheaf of groupoids, which we still denote by $G$.
Given a section $f \in G(c)$, i.e.~a smooth map $f \colon c \to G$, we can define a map
\begin{equation}
\begin{tikzcd}[column sep=1.25cm]
	\Phi_f \colon c \times M \ar[r, "f \times 1_M"] & G \times M
        \ar[r, "\Phi"] & M \ .
\end{tikzcd}
\end{equation}
We can then assign to $f$ the groupoid
\begin{equation}
	\Sym_G^{\PSh}(\Ga)(f)
	\coloneqq \BGrb(c \times M) (\pr_M^*\Ga, \Phi_f^*\Ga) \ ,
\end{equation}
where $\pr_M \colon c \times M \to M$ is the projection.
The groupoid $\Sym_G^{\PSh}(\Ga)(f)$ is non-empty:
since the map $f \colon c \to G$ is homotopic to the constant map at the identity in $G$, it follows that $\pr_M^*\Ga$ and $\Phi_f^*\Ga$ have the same Dixmier-Douady class as gerbes on $c \times M$, so that there exists an isomorphism $\pr_M^*\Ga \to \Phi_f^*\Ga$.

The assignment $f \mapsto \Sym_G^{\PSh}(\Ga)(f)$ is evidently not a presheaf of groupoids on $\Cart$ since it depends not only on the object $c$, but also on a choice of a smooth map $f \colon c \to G$.
We can reformulate this in the following way:
let $\ul{G}$ denote the category with objects the smooth maps $f \colon c \to G$, where $c \in \Cart$ is any Cartesian space.
The morphisms
\begin{equation}
	\varphi \colon (f \colon c \to G) \to (f' \colon c' \to G)
\end{equation}
in $\ul{G}$ are commutative triangles of smooth maps
\begin{equation}
\begin{tikzcd}
	c \ar[rr, "\varphi"] \ar[dr, "f"'] & & c' \ar[dl, "f'"]
	\\
	& G &
\end{tikzcd}
\end{equation}
Then $\Sym_G^{\PSh}(\Ga)$ is a presheaf of groupoids on $\ul{G}\,$:
to an object $f \colon c \to G$ in $\ul{G}$ we assign the groupoid $\Sym_G^{\PSh}(\Ga)(f)$, while to a morphism $\varphi \colon f \to f'$ we assign the pullback functor
\begin{equation}
	\Sym_G^{\PSh}(\Ga)(\varphi) \coloneqq (\varphi \times 1_M)^*
        \colon \Sym_G^{\PSh}(\Ga)(f') \to \Sym_G^{\PSh}(\Ga)(f) \ .
\end{equation}
By a slight abuse of notation, we will denote the functor $\Sym_G^{\PSh}(\Ga)(\varphi)$ by $\varphi^*$.
Explicitly, given a 1-isomorphism $A \colon \pr_M^*\Ga \to
\Phi_{f'}^*\Ga$ over $c'$, it is defined by the commutative diagram
\begin{equation}
\begin{tikzcd}[column sep=1.75cm, row sep=1cm]
	\pr^*_M \Ga \ar[rr, dashed, "\varphi^*A"] \ar[d, "\cong"'] & & \Phi_f^* \Ga
	\\
	(\varphi \times 1_M)^* \pr_M^{\prime\,\ast} \Ga \ar[r, "(\varphi \times 1_M)^*A"'] & (\varphi \times 1_M)^* \Phi_{f'}^* \Ga \ar[r, "\cong"'] & \big( \Phi_{f'} \circ (\varphi \times 1_M) \big)^* \Ga \ar[u, "\cong"']
\end{tikzcd}
\end{equation}
By construction of the 2-category of bundle gerbes, this defines a pseudofunctor
\begin{equation}
	\Sym_G^{\PSh}(\Ga) \colon \ul{G}^\opp \to \Grpd \ ,
\end{equation}
where $\Grpd$ is the 2-category of groupoids, functors, and natural transformations.

\begin{remark}
The assignment $f \mapsto \Sym_G^{\PSh}(\Ga)(f)$ is not a strict
presheaf of groupoids on $\ul{G}\,$, as it is only pseudofunctorial~\cite{Moerdijk:Stacks_and_gerbes}.
There are several ways to technically treat such pseudo-presheaves of groupoids:
\begin{myenumerate}
\item Encode the coherence morphisms by viewing pseudo-presheaves of
  groupoids as coherent simplicial presheaves, i.e.~as simplicial
  functors $\mathfrak{C} \circ {\rm N}_\Delta (\ul{G})^\opp \to \Set_\Delta$ in the notation of~\cite{Lurie:HTT}.

\item Use a strictification procedure to translate pseudo-presheaves of groupoids into presheaves of groupoids~\cite{Hollander:HoThy_for_stacks}.

\item Use the Grothendieck construction, or straightening, to translate pseudo-presheaves of groupoids into categories fibred in groupoids over $\ul{G}$~\cite{Vistoli:Fib_Cats,Lurie:HTT}.
\end{myenumerate}
We will follow the third approach here because the transition between the parameterising categories $\ul{G}$ and $\Cart$ becomes particularly easy in that framework.
\qen
\end{remark}

We will frequently make use of the Grothendieck construction to pass from $\Grpd$-valued pseudo-functors to categories fibred in groupoids; for background on the Grothendieck construction and fibred categories we refer to~\cite{Vistoli:Fib_Cats,Lurie:HTT}.
We will, however, describe the resulting fibred categories explicitly.
For example, the canonical projection functor $\pr:\ul{G} \to \Cart$ is the category fibred in groupoids obtained by applying the Grothendieck construction to the (pseudo)functor $c \mapsto G(c)$, where $G(c)$ is regarded as a groupoid with only identity arrows.

\begin{definition}
\label{def:GrFib in Grpds}
A functor $\pi \colon \Dscr \to \Cscr$ between categories is a \emph{Grothendieck
  fibration in groupoids}, or makes $\Dscr$ into a \emph{category
  fibred in groupoids} over $\Cscr$, if it satisfies the properties:
\begin{myenumerate}
\item For every object $d \in \Dscr$ and for every morphism $f \colon
  c \to \pi(d)$ in $\Cscr$, there exists a morphism $\widehat{f}
  \colon \widehat{c} \to d$ in $\Dscr$ with $\pi(\widehat{f}\ ) = f$.

\item For every pair of diagrams
\begin{equation}
\label{eq:cartesian morphism}
\begin{tikzcd}
	& d_0 \ar[dl, dashed, " \widehat{f_{01}}"'] \ar[dr, "\zeta_{02}"] &
	\\
	d_1 \ar[rr, "\zeta_{12}"'] & & d_2
	\\
	& \pi(d_0) \ar[dl, "f_{01}"'] \ar[dr, "\pi(\zeta_{02})"] &
	\\
	\pi(d_1) \ar[rr, "\pi(\zeta_{12})"'] & & \pi(d_2)
\end{tikzcd}
\end{equation}
in $\Dscr$ and $\Cscr$, respectively, there exists a unique lift $\widehat{f_{01}}$ of $f_{01}$ that makes the upper triangle commute.
\end{myenumerate}
\end{definition}

The first requirement resembles a path-lifting condition. The
second requirement can be viewed as a relative horn-filling property:
given any $\Lambda^2_2$-horn $\sigma$ in $\Dscr$ and a filling of
$\pi(\sigma)$ in $\Cscr$ to a 2-simplex, there exists a unique filling
of $\sigma$ to a 2-simplex in $\Dscr$ that lifts the 2-simplex in
$\Cscr$. Alternatively, consider an arbitrary functor $\pi \colon
\Dscr \to \Cscr$ between categories and a morphism $\zeta_{12} \colon
d_1 \to d_2$ in $\Dscr$.
If for every pair of solid arrow diagrams as in~\eqref{eq:cartesian morphism} the dashed arrow exists such that the upper triangle commutes and such that \smash{$\pi(\widehat{f_{01}}) = f_{01}$}, one says that $\zeta_{12}$ is \emph{$\pi$-Cartesian}.
In particular, if $\pi$ is a Grothendieck fibration in groupoids, then property~(2) of Definition~\ref{def:GrFib in Grpds} is equivalent to saying that every morphism in $\Dscr$ is $\pi$-Cartesian.
If $\pi \colon \Dscr \to \Cscr$ is a Grothendieck fibration in
groupoids and $c \in \Cscr$, we denote by $\Dscr_{|c} = \pi^{-1}(c)$ the
fibre over $c$, which is the groupoid with objects $d \in \Dscr$ such
that $\pi(d) = c$ and morphisms \smash{$\widehat{f} \colon d \to d'$} such
that \smash{$\pi(\widehat{f}\ ) = 1_c$}.

\begin{definition}
\label{def:Hscr}
A category fibred in groupoids over $\Cart$ is a \emph{smooth groupoid}.
Let $\Hscr$ denote the strict 2-category of smooth groupoids.
Its objects are smooth groupoids, its morphisms are (strictly) commutative diagrams of functors
\begin{equation}
\begin{tikzcd}[column sep=0.5cm]
	\sfX_0 \ar[rr] \ar[dr] & & \sfX_1 \ar[dl]
	\\
	& \Cart &
\end{tikzcd}
\end{equation}
and its 2-morphisms are natural transformations that project to the identity.
We denote by $\ul{\Hscr}(\sfX,\sfZ)$ the groupoid of functors $\sfX \to \sfZ$ that project to the identity on $\Cart$.
\end{definition}

\begin{example}
Let $M$ be a smooth manifold.
An important example of a smooth groupoid is given by the Grothendieck fibration $\HLBdl^M \to \Cart$, whose objects are pairs $(c,L)$ of a Cartesian space $c \in \Cart$ and a hermitean line bundle $L \to c \times M$, and whose morphisms $(c,L) \to (c',L')$ are pairs $(\varphi, \psi)$ of a smooth map $\varphi \colon c \to c'$ and an isomorphism $\psi \colon L \to (\varphi \times 1_M)^*L'$ of hermitean line bundles on $c \times M$.
One can interpret $\HLBdl^M$ as describing smooth families of hermitean line bundles on $M$.
For $M = *$, we write $\HLBdl^* \eqqcolon \HLBdl$.
\qen
\end{example}

\begin{definition}
\label{def:Sym(G) as fibred category}
Let $p\colon \Sym_G(\Ga) \to \ul{G}$ denote the category fibred in groupoids obtained by applying the Grothendieck construction to the pseudofunctor \smash{$\Sym_G^{\PSh}(\Ga) \colon \ul{G}^\opp \to \Grpd$}.
Explicitly, the category $\Sym_G(\Ga)$ consists of:
\begin{myitemize}
\item $\ul{\rm Objects:}$ pairs $(f,A)$, where $f \in \ul{G}$ is a smooth map $f \colon c \to G$ and $A \colon \pr_M^*\Ga \to \Phi_f^*\Ga$ is a 1-isomorphism of bundle gerbes over $c \times M$.

\item $\ul{\rm Morphisms:}$ a morphism $(f_0,A_0) \to (f_1, A_1)$ is a pair $(\varphi, \psi)$ of a morphism $\varphi \colon f_0 \to f_1$ in $\ul{G}$ and a 2-isomorphism $\psi \colon A_0 \to \varphi^*A_1$ in $\Sym_G^{\PSh}(\Ga)(f_0)$.
\end{myitemize}
\end{definition}

The functor $p\colon \Sym_G(\Ga) \to \ul{G}$ is automatically a fibration in groupoids, since it arises as the Grothendieck construction of a pseudo-presheaf of groupoids.
Since Grothendieck fibrations are stable under composition~\cite{Vistoli:Fib_Cats}, the composite functor
\begin{equation}
\begin{tikzcd}[column sep=0.75cm, row sep=1cm]
	\Sym_G(\Ga) \ar[rr, "p"] \ar[dr, dashed, "\pi"'] & & \ul{G} \ar[dl, "\pr"]
	\\
	& \Cart &
\end{tikzcd}
\end{equation}
makes $\Sym_G(\Ga)$ into a smooth groupoid.

\subsection{Smooth 2-groups}
\label{sect:smooth2groups}

We would now like to establish that $\Sym_G(\Ga)$ is not just a smooth groupoid, but can also be regarded as a higher group in some sense.
That is, we would like to find on $\Sym_G(\Ga)$ the same type of structure as we found on the bundle $\Sym_G(P) \to G$ in Section~\ref{sect:gauge bundles on groups}.
Here, however, we are working inside the ambient 2-category $\Hscr$, and so we will need to make precise what we mean by a group in $\Hscr$.
The notion of a group object in a 2-category goes back to~\cite{BL:2-Groups}.
The following definitions are taken from~\cite{SP:String_group} which are strongly based on~\cite{BL:2-Groups}.
Let $\Cscr$ be a 2-category with finite products; in particular, it has a terminal object $*$.
Examples are the 2-categories $\Grpd$ and~$\Hscr$.

\begin{definition}[\cite{BL:2-Groups}]
A \emph{monoid object} in $\Cscr$ is a sextuple $(\sfH, \sfu, \otimes, \sfa, \sfl, \sfr)$ of 
\begin{myitemize}
\item an object $\sfH \in \Cscr$,

\item 1-morphisms $\sfu \colon * \to \sfH$ and $\otimes \colon \sfH
  \times \sfH \to \sfH$, and

\item 2-isomorphisms
\begin{align}
	\sfa \colon \otimes \circ\, (\otimes \times 1_\sfH) &\to \otimes
                                                            \circ
                                                            (1_\sfH
                                                            \times
                                                            \otimes) \
                                                            ,
	\\[4pt]
	\sfl \colon \otimes \circ\, (\sfu \times 1_\sfH) &\to 1_\sfH \ ,
	\\[4pt]
	\sfr \colon \otimes \circ\, (1_\sfH \times \sfu) &\to 1_\sfH \ .
\end{align}
\end{myitemize}
These data are required to satisfy a pentagon and a triangle identity; see~\cite[Definition~41]{SP:String_group}.

An \emph{abelian monoid object} comes with an additional 2-isomorphism $\beta \colon \otimes \circ\, \tau \to \otimes$ satisfying the coherence conditions in~\cite[Definition~47]{SP:String_group}, where $\tau \colon \sfH \times \sfH \to \sfH \times \sfH$ is the interchange of factors.
\end{definition}

\begin{definition}
A \emph{1-morphism} of monoid objects $(\sfH, \sfu, \otimes, \sfa, \sfl, \sfr) \to (\sfH', \sfu', \otimes', \sfa', \sfl', \sfr')$ in $\Cscr$ consists of a triple $(F_1, F_\otimes, F_\sfu)$ of
\begin{myitemize}
\item a 1-morphism $F_1 \colon \sfH \to \sfH'$ and

\item 2-isomorphisms $F_\otimes \colon \otimes' \circ\, (F_1 \times F_1) \to F_1 \circ \otimes$ and $F_\sfu \colon \sfu' \to F_1 \circ \sfu$.
\end{myitemize}
These are required to satisfy the coherence conditions in~\cite[Definition~42]{SP:String_group}.

Morphisms of abelian monoid objects satisfy an additional compatibility condition for the symmetries $\beta$ and $\beta'$, which can be found in~\cite[Definition~48]{SP:String_group}.
\end{definition}

\begin{definition}[{\cite[Definition~43]{SP:String_group}}]
A \emph{2-morphism} $(F_1, F_\otimes, F_\sfu) \to (E_1, E_\otimes, E_\sfu)$ of monoid objects in $\Cscr$ is a 2-morphism $\theta \colon F_1 \to E_1$ such that the diagrams
\begin{equation}
\begin{tikzcd}[column sep=2cm, row sep=1cm]
	\otimes' \circ (F_1 \times F_1) \ar[r, "\otimes' \circ (\theta \times \theta)"] \ar[d, "F_\otimes"']
	& \otimes' \circ (E_1 \times E_1) \ar[d, "E_\otimes"]
	\\
	F_1 \circ \otimes \ar[r, "\theta \circ \otimes"'] & E_1 \circ \otimes
\end{tikzcd}
\end{equation}
\begin{equation}
\begin{tikzcd}[column sep={1.5cm,between origins}, row sep=1cm]
	& \sfu' \ar[dl, "F_\sfu"'] \ar[dr, "E_\sfu"] & 
	\\
	F_1 \circ \sfu \ar[rr, "\theta \circ \sfu"'] & & E_1\circ\sfu
\end{tikzcd}
\end{equation}
commute.
2-morphisms of abelian monoid objects are 2-morphisms of the underlying monoid objects.
\end{definition}

\begin{example}
A monoid object in the 2-category $\Cat$ of categories is precisely a monoidal category.
Similarly, 1-morphisms and 2-morphisms between monoid objects in $\Cat$ are precisely the monoidal functors and the monoidal natural transformations, respectively.
The abelian monoids in $\Cat$ are precisely the symmetric monoidal categories.
\qen
\end{example}

\begin{definition}[{\cite[Definition~41]{SP:String_group}}]
A \emph{group object} in $\Cscr$ is a monoid object $(\sfH, \sfu, \otimes, \sfa, \sfl, \sfr)$ in $\Cscr$ such that the 1-morphism
\begin{equation}
	(\otimes, \pr_1) \colon \sfH \times \sfH \to \sfH \times \sfH
\end{equation}
is (weakly) invertible.
An \emph{abelian group object} in $\Cscr$ is an abelian monoid object whose underlying monoid object is a group object.
\end{definition}

For $\Cscr$ a 2-category with finite products, we denote the 2-category of group objects in $\Cscr$ by $2\Grp(\Cscr)$.

\begin{definition}
\label{Def: 2-group}
The 2-category of \emph{2-groups} is $2\Grp(\Grpd)$.
The 2-category of \emph{smooth 2-groups} is $2\Grp(\Hscr)$.
\end{definition}

Both these 2-categories are enriched in groupoids.
Let us examine Definition~\ref{Def: 2-group} a little more closely.
Consider two objects $\pi_\sfC \colon \sfC \to \Cart$ and $\pi_\sfD \colon \sfD \to \Cart$ in $\Hscr$.
The product in $\Hscr$ is given by the pullback in $\Cat$:
\begin{equation}
	\big (\sfC \xrightarrow{ \ \pi_\sfC \ } \Cart \big) \times \big (\sfD
        \xrightarrow{ \ \pi_\sfD \ } \Cart \big)
	= \big( \sfC \times_\Cart \sfD \to \Cart \big) \ .
\end{equation}
Explicitly, the category $\sfC \times_\Cart \sfD$ has
\begin{myitemize}
\item $\ul{\rm Objects:}$ pairs $(c \in \sfC, d \in \sfD)$ such that $\pi_\sfC(c) = \pi_\sfD(d)$.

\item $\ul{\rm Morphisms:}$ pairs $(\phi, \psi)$ of morphisms $\phi$ in $\sfC$ and $\psi$ in $\sfD$ such that $\pi_\sfC (\phi) = \pi_\sfD (\psi)$.
\end{myitemize}
A monoid structure on $\sfC \in \Hscr$ thus allows us to multiply pairs of objects \emph{in the same fibre} and pairs of morphisms that lie \emph{over the same morphism} in $\Cart$.

\begin{example}
The tensor product of line bundles turns the presheaf of groupoids of hermitean line
bundles with connection $\HLBdl^{\nabla} \to \Cart$ into an abelian group object in $\Hscr$.
Similarly, for any manifold $M$ it also turns the internal hom
$\big(\HLBdl^{\nabla}\big)^M$ into an abelian group object in $\Hscr$.
\qen
\end{example}

\subsection{Smooth principal 2-bundles}
\label{sect:principal2bundles}

We shall now establish our precise notion of an extension of smooth 2-groups.

\begin{definition}
Let $\Cscr$ be a 2-category with finite products, let $(\sfH, \sfu, \otimes_\sfH, \sfa, \sfl, \sfr)$ be a monoid object in $\Cscr$, and let $\sfC \in \Cscr$.
A \emph{right action} of $\sfH$ on $\sfC$ is a morphism $\otimes
\colon \sfC \times \sfH \to \sfC$ in $\Cscr$, together with
2-morphisms $\alpha$ and $ u$ in $\Cscr$ that witness the commutativity of the diagrams
\begin{equation}
\begin{tikzcd}[row sep=1cm, column sep=1.75cm]
	\sfC \times \sfH \times \sfH \ar[r, "1_\sfC \times \otimes_\sfH"] \ar[d, "\otimes \times 1_\sfH"']
	& |[alias=A]| \sfC \times \sfH \ar[d, "\otimes"]
	\\
	|[alias=B]| \sfC \times \sfH \ar[r, "\otimes"']
	& \sfC
	\arrow[Rightarrow, from=A, to=B, shorten <= 5, shorten >= 5, "\alpha"]
\end{tikzcd}
\qquad \mbox{and} \qquad
\begin{tikzcd}[row sep=1cm, column sep=1.5cm]
	\sfC \ar[r, "1_\sfC \times \sfu"] \ar[dr, "1_\sfC"'""{name=V, inner sep=0pt, below, pos=.5}]
	& |[alias=U]| \sfC \times \sfH \ar[d, "\otimes"] 
	\\
	& \sfC 
	\arrow[Rightarrow, from=U, to=V, shorten <= 3, shorten >= 3, "u"]
\end{tikzcd}
\end{equation}
and that are coherent with respect to the 2-isomorphism $\sfa$, $\sfl$
and $\sfr$.
Left actions are defined analogously.
\end{definition}

\begin{example}
The standard example for an action of a monoid object is that of a module category $\sfC$ over a monoidal category $\sfH$ in $\Cscr = \Cat$.
\qen
\end{example}

\begin{definition}
\label{def:hopb of fibd Grpds}
Let $\Cscr$ be a category.
Suppose there are categories fibred in groupoids $\pi_i \colon \sfD_i \to
\Cscr$, for $i=0,1$, and $\pi_\sfE \colon \sfE \to \Cscr$ over
$\Cscr$, and suppose there is a diagram
\begin{equation}
\begin{tikzcd}
	& \sfE & \\ \sfD_0 \ar[ur, "F_0"] & & \sfD_1 \ar[ul, "F_1"']
\end{tikzcd}
\end{equation}
of categories fibred in groupoids over $\Cscr$.
The \emph{homotopy pullback} $\sfD_0 \times_{\sfE}^{\mathtt{h}}
\sfD_1$ is the category with
\begin{myitemize}
\item $\ul{\rm Objects:}$ triples $(d_0,\eta, d_1)$, where $d_i \in \sfD_i$ and $\eta \colon F_0(d_0) \to F_1(d_1)$ is an isomorphism in $\sfE$ that projects to the identity under $\pi_\sfE$.

\item $\ul{\rm Morphisms:}$ a morphism $(d_0,\eta,d_1) \to (d'_0, \eta',
  d'_1)$ is a pair $(\psi_0, \psi_1)$ of morphisms $\psi_i \colon d_i
  \to d'_i$ such that the diagram
\begin{equation}
\begin{tikzcd}[row sep=1cm]
	F_0(d_0) \ar[r, "\eta"] \ar[d, "F_0(\psi_0)"']
	& F_1(d_1) \ar[d, "F_1(\psi_1)"]
	\\
	F_0(d'_0) \ar[r, "\eta'"'] & F_1(d'_1)
\end{tikzcd}
\end{equation}
commutes in $\sfE$.
\end{myitemize}
\end{definition}

This comes with a canonical functor
\begin{align}
	\pi_{\mathtt{h}} \colon \sfD_0 \times_{\sfE}^{\mathtt{h}}
  \sfD_1 &\to \Cscr
	\\
	(d_0,\eta,d_1) &\mapsto \pi_0(d_0)=\pi_1(d_1)
	\\
	(\psi_0,\psi_1) &\mapsto \pi_0(\psi_0)=\pi_1(\psi_1) \ ,
\end{align}
which, as we show in Appendix~\ref{app:Principal 2-bundles}, is a
Grothendieck fibration in groupoids.

\begin{definition}
\label{def:sfH-bundle}
Let $\sfH$ be a smooth 2-group, and let $\sfX \in \Hscr$ be any smooth groupoid.
An \emph{$\sfH$-principal 2-bundle} on $\sfX$ is an object $\sfP \in
\Hscr$ with a morphism $\pi \colon \sfP \to \sfX$, a right action $(\otimes, \alpha)$ of $\sfH$ on $\sfP$ and a 2-isomorphism
\begin{equation}
\begin{tikzcd}[row sep=1cm]
	\sfP \times \sfH \ar[r, "\otimes"] \ar[d, "\pr"'] & |[alias=A]| \sfP \ar[d, "\pi"]
	\\
	|[alias=B]| \sfP \ar[r, "\pi"'] & \sfX
	\arrow[Rightarrow, from=B, to=A, shorten <= 5, shorten >= 5, "\eta"]
\end{tikzcd}
\end{equation}
such that
\begin{myenumerate}
\item the functor $\pi \colon \sfP \to \sfX$ is an essentially
  surjective Grothendieck fibration,

\item the action $(\otimes, \alpha)$ of $\sfH$ on $\sfP$ and the
  2-isomorphism $\eta$ are compatible in the sense that the diagram 
\begin{equation}
\begin{tikzcd}
	& \sfP \times \sfH \ar[dd, "\pr"' {description, pos=0.25}] \ar[rr, "\otimes"] & & \sfP \ar[dd, "\pi"]
	\\
	\sfP \times \sfH \times \sfH \ar[dd, "1_\sfP \times \otimes_\sfH"'] \ar[rr, crossing over, "\otimes \times 1_\sfH" pos=0.75] \ar[ur, "\pr"]
	& & \sfP \times \sfH \ar[ur, "\pr"'] 
	&
	\\
	& \sfP \ar[rr, "\pi"' pos=0.25] & & \sfX
	\\
	\sfP \times \sfH \ar[ur, "\pr"] \ar[rr, "\otimes"'] & & \sfP \ar[ur, "\pi"'] \ar[from=uu, crossing over, "\otimes" {description, pos=0.25}] &
\end{tikzcd}
\end{equation}
is coherent, where the front face carries the 2-isomorphism $\alpha$ that is part of the action of $\sfH$ on $\sfP$, the back, right-hand, and bottom faces carry the 2-isomorphism $\eta$, and the left-hand face commutes strictly,

\item the composition $\sfP
  \times \sfH \to \sfP \times_\sfX \sfP \times \sfH \to \sfP
  \times_\sfX^{\mathtt{h}} \sfP$ is an equivalence, where the first functor is
  induced by the diagonal functor $\sfP \to \sfP \times_\sfX \sfP$.
\end{myenumerate}
\end{definition}

The first condition can be understood as demanding that $\sfP \to
\sfX$ has local sections (see Lemma~\ref{st:surjectivity lemma} from Appendix~\ref{app:Principal 2-bundles}).
The second condition implements the property that the $\sfH$-action preserves the projection to $\sfX$ up to coherent homotopy.
The third condition says that the $\sfH$-action is principal.
Note that upon choosing an inverse to the equivalence $\sfP
\times_\sfX \sfP \hookrightarrow \sfP \times_\sfX^{\mathtt{h}} \sfP$,
one could equivalently formulate condition (3) using strict pullbacks
alone (again by Lemma~\ref{st:surjectivity lemma} from
Appendix~\ref{app:Principal 2-bundles}).

In order to understand the notion of an extension of smooth 2-groups, we first need to define the kernel of a morphism of smooth 2-groups.
Naively, the kernel could easily be defined as a fibre over $\sfu$, but the resulting category will not generally be fibred in groupoids over $\Cart$.
As it turns out, the homotopy pullback does satisfy this property.

\begin{definition}
\label{def:ker^h(p)}
Let $p \colon \sfH \to \sfG $ be a morphism of smooth 2-groups in $\Hscr$.
Its \emph{kernel} $\ker^{\mathtt{h}}(p)$ is the homotopy pullback
\begin{equation}
\begin{tikzcd}[row sep=1cm]
	\ker^{\mathtt{h}}(p) \ar[r] \ar[d, dashed,  "\kappa"'] & \sfH \ar[d, "p"]
	\\
	\Cart = *_\Hscr \ar[r, "\sfu_\sfG"'] & \sfG
\end{tikzcd}
\end{equation}
\end{definition}

Explicitly, $\ker^{\mathtt{h}}(p)$ is given by
\begin{equation}
	\ker^{\mathtt{h}}(p) \coloneqq *_\Hscr \times_\sfG^{\mathtt{h}} \sfH \ .
\end{equation}
Using Definition~\ref{def:hopb of fibd Grpds} we can equivalently describe it as the category with objects given by pairs $(h,\eta)$ of an object $h \in \sfH$ and an isomorphism $\eta \colon p(h) \to \sfu_\sfG(\pi_\sfH(h))$ in $\sfG$.
Its morphisms $(h_0,\eta_0) \to (h_1,\eta_1)$ are given by morphisms $\zeta \colon h_0 \to h_1$ such that $\eta_1 \circ p(\zeta) = u_\sfG(\pi_\sfH(\zeta))\circ\eta_0$.
We readily observe that the restrictions of the structure morphisms
$\otimes_\sfH$, $\sfa_\sfH$, $\sfl_\sfH$ and $\sfr_\sfH$, together with the morphism $\sfu_\sfH$, turn $\ker^{\mathtt{h}}(p)$ into a smooth 2-group.
It should also be possible to turn the strict kernel $\ker(p)$ into a smooth 2-group in this case, using an inverse to the equivalence $\ker(p) \hookrightarrow \ker^{\mathtt{h}}(p)$ (compare Lemma~\ref{st:GrFib Lemma} from
Appendix~\ref{app:Principal 2-bundles}), but the homotopy-kernel
$\ker^{\mathtt{h}}(p)$ carries a \emph{canonical} 2-group structure,
and using $\ker(p)$ instead of $\ker^{\mathtt{h}}(p)$ would make Construction~\ref{cstr: fun G -> Aut} below rather cumbersome.

\begin{lemma}
\label{st:ker^h is GrFibd over Cart}
In the setting of Definition~\ref{def:ker^h(p)}, the functor $\kappa$
is a Grothendieck fibration in groupoids.
\end{lemma}

\begin{proof}
This follows directly by Lemma~\ref{st:GrFib Lemma}~(1) from Appendix~\ref{app:Principal 2-bundles}.
\end{proof}

Let $\ker(p)$ denote the strict pullback of the diagram $\Cart \overset{\sfu_\sfG}{\to} \sfG \overset{p}{\longleftarrow} \sfH$.
Explicitly, it is the category with objects $h \in \sfH$ such that
$p(h) = \sfu_\sfG(\pi_\sfH(h))$ and morphisms $\zeta \colon h_0 \to h_1$
such that $p(\zeta) = \sfu_\sfG(\pi_\sfH(\zeta))$.
The functor $\ker(p) \to \Cart$ is not a Grothendieck fibration in groupoids in general.
However, if $p \colon \sfH \to \sfG$ is a Grothendieck fibration in groupoids, then so are the functors $\ker^{\mathtt{h}}(p) \to \Cart$ and $\ker(p) \to \Cart$, and
 the canonical inclusion $\ker(p) \hookrightarrow
\ker^{\mathtt{h}}(p)$ is an equivalence.
The next definition is loosely modelled on~\cite[Definition~75]{SP:String_group}.

\begin{definition}
\label{def:2-group extensions}
Let $\sfA$ and $ \sfG$ be smooth 2-groups.
An \emph{extension} of $\sfG$ by $\sfA$  is a pair $(F,p)$ of a morphism of smooth 2-groups $p \colon \sfH \to \sfG$ that turns $\sfH$ into a $\ker^{\mathtt{h}}(p)$-principal 2-bundle over $\sfG$, and an equivalence of smooth 2-groups $F \colon \sfA \to \ker^{\mathtt{h}}(p)$.
\end{definition}

By Lemma~\ref{st:GrFib Lemma} from Appendix~\ref{app:Principal 2-bundles}, we could equivalently require $p$ to turn $\sfH$ into a $\ker(p)$-principal 2-bundle, but then we would need to use the non-canonical 2-group structure on $\ker(p)$.
This essentially amounts to choosing an inverse for the equivalence
$\ker(p) \hookrightarrow \ker^{\mathtt{h}}(p)$.

Our goal now is to define when an extension of smooth 2-groups is central.
Again, we follow the ideas of~\cite{SP:String_group}, where the criterion for an extension of $\sfG$ by $\sfA$ to be central is formulated using a functor $\sfG \to \Aut(\sfA)$ from $\sfG$ into the automorphisms of $\sfA$ as a 2-group; the smooth structure does not matter here.
In~\cite{SP:String_group}, this functor is obtained from abstract arguments.

\begin{construction}
\label{cstr: fun G -> Aut}
In our formalism, we can understand this construction as follows:
consider smooth 2-groups $\sfG$ and $\sfA$, where $\sfA$ is abelian,
and let $(F,p)$ be a smooth 2-group extension of $\sfG$ by $\sfA$,
with morphism $p \colon \sfH \to \sfG$.
Then $\sfA$ is abelian if and only if $\ker^{\mathtt{h}}(p)$ is
abelian, which is true if and only if $\ker(p)$ is abelian (since the 2-group structure induces Picard groupoid structures on the fibres of these smooth 2-groups, where $F$ induces monoidal equivalences).
Fix an arbitrary Cartesian space $c \in \Cart$.
Let $\Aut(\ker^{\mathtt{h}}(p)_{|c})$ denote the Picard groupoid of monoidal autoequivalences of the fibre $\ker^{\mathtt{h}}(p)_{|c}$ of $\ker^{\mathtt{h}}(p)$ over $c$.
Note that we do not claim that the Picard groupoids $\Aut(\ker^{\mathtt{h}}(p)_{|c})$ assemble into a smooth 2-group (though it might be possible to achieve this).
We claim that there is a functor $\sfG_{|c} \to \Aut(\ker^{\mathtt{h}}(p)_{|c})$ which is canonical up to unique natural isomorphism.
 
 Let $(\,\boldsymbol\cdot\,)^\vee \colon \sfH_{|c} \to \sfH_{|c}$ denote a choice of functorial inverse in $\sfH_{|c}$.
 This can always be enhanced to a functorial choice of \emph{adjoint inverse}, i.e.~a functor $k \mapsto (k^\vee, \ev_k, \coev_k)$ that maps $k$ to a triple of a dual object $k^\vee$, and duality morphisms (which are isomorphisms in this case) $\ev_k \colon k \otimes_\sfH k^\vee \to \sfu_\sfH(c)$ and $\coev_k \colon \sfu_\sfH(c) \to k^\vee \otimes_\sfH k$ which satisfy the triangle identities.
 The functor $(\,\boldsymbol\cdot\,)^\vee$ acts on morphisms $\psi \colon k \to k'$ by taking the dual of $\psi^{-1}$ with respect to the chosen duality data on $k$ and $k^\vee$.
 This enhancement can be achieved by choosing an \emph{adjoint inverse} for the equivalence of categories $(\otimes_\sfH, \pr_1)_{|c} \colon \sfH_{|c} \times \sfH_{|c} \to \sfH_{|c} \times \sfH_{|c}$ (which is always possible).

To an object $k \in \sfH_{|c}$ we associate the functor
\begin{align}
	\Ad_k \colon \ker^{\mathtt{h}}(p)_{|c} &\to \ker^{\mathtt{h}}(p)_{|c} 
	\\
	(h,\varphi) &\mapsto (k \otimes_\sfH h \otimes_\sfH k^\vee, \varphi_k)
	\\
	\big( \zeta \colon (h_0, \varphi_0) \to (h_1, \varphi_1) \big)
	&\mapsto 1_k \otimes_\sfH \zeta \otimes_\sfH 1_{k^\vee} \ ,
\end{align}
where the morphism $\varphi_k$ is the composition
\begin{equation}
\begin{tikzcd}[row sep=1.25cm]
	p(k \otimes_\sfH h \otimes_\sfH k^\vee) \ar[rrr] \ar[d,
        dashed, swap,  "\varphi_k"]
	& & & p(k) \otimes_\sfG p(h) \otimes_\sfG p(k^\vee) \ar[d, "1_{p(k)} \otimes_\sfG \varphi \otimes_\sfG 1_{p(k^\vee)}"]
	\\
	\sfu_\sfG (c) 
	& p(k \otimes_\sfH k^\vee) \ar[l]
	& p(k) \otimes_\sfG p(k^\vee) \ar[l]
	& p(k) \otimes_\sfG \sfu_\sfG(c) \otimes_\sfG p(k^\vee) \ar[l]
\end{tikzcd}
\end{equation}
Given another object $k' \in \sfH_{|c}$ such that $p(k) = p(k')$ in $\sfG$, the principality condition implies that there exists an object $(b, \beta) \in \ker^{\mathtt{h}}(p)$ and an isomorphism $\psi \colon k' \to k \otimes_\sfH b$.
Since $\ker^{\mathtt{h}}(p)$ is abelian, this induces an isomorphism
\begin{equation}
	\alpha_{k,k'} \colon (1_k \times \ev_b \times 1_{k^\vee})
        \circ (\psi \times 1_h \times \psi^\vee) \colon \Ad_{k'}(h)
        \to \Ad_k(h) \ .
\end{equation}
By the functoriality of $(\,\boldsymbol\cdot\,)^\vee$, any other choice of $(b,\beta)$ and $\psi$ yields the same isomorphism $\Ad_{k'}(h) \to \Ad_k(h)$ in this way.
Furthermore, this isomorphism is natural in $k$ and $h$ by the functoriality of $\otimes_\sfH$ and $(\,\boldsymbol\cdot\,)^\vee$.
That is, the pair $(\Ad,\alpha)$ defines an object
\begin{equation}
	(\Ad,\alpha)_c \in \holim^{\Grpd} \ul{2\Grp} \big(
        \sfH_{|c}^{[\bullet]}, \Aut(\ker^{\mathtt{h}}(p)_{|c}) \big) \ ,
\end{equation}
where $\sfH_{|c}^{[\bullet]} \to \sfG_{|c}$ denotes the \v{C}ech nerve of the functor $p_{|c}$.
As we show in the proof of Proposition~\ref{st:effective epi from surj
  GrFib} in Appendix~\ref{app:Principal 2-bundles}, any choice of preimages of the objects $g \in \sfG_{|c}$ under $p_{|c}$ now induces a functor $\sfG_{|c} \to \Aut(\ker^{\mathtt{h}}(p)_{|c})$ from these data.
Moreover, any other choice of such preimages will induce a canonical natural isomorphism of functors.
Hence we obtain a well-defined isomorphism class of functors, which we denote by
\begin{equation}
	[\Ad,\alpha]_c \in \pi_0 \Big( \ul{2\Grp} \big( \sfG_{|c},
        \Aut(\ker^{\mathtt{h}}(p)_{|c}) \big) \Big) \ .
\end{equation}
This class allows us to state when a smooth 2-group extension is central.
\qen
\end{construction}

\begin{definition}
\label{def:central sm2Grp ext}
Let $(F,p)$ be an extension of a smooth 2-group $\sfG$ by a smooth 2-group $\sfA$.
Then $(F,p)$ is \emph{central} if $\sfA$ is abelian, and for every $c
\in \Cart$ the isomorphism class $[\Ad,\alpha]_c$ agrees with the
isomorphism class of the trivial 2-group morphism $\sfG_{|c} \to \Aut(\ker^{\mathtt{h}}(p)_{|c})$.
\end{definition}

\subsection{Global description of the 2-group extension}
\label{sec:Sym(Ga) global}

We shall now apply the general considerations of
Sections~\ref{sect:smooth2groups}
and~\ref{sect:principal2bundles} to the smooth groupoid
$\Sym_G(\Ga)$ constructed in Section~\ref{sect:smoothsymgerbes}.

\begin{theorem}
\label{Thm: Sym is a 2-group}
The functor $\pi \colon \Sym_G(\Ga) \to \Cart$ is a smooth 2-group.
\end{theorem}

\begin{proof}
First we show that $\Sym_G(\Ga)$ carries the structure of a monoid object in $\Hscr$.
The terminal object $*_\Hscr \in \Hscr$ is the identity functor $1_\Cart \colon \Cart \to \Cart$.
We start by defining the 1-morphism $\sfu \colon *_\Hscr \to
\Sym_G(\Ga)$; it is the functor that assigns to every object $c \in
\Cart$ the object $(e_c, \pr_M^*1_\Ga) \in \Sym_G(\Ga)$, where $e_c
\colon c \to G$ is the constant map at the identity object $e \in G$.

Next we define the 1-morphism 
$$
\otimes \colon \Sym_G(\Ga) \times_\Cart \Sym_G(\Ga) \to \Sym_G(\Ga)
$$
in the following way:
consider two arbitrary objects $(f_0,A_0), (f_1,A_1) \in \Sym_G(\Ga)$ in
the same fibre of $\pi \colon \Sym_G(\Ga) \to \Cart$, i.e.~$f_0,f_1$ are defined over the same object $c \in \Cart$.
We define the map $(1,\Phi_{f_0})$ as the composition
\begin{equation}
\begin{tikzcd}[column sep=1.25cm, row sep=1cm]
	c \times M \ar[d, dashed, "{(1,\Phi_{f_0})}"'] \ar[r, "{\mit\Delta}
        \times 1"] & c \times c \times M \ar[d, "1 \times f_0\times 1"]
	\\
	c \times M & c \times G \times M \ar[l, "1 \times \Phi"]
\end{tikzcd}
\end{equation}
Observe that
\begin{equation}
\label{eq:Phi_f props}
	\pr_M \circ (1,\Phi_{f_0}) = \Phi_{f_0} \ ,
	\quad
	\Phi_{f_1} \circ (1,\Phi_{f_0}) = \Phi_{f_1 \, f_0}
	\qandq
	(1,\Phi_{f_1}) \circ (1,\Phi_{f_0}) = (1, \Phi_{f_1 \, f_0}) \
        ,
\end{equation}
where the second and third identities use the fact that $\Phi$ is a group action.
Thus we can form the 1-morphism
\begin{equation}
\begin{tikzcd}[column sep=1.5cm, row sep=1.25cm]
	\pr_M^*\Ga \ar[r, "A_0"] \ar[d, dashed] & \Phi_{f_0}^*\Ga \ar[r, "\cong"] & (1,\Phi_{f_0})^* \pr_M^* \Ga \ar[d, "{(1,\Phi_{f_0})^*A_1}"]
	\\
	\Phi_{f_1 \, f_0}^*\Ga & & (1,\Phi_{f_0})^*\Phi_{f_1}^* \Ga \ar[ll, "\cong"]
\end{tikzcd}
\end{equation}
The solid unlabelled arrows are canonical isomorphisms that stem from
the fact that $\BGrb$ is a (pre)sheaf of 2-categories on the category
of manifolds $\Mfd$~\cite{Waldorf--More_morphisms,NS:Equivar}.
By a slight abuse of notation, we denote the composite morphism by $(1,\Phi_{f_0})^*A_1 \circ A_0$.
Then we set
\begin{equation}
\label{eq:2-grp str on Sym(G)}
	(f_1,A_1) \otimes (f_0,A_0)
	\coloneqq \big( f_1 \, f_0,\, (1,\Phi_{f_0})^*A_1 \circ A_0
        \big) \ ,
\end{equation}
and analogously on 2-isomorphisms.
The associator and unitors are readily obtained from those in the sheaf of 2-categories $\BGrb$.
The coherence conditions in $\BGrb$ imply that $\Sym_G(\Ga)$, endowed with the multiplication and coherence morphisms defined here, is a monoid object in $\Hscr$.

Now we show that $\Sym_G(\Ga)$ is in fact a group object in $\Hscr$.
Set
\begin{equation}
\label{eq:inverse on Sym(G)}
	(f,A)^{-1} \coloneqq \big( f^{-1},\, (1, \Phi_{f^{-1}})^* A^{-1} \big)
\end{equation}
and analogously on morphisms, where $f^{-1}$ denotes the composition
of the map $f \colon c \to G$ with the inversion map in the group $G$.
It follows from the properties~\eqref{eq:Phi_f props} of
$\Phi_{(\,\boldsymbol\cdot\,)}$ that this provides a functorial
(two-sided) inverse object with respect to the 1-morphism $\otimes$,
and hence shows that the morphism 
$$
(\otimes,\pr_1) \colon \Sym_G(\Ga) \times_\Cart \Sym_G(\Ga) \to
\Sym_G(\Ga) \times_\Cart \Sym_G(\Ga)
$$ 
is an equivalence in $\Hscr$ (where the product is taken in $\Hscr$).
Thus $\Sym_G(\Ga) \to \Cart$ is indeed a group object in $\Hscr$.
\end{proof}

\begin{theorem}
\label{st: 2-group extension}
There is a smooth 2-group extension
\begin{equation}
\label{eq:Sym(Ga)-extension}
\begin{tikzcd}
	1 \ar[r] & \HLBdl^M \ar[r, "\iota"] & \Sym_G(\Ga) \ar[r, "p"] &
        \ul{G} \ar[r] & 1 \ ,
\end{tikzcd}
\end{equation}
where we abbreviate $\ast_\Hscr=\Cart$ by $1$.
\end{theorem}

\begin{proof}
The projection functor $\pr \colon \ul{G} \to \Cart$ is a smooth
2-group via $\otimes_G \colon \ul{G} \times_\Cart \ul{G} \to \ul{G}$ defined by
\begin{equation}
(f_1 \colon c \to G) \otimes_G (f_0 \colon c \to G) = (f_1 \, f_0
\colon c \to G) \ .
\end{equation}
It is evident from~\eqref{eq:2-grp str on Sym(G)} that $p \colon \Sym_G(\Ga) \to \ul{G}$ is a morphism of smooth 2-groups.
It is a Grothendieck fibration in groupoids by construction, and it is surjective on objects since $G$ is connected (as we have argued at the beginning of this section).

Next we define the morphism $\iota \colon \HLBdl^M \to \Sym_G(\Ga)$ in $\Hscr$.
Over a Cartesian space $c\in \Cart$, it is simply the canonical inclusion
\begin{align}
	\HLBdl(c \times M) &\to \BGrb(c \times M)(\pr_M^*\Ga,
                             \Phi_{e_c}^*\Ga) = \BGrb(c \times
                             M)(\pr_M^*\Ga, \pr_M^*\Ga) \ .
\end{align}
Here $e_c \colon c \to G$ is the constant map at the unit element of $G$.
Since the inclusion of line bundles into morphisms of bundle gerbes strictly maps the tensor product to the composition~\cite{Waldorf--More_morphisms,Bunk--Thesis}, we readily find that $\iota$ is a morphism of smooth 2-groups.

To see that~\eqref{eq:Sym(Ga)-extension} is an extension of smooth 2-groups, we first show that the inclusion $\iota$ is an equivalence $\HLBdl^M \to \ker^{\mathtt{h}}(p)$.
By Lemma~\ref{st:GrFib Lemma} from Appendix~\ref{app:Principal
  2-bundles} and the fact that $p \colon \Sym_G(\Ga) \to \ul{G}$ is a
Grothendieck fibration in groupoids, it follows that the canonical inclusion $\ker(p) \hookrightarrow \ker^{\mathtt{h}}(p)$ is an equivalence.
Consequently, it suffices to show that $\iota$ induces an equivalence $\HLBdl^M \to \ker(p)$.
Over an object $c \in \Cart$, the fibre of $\ker(p)$ consists of the automorphism groupoid of $\pr_M^* \Ga \in \BGrb(c \times M)$.
It is well-known~\cite{Waldorf--More_morphisms} that the inclusion
$\HLBdl(c \times M) \to \BGrb(c \times M)$ given by $L \mapsto L \otimes 1_{\pr_M^*\Ga}$ is an equivalence of groupoids.

To see that the functor $p \colon \Sym_G(\Ga) \to \ul{G}$ is an $\HLBdl^M$-principal 2-bundle (see Definition~\ref{def:sfH-bundle}), it now suffices to show that the functor
\begin{align}
	(1, \alpha) \colon \Sym_G(\Ga) \times_\Cart \HLBdl^M
	&\to \Sym_G(\Ga)\times_{\ul{G}} \Sym_G(\Ga)
	\\
	\big( (f,A), L \big) &\mapsto \big( (f,A), (f, A \otimes L) \big)
\end{align}
is an equivalence in $\Hscr$, where we have written out the product in $\Hscr$ as the fibre product over $\Cart$.
Observe that by the equivalence $\HLBdl^M \to \ker^{\mathtt{h}}(p)$,
it is enough to consider the action of $\HLBdl^M$, and since $\ul{G}$
has discrete fibres, i.e.~the fibres have no non-identity morphisms,
there is an identity $\Sym_G(\Ga)\times_{\ul{G}}^{\mathtt{h}} \Sym_G(\Ga) = \Sym_G(\Ga)\times_{\ul{G}} \Sym_G(\Ga)$, and hence we can work with the strict pullback instead of the homotopy pullback.

Since both sides are fibred over $\Cart$, it suffices to show that this functor is an equivalence on all fibres~\cite[Proposition~3.36]{Vistoli:Fib_Cats}.
Thus we fix an object $c \in \Cart$.
We need to check that the functor
\begin{align}
\label{eq:torsor fctr, fixed c}
	(1, \alpha)_{|c} \colon \Sym_G(\Ga)_{|c} \times \HLBdl(c \times M)
	&\to \Sym_G(\Ga)_{|c} \times_{\ul{G}_{\,|c}} \Sym_G(\Ga)_{|c}
	\\
	\big( (f,A), L \big) &\mapsto \big( (f,A), (f, A \otimes L) \big)
\end{align}
is an equivalence.
By construction, both sides decompose into coproducts
\begin{equation}
	\Sym_G(\Ga)_{|c} \times \HLBdl(c \times M) = \coprod_{f \colon
                                                   c \to G}\, \BGrb(c
                                                   \times
                                                   M)(\pr_M^*\Ga,
                                                   \Phi_f^*\Ga) \times
                                                   \HLBdl(c \times M)
\end{equation}
and
\begin{equation}
	\Sym_G(\Ga)_{|c} \times_{\ul{G}_{\,|c}} \Sym_G(\Ga)_{|c} =
                                                               \coprod_{f
                                                               \colon
                                                               c \to
                                                               G}\,
                                                               \BGrb(c
                                                               \times
                                                               M)(\pr_M^*\Ga,
                                                               \Phi_f^*\Ga)
                                                               \times
                                                               \BGrb(c
                                                               \times
                                                               M)(\pr_M^*\Ga,
                                                               \Phi_f^*\Ga)
                                                               \ ,
\end{equation}
so the functor $(1, \alpha)_{|c}$ decomposes into functors
\begin{align}
	(1, \alpha)_{|f} \colon \BGrb(c \times M)&(\pr_M^*\Ga, \Phi_f^*\Ga) \times \HLBdl(c \times M)
	\\
	&\to \BGrb(c \times M)(\pr_M^*\Ga, \Phi_f^*\Ga) \times \BGrb(c
          \times M)(\pr_M^*\Ga, \Phi_f^*\Ga) \ .
\end{align}
This functor acts as the identity on the first factor and as the standard action of line bundles on isomorphisms of bundle gerbes in the second factor.
Thus $(1, \alpha)_{|f}$ is an equivalence since on any manifold $X$, the category of 1-isomorphisms between any given bundle gerbes is a torsor category over $\HLBdl(X)$ with respect to this action~\cite{Waldorf--More_morphisms}.
\end{proof}

\begin{proposition}
If $G$ acts non-trivially on $M$, then the extension~\eqref{eq:Sym(Ga)-extension} is not central.
\end{proposition}

\begin{proof}
This follows readily from the explicit forms~\eqref{eq:2-grp str on Sym(G)} and~\eqref{eq:inverse on Sym(G)} of the product and the inverse in $\Sym_G(\Ga)$, together with the fact that composition of morphisms of bundle gerbes is compatible with tensoring by line bundles.
Explicitly, given $(f,A) \in \Sym_G(\Ga)_{|c}$ and $L \in \HLBdl(c\times
M)$ we find 
\begin{equation}
	(f,A) \otimes \iota(L) \otimes (f,A)^{-1} \cong (e_c,
        \Phi_f^*L) = \iota(\Phi_f^*L) \ .
\end{equation}
Hence $(\Ad, \alpha)_c(f)(L) \cong \Phi_f^* L$.
Observe that since $\ul{G}$ has discrete fibres, we have $\ker^{\mathtt{h}}(p) = \ker(p)$, and by the equivalence \smash{$\HLBdl^M \to \ker(p)$} it is sufficient to consider the adjoint action on the smooth 2-group $\HLBdl^M$ here.

Let $c = *$, so that the data $f$ corresponds to an element $g \in G$.
Assume that the extension~\eqref{eq:Sym(Ga)-extension} is central.
Then, by Construction~\ref{cstr: fun G -> Aut} and Definition~\ref{def:central sm2Grp ext}, there is an isomorphism $\varphi \colon (\Ad, \alpha)_{*} \to 1_{\HLBdl(M)}$ of morphisms of 2-groups $G \to \Aut(\HLBdl(M))$.
Let $I \in \HLBdl(M)$ denote the trivial line bundle, and let $\psi \colon I \to I$ be any automorphism, i.e.~any $\U(1)$-valued function on $M$.
The naturality of $\varphi$ then implies, in particular, that the diagram
\begin{equation}
\begin{tikzcd}[column sep=1.5cm, row sep=1cm]
	g^*I = I \ar[r, "g^*\psi"] \ar[d, "\varphi_{I,g}"'] & I = g^*I \ar[d, "\varphi_{I,g}"]
	\\
	I \ar[r, "\psi"'] & I
\end{tikzcd}
\end{equation}
commutes.
But this is equivalent to the identity $\psi = g^*\psi$ for any $g \in
G$ and $\psi \colon M \to \U(1)$, which is a contradiction if the $G$-action on $M$ is non-trivial.
\end{proof}

We now obtain an action of $\Sym_G(\Ga)$ on $\Ga$ in the following sense:
let $\widehat{\Ga} \in \Hscr$ denote the category fibred in groupoids
over $\Cart$ which is defined as follows. Consider the presheaf of groupoids on $\ul{M}$ that assigns to $f \colon c \to M$ the category $\BGrb(c)(\Ia, f^*\Ga)$.
Then $q:\widehat{\Ga}\to\ul{M}$ is obtained by applying the Grothendieck construction to this presheaf.
The action of $\Sym_G(\Ga)$ on $\Ga$ is then the morphism of categories
over $\ul{G \times M} \cong \ul{G} \times \ul{M}$ obtained through the diagram
\begin{equation}
\begin{tikzcd}[column sep={2cm,between origins}, row sep=1.25cm]
	\Sym_G(\Ga) \times\Phi^*\widehat{\Ga} \ar[dr, "p \times q"']
        \ar[rr, dashed, "\widehat{\Phi}"] & & \pr_M^*\widehat{\Ga} \ar[dl,
        "1\times q"] & & \widehat{\Ga} \ar[d,"q"]
	\\
	& \ul{G} \times \ul{M} \ar[rrr, shift left=0.1cm, "{\pr_M}"] \ar[rrr, shift left=-0.1cm, "\Phi", swap] & & & \ul{M}
\end{tikzcd}
\end{equation}
where we suppress pullbacks and denote by $\pr_M , \Phi\colon  \ul{G} \times \ul{M} \longrightarrow
 \ul{M}  $ the functors induced from the smooth maps $\pr_M, \Phi$ via postcomposition.
 The functor $\widehat{\Phi}$ sends an object $(\Aa,\Ja) \in
 \Sym_G(\Ga)_{|c} \times \widehat{\Ga}_{|f}$ to the composition
 ${\mit\Delta}^*(1_c \times f)^*(\Aa , \Ja)$, where ${\mit\Delta} \colon c \to c
 \times c$ is the diagonal map.

The construction $\Sym_G$ is 2-functorial: let $E \colon \mathcal{G}\longrightarrow \mathcal{G}'$ be a 1-isomorphism of 
bundle gerbes. Pick an adjoint inverse $E^{\vee}$.
The 1-isomorphism $E$ induces a 1-isomorphism of smooth 2-groups
\begin{align}
\widehat{E} \colon \Sym_G(\mathcal{G}) & \longrightarrow \Sym_G(\mathcal{G}') \\
\big(f\colon c \longrightarrow G ,A \colon
  \pr_M^*\mathcal{G}\longrightarrow \Phi_f^*\mathcal{G}\big)
                                     &\longmapsto \Big(f,
                                       \pr_M^*\mathcal{G'}
                                       \xrightarrow{\pr_M^*E^{\vee}}\pr_M^*\mathcal{G}\overset{A}{\longrightarrow}
                                       \Phi_f^* \mathcal{G}
                                       \xrightarrow{\Phi_f^* E}
                                       \Phi_f^* \mathcal{G}'\Big) \ .
\end{align}
Let $E, E' \colon \Ga \longrightarrow \Ga'$ be 1-isomorphisms
and $\eta\colon E \longrightarrow E'$ a 2-isomorphism. We 
construct a smooth 2-isomorphism $\widehat{\eta}\colon \widehat{E}
\longrightarrow \widehat{E'}$ 
whose component at an object $(f,A)$ of $\Sym_G(\Ga)$ is given by
\begin{equation}
\begin{tikzcd}[column sep=1.6cm]
\pr_M^* \Ga' \ar[r,"\pr_M^*E^{\vee}", bend left=40,""name=LL1]
\ar[r,"\pr_M^*{E'}^{\vee}", bend right=40,swap,""name=LL2] & \ \ \pr_M^* \Ga \ar[r,"A"] & \Phi_f^* \Ga \ar[r,"\Phi_f^*E", bend left=40,""name=LL1'] \ar[r,"\Phi_f^*E'", bend right=40,swap,""name=LL2'] & \Phi_f^* \Ga'
\ar[from=LL1, to=LL2,Rightarrow, "\pr_M^*\!\eta^\vee",shorten <= 3pt, shorten >= 3pt]
\ar[from=LL1', to=LL2',Rightarrow, "\Phi_f^*\eta",shorten <= 3pt, shorten >= 3pt] 
\end{tikzcd} 
\end{equation}

\subsection{Descent description of the 2-group extension}
\label{Sec:Desc description of Sym}

We can describe the smooth 2-group $\Sym_G(\Ga)$ in a way analogously to Section~\ref{sect:non-triv fibs via connections}; that is, we can construct $\Sym_G(\Ga)$ in terms of descent data for the path fibration $P_0 G \to G$ and the parallel transport on $\Ga$ introduced in Section~\ref{Sec: parallel transport}.
However, for a bundle gerbe $\Ga$ this construction is more involved
compared to the case of a principal bundle $P$.
In particular, we need to replace the associated bundle construction $(P_0 G \times \Gau(P))/{\sim}$ of $\La_G$ (cf.~Section~\ref{sect:non-triv fibs via connections}) by a homotopy-coherent version.
Once established, the descent presentation of $\Sym_G(\Ga)$ allows us to study and compute this smooth 2-group very explicitly in certain situations, as we demonstrate in Section~\ref{Sec: Magnetic translations}.

Recalling the notation of Section~\ref{sect:non-triv fibs via
  connections}, let $\Ga \in \BGrb^\nabla(M)$ be a bundle gerbe with
connection on $M$.
Using the smooth map~\eqref{EQ: Descent description line bundle} we obtain a diffeological hermitean line bundle
\begin{equation}
\label{eq:L^dashv}
	\sfL^\dashv \coloneqq (L\Phi)^*(\Ta \Ga) \to (P_0 G)^{[2]}
        \times M \ .
\end{equation}
This object is completely analogous to~\eqref{EQ: Descent description
  line bundle} when one views the holonomy of a line bundle $L$
on $M$ as the transgression of $L$ to the loop space $LM$, and subsequently the
transgression line bundle $\Ta\Ga$ as the holonomy of the bundle gerbe $\Ga$ on $M$ (cf.~Section~\ref{sec: Transgression and PT}).
In the adjoint picture, we can interpret $\sfL^\dashv$ as a smooth
assignment of a line bundle with connection $\sfL_{(\gamma,\alpha)}$
on $M$ to each pair of based paths $(\gamma, \alpha) \in (P_0G)^{[2]}$.

Consider the simplicial diffeological space $(P_0 G)^{[\bullet]}
\times M$ with face maps 
$$
d_i:(P_0 G)^{[n]}
\times M\to (P_0 G)^{[n-1]}
\times M
$$ 
for $0\leq i\leq n-1$ defined by deleting the $i$-th entry of
$(P_0G)^{[n]}$.
The fusion product $\lambda$ on the transgression line bundle $\Ta\Ga$ over the
loop space $LM$ induces an isomorphism
\begin{equation}
\label{eq:L^dashv coherence iso}
	(L\Phi)^*\lambda \colon d_0^*\sfL^\dashv \otimes d_2^*\sfL^\dashv \to d_1^*\sfL^\dashv
\end{equation}
of hermitean line bundles over $(P_0 G)^{[3]} \times M$, which is coherent over $(P_0 G)^{[4]} \times M$.

\begin{remark}
In an adjoint fashion, the hermitean line bundle $\sfL^\dashv$ on $(P_0 G)^{[2]} \times M$ from~\eqref{eq:L^dashv} can be seen as a morphism of smooth groupoids $\sfL \colon (P_0 G)^{[2]} \to \HLBdl^M$ (under the equivalence of diffeological vector bundles and morphisms to $\HLBdl$ following from~\cite[Thm.~4.8]{Bunk:2020ifw}).
In this picture, the coherent isomorphism $(L\Phi)^*\lambda$ from~\eqref{eq:L^dashv coherence iso} corresponds to a coherent isomorphism $d_0^* \sfL \otimes d_2^* \sfL \to d_1^*\sfL$ of morphisms $(P_0 G)^{[3]} \to \HLBdl^M$.
In this sense, $(\sfL^\dashv, (L\Phi)^*\lambda)$ represent a smooth $\HLBdl^M$-valued \v{C}ech 1-cocycle on $G$ with respect to the \v{C}ech nerve of the subduction $P_0 G \to G$.
Note that this nicely fits the formalism for higher principal bundles with not-necessarily strict structure groups---such as $\HLBdl^M$---from~\cite{NSS:oo-bdls_I}.
\qen
\end{remark}

\begin{definition}
Let $c\in\Cart$ be a Cartesian space and $f \colon c \to G$ a smooth
map. We define a category $\Des^\PSh_\sfL(f)$ with
\begin{myitemize}
\item $\ul{\rm Objects:}$ pairs $(J, \jmath)$, where $J \in
  \HLBdl(f^*P_0 G \times M)$ and where $\jmath$ is an isomorphism of hermitean
  line bundles
\begin{equation}
	\jmath \colon d_1^*J \to \big(\widehat{f}^{\,[2]} \times
        1\big)^*\sfL^\dashv \otimes d_0^*J 
\end{equation}
over $(f^*P_0G)^{[2]}\times M$, where $\widehat{f}^{\,[n]} \colon (f^*P_0 G)^{[n]} \to P_0 G^{[n]}$ is the canonical map induced by the pullback of the subduction $P_0 G \to G$ along $f$.
These data are required to satisfy the compatibility condition that
\begin{equation}
\label{eq:obs in Des_L and fusion}
\begin{tikzcd}[row sep=1cm]
	d_2^*d_1^*J \ar[r, "d_2^*\jmath"] \ar[dd, "d_1^*\jmath"']
	& d_2^* \big( (\widehat{f}^{\,[2]} \times 1)^*\sfL^\dashv \otimes d_0^*J \big) \ar[d, "1 \otimes d_0^*\jmath"]
	\\
	& d_2^*\big(\widehat{f}^{\,[2]} \times 1\big)^*\sfL^\dashv \otimes d_0^*\big(\widehat{f}^{\,[2]} \times 1\big)^*\sfL^\dashv \otimes d_1^*d_0^*J \ar[d, "\widehat{f}^{\,[3]*} \lambda \otimes 1"]
	\\
	d_1^*\big(\widehat{f}^{\,[2]} \times 1\big)^*\sfL^\dashv \otimes d_2^*d_0^*J \ar[r, equal]
	& d_1^*\big(\widehat{f}^{\,[2]} \times 1\big)^*\sfL^\dashv \otimes d_1^*d_0^*J
\end{tikzcd}
\end{equation}
is a commutative diagram in $\HLBdl((f^*P_0 G)^{[3]} \times M)$, where
we use the simplicial relations.

\item $\ul{\rm Morphisms:}$ a morphism $(J,\jmath) \to (J',\jmath')$ is an isomorphism $\psi \colon J \to J'$ such that
\begin{equation}
\label{eq:mps in Des_L and fusion}
\begin{tikzcd}[column sep=1.25cm, row sep=1cm]
	d_1^*J \ar[r, "\jmath"] \ar[d, "d_1^*\psi"']
	& \big(\widehat{f}^{\,[2]} \times 1\big)^*\sfL^\dashv \otimes d_0^*J \ar[d, "1 \otimes d_0^*\psi"]
	\\
	d_1^*J' \ar[r, "\jmath'"']
	& \big(\widehat{f}^{\,[2]} \times 1\big)^*\sfL^\dashv \otimes d_0^*J'
\end{tikzcd}
\end{equation}
is a commutative diagram in $\HLBdl((f^*P_0 G)^{[2]} \times M)$.
\end{myitemize}
\end{definition}

Pullbacks of morphisms of bundle gerbes turns the assignment $(f \colon c \to G) \mapsto \Des^\PSh_\sfL(f)$ into a presheaf of groupoids on $\ul{G}\,$.
(This is actually even a sheaf of groupoids, but we will not need this fact here.)
Applying the Grothendieck construction, we obtain a category fibred in
groupoids over $\ul{G}\,$, $p_\sfL:\Des_\sfL\to\ul{G}\,$, and composing with the canonical projection functor $\pr \colon \ul{G} \to \Cart$ we obtain a category fibred in groupoids over $\Cart$
\begin{equation}
\begin{tikzcd}[row sep=.75cm, column sep=0.25cm]
	\Des_\sfL \ar[rr, "p_\sfL"] \ar[dr, dashed, "\pi_\sfL"'] & & \ul{G} \ar[dl, "\pr"]
	\\
	& \Cart &
\end{tikzcd}
\end{equation}
which defines a smooth groupoid $\Des_\sfL$.

\begin{proposition}
\label{st:group str on Des_L}
The functor $\pi_\sfL:\Des_\sfL\to\Cart$ is a smooth 2-group.
\end{proposition}

\begin{proof}
Let $(f_0,J_0,\jmath_0), (f_1,J_1,\jmath_1) \in \Des_\sfL$ be two objects, where $(J_i,\jmath_i)$ lies in the fibre over a smooth map $f_i \colon c \to G$ for $i = 0,1$.
The product $(f_1,J_1,\jmath_1) \otimes (f_0,J_0,\jmath_0)$ is defined as follows.
First, observe that it should lie in the fibre of $\Des_\sfL$ over the
pointwise product map $f_1 \, f_0 \colon c \to G$, $u \mapsto f_1(u) \, f_0(u)$.
Consider the smooth map
\begin{align}
	F \colon m^*P_0G \times^{\phantom{\dag}}_{G\times G} \pr_1^*P_0G
  \times^{\phantom{\dag}}_{G\times G} \pr_2^*P_0G &\to P_{\partial \Delta^2} G
	\\
	(\gamma_{10},\gamma_1,\gamma_0) &\mapsto F(\gamma_{10},\gamma_1,\gamma_0)
	= \big( \gamma_{10}, \gamma_1 \, \gamma_0(1),\, \gamma_0 \big)
                                           \ ,
\end{align}
where $m \colon G\times G\to G$ is the multiplication of
$G$, and $\pr_1$ and $\pr_2$ are the projections to the first and
second factors of $G\times G$.
Let us denote by $\Phi_{\ev_1}$ the composition $P_0G \times M \to G
\times M \to M$, where the first map evaluates a based path at its end
point and the second map is the action $\Phi$ of $G$ on $M$.
The pair $(f_1,f_0)$ defines a map $c \to G\times G$.
Let $\sfs \colon P_{\partial \Delta^2}G \to LG$ be the map defined in~\eqref{eq:sfs}, and let (by a slight abuse of notation) $L\Phi \colon LG \times M \to LM$ denote the map $(\gamma, x) \mapsto \gamma_x$, with $\gamma_x(t) = \Phi_{\gamma(t)}(x)$.
Consider the hermitean line bundle
\begin{equation}
\label{eq:otimes_Des}
	K \coloneqq F^*\sfs^*L\Phi^*\Ta\Ga \otimes (1_{P_0 G} \times \Phi_{\ev_1})^*J_1 \otimes J_0
\end{equation}
on the diffeological space 
$$
Y_{f_1,f_0} \coloneqq \big((f_1 \, f_0)^*P_0G \times_c f_1^*P_0G
\times_c f_0^*P_0G\big) \times M \ .
$$
We claim that the bundle $K$ descends along the projection $p_1 \colon Y_{f_1,f_0} \to (f_1 \, f_0)^*P_0G \times M$.
The descended bundle is the hermitean line bundle underlying the product $(f_1,J_1,\jmath_1) \otimes (f_0,J_0, \jmath_0)$.

We thus endow the bundle $K$ with an isomorphism $\kappa \colon
(p_1)_0^*K \to (p_1)_1^*K$ over $Y_{f_1,f_0}^{[2]}$, which is required
to satisfy a cocycle relation over \smash{$Y_{f_1,f_0}^{[3]}$}.
An element of \smash{$Y_{f_1,f_0}^{[2]}$} can be identified with a pair of triples $((\gamma_{10},\gamma_1,\gamma_0), (\gamma_{10}, \gamma'_1, \gamma'_0))$, where $(\gamma_i, \gamma'_i) \in (P_0G)^{[2]}$ for $i = 0,1$.
We define the isomorphism $\kappa$ as the composition of the fusion product $\lambda$ on $\Ta\Ga$ with $(1 \times \Phi_{\ev_1})^*\jmath_1 \otimes \jmath_0$.
Then the cocycle condition simply follows from the compatibility condition~\eqref{eq:obs in Des_L and fusion} and the associativity of the fusion product.
(We also need to use thin reparameterisations, but these are
implemented in a completely coherent way by the thin-homotopy invariant connection on $\Ta\Ga$.)

Thus we obtain a descended hermitean line bundle $\Des(K,\kappa)$ on
$(f_1 \, f_0)^*P_0G \times M$ (for descent properties of diffeological
vector bundles, see~\cite{Bunk:2020ifw}).
Applying the fusion product in the first tensor factor of $K$, we obtain an isomorphism which (by the associativity of $\lambda$) descends to an isomorphism
\begin{equation}
	\jmath_K \colon d_1^*\Des(K,\kappa) \to \big(\,\widehat{f_1 \, f_0}{}^{[2]} \times 1 \big)^*\sfL^\dashv \otimes d_0^*\Des(K,\kappa)
\end{equation}
over $((f_1 \, f_0)^*P_0G)^{[2]} \times M$.
Again by the associativity of $\lambda$ and thin-homotopy invariance, the pair $(\Des(K,\kappa),\jmath_K)$ satisfies the relation~\eqref{eq:obs in Des_L and fusion}, and hence it makes sense to set
\begin{equation}
	(f_1,J_1,\jmath_1) \otimes (f_0,J_0,\jmath_0) \coloneqq \big(f_1\,f_0,
        \Des(K,\kappa),\jmath_K \big) \ .
\end{equation}
The action of the product $\otimes$ in $\Des_\sfL$ on morphisms simply sends $(\psi_1,\psi_0)$ to the descent along $p_1$ of the isomorphism $1_{\Ta\Ga} \otimes (1 \times \Phi_{\ev_1})^*\psi_1 \otimes \psi_0$.
The unitors of $\otimes$ are readily obtained from the construction, and the associator is defined from the fusion product; its coherence is yet another application of the associativity of $\lambda$ and the superficiality of the parallel transport on $\Ta\Ga$.
Inverses are constructed analogously to~\eqref{eq:inverse on Sym(G)}.
Finally, all constructions are compatible with pullbacks along maps $\varphi \colon c' \to c$ of Cartesian spaces, so that we obtain the structure of a smooth 2-group on $\Des_\sfL$.
\end{proof}

\begin{theorem}
\label{st:Sym = Des}
There is a weakly commutative diagram of smooth 2-groups
\begin{equation}
\label{eq:2grp ext iso}
\begin{tikzcd}[row sep=1cm]
	1 \ar[r] & \HLBdl^M \ar[r, "\iota"] \ar[d, "1"] & \Sym_G(\Ga) \ar[r, "p"] \ar[d, "\Psi"] & \ul{G} \ar[r] \ar[d, "1"] & 1
	\\
	1 \ar[r] & \HLBdl^M \ar[r, "\iota_\sfL"] & \Des_\sfL \ar[r, "p_\sfL"] & \ul{G} \ar[r] & 1
\end{tikzcd}
\end{equation}
where the functor $\Psi$ is an equivalence.
\end{theorem}

\begin{proof}
By the functoriality of $\Ga \mapsto \Sym_G(\Ga)$ (see Section~\ref{sec:Sym(Ga) global}) we can assume that we are in the case where $\Ga' = \Ra\Ta (\Ga)$ is the regression of a transgression, so that we have direct access to our construction of a parallel transport on $\Ga'$ from Section~\ref{sec: Transgression and PT}.
We start by constructing the functor $\Psi$.
For this, we construct a diagram in the 2-category $\Hscr$ of the form
\begin{equation}
\label{eq:Psi diag}
\begin{tikzcd}[column sep=1.5cm]
	\Sym_G(\Ga') \ar[r, "\varpi^*"]
	& \Des \big( \Sym_G(\Ga') \big) \ar[rr, shift left=0.1cm, "{\Hom_1(\pt_1^{\Ga'},\,\boldsymbol\cdot\,)}"]
	& & \Des_\sfL \ar[ll, shift left=0.1cm, "{(\,\boldsymbol\cdot\,) \otimes \pt_1^{\Ga'}}"]
\end{tikzcd}
\end{equation}
and the functor $\Psi$ is the composition from left to right.
Each of  the functors in~\eqref{eq:Psi diag} is an equivalence of
categories fibred in groupoids over $\Cart$, and hence so is $\Psi$.

For a smooth map $f \colon c \to G$, let $\varpi_f \colon f^*P_0 G \times M \to c \times M$ denote the canonical projection.
First we define the category $\Des(\Sym_G(\Ga'))$.
It is obtained via the Grothendieck construction applied to the
presheaf $\Des^\PSh(\Sym_G(\Ga'))$ of groupoids on $\ul{G}\,$, which
assigns to a smooth map $f \colon c \to G$ the groupoid $\Des^\PSh(\Sym_G(\Ga'))(f)$ where
\begin{myitemize}
\item objects are pairs $(A,\alpha)$ of a 1-isomorphism $A\colon
  \varpi_f^*\Ga' \to \varpi_f^*\Phi_f^*\Ga'$ over $f^*P_0G \times M$
  and a 2-isomorphism $\alpha \colon d_1^*A\to d_0^*A$ over
  $(f^*P_0 G)^{[2]} \times M$, which is coherent over $(f^*P_0 G)^{[3]} \times M$, and

\item morphisms $(A, \alpha) \to (A', \alpha')$ are given by
  2-isomorphisms $\psi \colon A \to A'$ satisfying $\alpha' \circ d_1^*\psi = d_0^*\psi \circ \alpha$.
\end{myitemize}
The functor $\varpi^*$ simply pulls back 1-morphisms \smash{$A \colon \pr_M^*\Ga' \to \Phi_f^*\Ga'$} along the subductions $\varpi_f$.
This functor is an equivalence since morphisms of bundle gerbes satisfy descent%
\footnote{In~\cite{Bunk--Thesis} the descent property was proven along surjective submersions of manifolds, but the proof directly carries over to subductions of diffeological spaces.}%
~\cite[Theorem~A.19]{Bunk--Thesis}.

Next we introduce some notation:
we define the map
\begin{align}
	P_0 \Phi \colon P_0 G \times M &\to PM
	\\
	(\gamma,x)&\mapsto \gamma_x \ ,
\end{align}
where
\begin{align}
\gamma_x(t) = \Phi_{\gamma(t)}(x) =: P_0 \Phi (\gamma,x)(t)
\end{align}
for all $t\in[0,1]$.
Observe that
\begin{equation}
	\ev_0 \circ P_0 \Phi = \pr_M
	\qquad \mbox{and} \qquad
	\ev_1 \circ P_0 \Phi = \Phi \circ (\ev_1 \times 1_M) \ .
\end{equation}
Thus the pullback of the parallel transport 1-isomorphism
\eqref{eq:pt1Ga'} by the map $P_0\Phi$ is a morphism
\begin{equation}
	(P_0 \Phi)^* \pt^{\Ga'}_1 \colon \pr_M^*\Ga' \to (\ev_1 \times 1_M)^* \Phi^*\Ga'
\end{equation}
in $\BGrb(P_0G\times M)$.
Given a smooth map $f \colon c \to G$, we obtain a smooth map
\begin{equation}
	P_0 \Phi \circ\big(\widehat{f} \times 1_M\big) \colon f^*P_0 G
        \times M \to PM \ .
\end{equation}
It satisfies
\begin{equation}
	\ev_0 \circ P_0 \Phi \circ \big(\widehat{f} \times 1_M\big) = \pr_M
	\qquad \mbox{and} \qquad
	\ev_1 \circ P_0 \Phi \circ \big(\widehat{f} \times 1_M\big) =
        \Phi_f \circ \varpi_f \ ,
\end{equation}
where $\varpi_f \colon f^*P_0 G \times M \to c \times M$ is the projection.
Hence we obtain a morphism
\begin{equation}
	\pt^{\Ga'}_f \coloneqq \big( P_0 \Phi \circ (\widehat{f} \times 1) \big)^* \pt^{\Ga'}_1
	\colon \varpi_f^* \pr_M^*\Ga' \to \varpi_f^* \Phi_f^*\Ga' \ ,
\end{equation}
which is defined over $f^*P_0 G \times M$.
By Proposition~\ref{st:iso L cong pt_loop} there is a canonical 2-isomorphism
\begin{equation}
	\big(d_0^* \pt^{\Ga'}_f\big)^{-1} \circ d_1^* \pt^{\Ga'}_f
        \xrightarrow{ \ \cong \ } \big( L\Phi \circ (\tau \times 1_M) \circ (\widehat{f}^{\,[2]} \times 1_M) \big)^*\Ta\Ga
	\cong \big(\widehat{f}^{\,[2]} \times 1_M)^* \sfL^{\dashv\, \vee}
\end{equation}
of 1-automorphisms of the pullback of $\pr_M^*\Ga'$ to $(f^*P_0 G)^{[2]} \times M$, where $\tau(\gamma_0, \gamma_1) = (\gamma_1, \gamma_0)$.
Equivalently, this is a 2-isomorphism
\begin{equation}
\label{eq:pullback of gerbe hol}
	\beta_f \colon d_0^* \pt^{\Ga'}_f \xrightarrow{ \ \cong \ } \big(\widehat{f}^{\,[2]} \times 1_M\big)^* \sfL^\dashv \otimes d_1^* \pt^{\Ga'}_f
\end{equation}
of 1-isomorphisms $\varpi_f^* \pr_M^* \Ga' \to \varpi_f^*\Phi_f^*\Ga'$ over $(f^*P_0 G)^{[2]} \times M$.

Now we come to the definition of the functor $(\,\boldsymbol\cdot\,) \otimes \pt_1^{\Ga'}$.
Given an object $(f,J,\jmath) \in \Des_\sfL$, define a morphism of bundle gerbes over $f^*P_0 G \times M$ via
\begin{equation}
	J \otimes \pt^{\Ga'}_f \colon \varpi_f^* \pr_M^*\Ga' \to
        \varpi_f^* \Phi_f^*\Ga' \ .
\end{equation}
Using the 2-isomorphism~\eqref{eq:pullback of gerbe hol}, we obtain a 2-isomorphism
\begin{equation}
\begin{tikzcd}[row sep=1cm, column sep=2cm]
	d_1^* J \otimes d_1^* \pt^{\Ga'}_f \ar[r, "\jmath \otimes 1"] \ar[dr, dashed, "\widehat{\jmath}"']
	& d_0^* J \otimes \big(\widehat{f}^{\,[2]} \times 1_M\big)^*\sfL^\dashv \otimes d_1^*\pt^{\Ga'}_f \ar[d,"1\otimes\beta_f^{-1}"]
	\\
	& d_0^* J \otimes d_0^* \pt^{\Ga'}_f
\end{tikzcd}
\end{equation}
over $(f^*P_0 G) ^{[2]} \times M$.
By construction, the 2-isomorphism $\widehat{\jmath}\,$ is coherent over $(f^*P_0 G)^{[3]} \times M$, and thus the pair $(J \otimes \pt^{\Ga'}_f, \widehat{\jmath}\,)$ defines a descent datum (with respect to the subduction $f^*P_0 G \times M \to c \times M$) for a 1-isomorphism of bundle gerbes $\pr_M^* \Ga' \to \Phi_f^*\Ga'$.
Analogously, morphisms in $\Des_\sfL$ give rise to morphisms of descent data as constructed above.
This defines the functor
\begin{equation}
	(\,\boldsymbol\cdot\,) \otimes \pt^{\Ga'}_1 \colon \Des_\sfL \to \Des \big(
        \Sym_G(\Ga') \big) \ .
\end{equation}
This is a functor of categories fibred in groupoids over $\Cart$ by the compatibility of pullbacks of bundles and their morphisms with the tensor product.

Finally, we introduce an inverse functor $\Hom_1\big(\pt^{\Ga'}_1,\,\boldsymbol\cdot\,\big)$ for $(\,\boldsymbol\cdot\,) \otimes \pt^{\Ga'}_1$.
An object $(f, A, \alpha) \in \Des(\Sym_G(\Ga'))$ consists, in particular, of a 1-isomorphism \smash{$A \colon \varpi_f^* \pr_M^*\Ga' \to \varpi_f^* \Phi_f^* \Ga'$} of bundle gerbes over $f^*P_0 G \times M$.
Another such morphism is given by $\pt^{\Ga'}_f$.
We can hence use the internal hom-functor $\Hom_1$ in the 2-category $\BGrb(f^*P_0G \times M)$ (see~\cite[Section~3.2]{Bunk--Thesis} and also~\cite[Section~2.1]{BW:Transgression_of_D-branes}) to produce a hermitean line bundle
\begin{equation}
	\Hom_1\big(\pt^{\Ga'}_f, A\big) \in \HLBdl(f^*P_0 G \times
        M) \ .
\end{equation}
This comes with an isomorphism over $(f^*P_0 G)^{[2]} \times M$ defined
by the diagram
\begin{equation}
\begin{tikzcd}[column sep=1.5cm, row sep=1cm]
	d_1^*\Hom_1\big(\pt^{\Ga'}_f, A\big) \ar[r] \ar[d, dashed]
	& \Hom_1\big(d_1^*\pt^{\Ga'}_f, d_1^*A\big) \ar[d, "{\Hom_1(\beta_f^{-1}, \alpha)}"]
	\\
	\big(\widehat{f}{}^{\,[2]}\times 1_M\big)^*\sfL^\dashv \otimes d^*_0\Hom_1\big(\pt^{\Ga'}_f, A\big)
	& \Hom_1\big((\widehat{f}{}^{\,[2]}\times 1_M)^*\sfL^{\dashv\, \vee} \otimes d_0^*\pt^{\Ga'}_f, d_0^*A\big) \ar[l]
\end{tikzcd}
\end{equation}
where the 2-isomorphism $\beta_f$ stems from~\eqref{eq:pullback of gerbe hol} and where the lower horizontal isomorphism of line bundles stems from the (categorified) linearity of $\Hom_1$~\cite[Theorem~3.63]{Bunk--Thesis}.
It follows from the properties of $\alpha$ and $\beta_f$ that $\Hom_1\big(\pt^{\Ga'}_f, A\big)$ defines an object in $(\Des_\sfL)_{|f}$.
By mapping a morphism $\psi$ in $\Des\big(\Sym_G(\Ga')\big)_{|f}$ to $\Hom_1\big(\pt^{\Ga'}_f, \psi\big)$, we obtain a functor
\begin{equation}
	\Hom_1\big(\pt_1^{\Ga'}, \,\boldsymbol\cdot\, \big) \colon \Des \big( \Sym_G(\Ga') \big) \to \Des_\sfL
\end{equation}
of categories fibred in groupoids over $\Cart$.
Again by the linearity of $\Hom_1$, it follows straightforwardly that
$(\,\boldsymbol\cdot\,) \otimes \pt^{\Ga'}_1$ and $\Hom_1\big(\pt^{\Ga'}_1,\,\boldsymbol\cdot\,\big)$ are mutually inverse functors.

To conclude the proof, we need to check that $\Psi$ is compatible with the monoid structures on $\Sym_G(\Ga')$ and on $\Des_\sfL$.
For $i = 0,1$, let $f_i \colon c \to G$ be smooth maps from $c \in \Cart$ to $G$, and consider objects $A_i \in \Sym_G(\Ga')_{|f_i}$ in the fibres over $f_i$.
By the explicit construction of the 2-group structure on $\Des_\sfL$ in the proof of Proposition~\ref{st:group str on Des_L} it follows that the hermitean line bundle underlying the object $\Psi(f_1,A_1) \otimes \Psi(f_0,A_0)$ of $\Des_\sfL$ is given as the descent of the bundle
\begin{align}
\label{eq:Psi and grp str}
	\sfL^\dashv \otimes (1 \times \Phi_{\ev_1})^*& \big( \Hom_1(\pt^{\Ga'}_{f_1}, \varpi_{f_1}^* A_1) \big) \otimes \Hom_1(\pt^{\Ga'}_{f_0}, \varpi_{f_0}^* A_0)
	\\[4pt]
	&\cong \sfL^\dashv \otimes \Hom_1 \big( (1 \times \Phi_{\ev_1})^* \widehat{f}_1^{\,*} \pt^{\Ga'}_1, (1 \times \Phi_{\ev_1})^*  \varpi_{f_1}^* A_1 \big) \otimes \Hom_1(\widehat{f}_0^{\,*} \pt^{\Ga'}_1, \varpi_{f_0}^* A_0)
	\\[4pt]
	&\cong \sfL^\dashv \otimes \Hom_1 \big( (1 \times \Phi_{\ev_1})^* \widehat{f}_1^{\,*} \pt^{\Ga'}_1 \circ \widehat{f}_0^{\,*} \pt^{\Ga'}_1,\, (1 \times \Phi_{\ev_1})^*  \varpi_{f_1}^* A_1 \circ \varpi_{f_0}^* A_0 \big)
	\\[4pt]
	&\cong \Hom_1 \big( (\widehat{f_1 \, f_0})^* \pt^{\Ga'}_1,\,
          \varpi_{f_1 \, f_0}^* (\Phi_{f_0}^* A_1 \circ A_0) \big) \ .
\end{align}
The first and second isomorphisms follow from the properties of the
internal hom-functor $\Hom_1$.
The third isomorphism is a direct application of Proposition~\ref{st:iso L cong pt_loop}.
The bundle in the last line is the line bundle underlying the object
$\Psi\big((f_1,A_1) \otimes(f_0,A_0)\big)$ of $\Des_\sfL$ (see~\eqref{eq:2-grp str on Sym(G)}).
Hence the canonical isomorphism~\eqref{eq:Psi and grp str} establishes the compatibility of $\Psi$.
Its coherence again follows from the properties of the transgression line bundle $\Ta\Ga$.
The proofs that $\Psi$ respects unitors as well as the weak
commutativity of the diagram~\eqref{eq:2grp ext iso} are straightforward.
\end{proof}

\subsection{Equivariant bundle gerbes}
\label{Sec: Equivar BGrbs}

We shall now investigate the relation between sections of the smooth 2-group extension $\Sym_G(\Ga) \to \ul{G}$ and equivariant structures on $\Ga$.
We first recall an explicit definition of an equivariant bundle gerbe from~\cite{GSW:Global_gauge_anomalies} (see also~\cite{MRSV:Equivariant_BGrbs}), which can be understood very nicely from the perspective of the formalism developed in~\cite{NS:Equivar}.

Let $G$ be a connected Lie group, $M$ a manifold with $G$-action $\Phi \colon G\times M \longrightarrow M$, and
$\Ga$ a hermitean bundle gerbe over $M$. 
Corresponding to the action groupoid $G\times M\rightrightarrows M$ there is a simplicial manifold 
\begin{equation}
\begin{tikzcd}
\cdots \  G^{\times 2} \times M \ar[rr, shift left] \ar[rr, shift right]
\ar[rr] & & G\times M \ar[rr, shift left] \ar[rr, shift right] & & M \ ,
\end{tikzcd}
\end{equation} with face maps $d_i:G^{\times n}\times M\to G^{\times
  n-1}\times M$ for $0\leq i\leq n$ given by
\begin{align}
 d_i(g_{0},g_1,\dots, g_{n-1},x)=  \begin{cases}
\big(g_{0},g_1,\dots , g_{n-2},\Phi_{g_{n-1}}(x)\big) & \text{for }i=n \\
(g_{0},g_1,\dots ,g_{i-1}\, g_{i},\dots, g_{n-1},x) & \text{for } 0 < i < n \\
(g_{1},\dots, g_{n-1},x) & \text{for } i=0
\end{cases} \ .
\end{align}
On $G\times M$, the face maps $d_0=\pr_M$ and $d_1=\Phi$ are the source and target
maps of the action groupoid.

\begin{definition}
Let $G$ be a connected Lie group, $M$ a manifold with $G$-action $\Phi \colon G\times M \longrightarrow M$, and $\Ga$ a hermitean bundle gerbe over $M$.
A \emph{$G$-equivariant structure} on $\Ga$ consists of a 1-isomorphism $A\colon \pr_M^*\Ga \longrightarrow \Phi^*\Ga$ 
over $G\times M$
and a 2-isomorphism $\chi \colon  d_2^*A \circ d_0^*A \longrightarrow
d_1^*A $ over $G^{\times 2}\times M$ such that
\begin{equation}
	d_2^*\chi \circ \big( 1_{(d_3 \circ d_2)^*A} \circ d_0^*\chi \big)= d_1^*\chi \circ \big( 1_{(d_0\circ d_0)^* A} \circ d_3^*\chi\big)
\end{equation}
over $G^{\times 3}\times M$. A \emph{morphism} $(A,\chi) \longrightarrow (A',\chi')$ between equivariant structures on $\Ga$
consists of a 2-isomorphism $\vartheta  \colon A \longrightarrow A'$
such that the diagram
\begin{equation}
\begin{tikzcd}[row sep=1cm]
d_2^*A\circ d_0^*A \ar[rr, "d_2^*\vartheta \circ d_0^*\vartheta "] \ar[d,"\chi",swap] & & d_2^*A'\circ d_0^*A' \ar[d,"\chi'"] \\ 
d_1^*A \ar[rr, "d_1^*\vartheta  ",swap] & & d_1^*A'
\end{tikzcd}
\end{equation}
commutes. 
We denote by $\mathcal{E}(\Ga)$ the groupoid of equivariant structures on $\Ga$. 
\end{definition}  

A \emph{splitting} of $p \colon \Sym_G(\Ga)\longrightarrow \underline{G}$ is a smooth 2-group homomorphism 
$s\colon \underline{G} \longrightarrow \Sym_G(\Ga) $ such that $p\circ s = 1_{\underline{G}}$. We assume here
for simplicity and without loss of generality that unitors are strictly preserved. We denote by $\mathcal{S}(\ul{G}\,;\Sym_G(\Ga))$ the 
groupoid of splittings of $p \colon \Sym_G(\Ga)\longrightarrow \underline{G}\,$. 
Concretely, a splitting consists of
\begin{myitemize}
\item 
a 1-isomorphism $ s(f) \colon \pr_M^*\Ga \longrightarrow \Phi_f^*\Ga$
of bundle gerbes on $c\times M$ for every
Cartesian space $c\in\Cart$ and sections $f\in G(c)$, 
\item 
a 2-isomorphism $s(\varphi)\colon s(f)\longrightarrow 
\varphi^* s(f') $ for every morphism $\varphi \colon f \longrightarrow f'$ in $\underline{G}$\,, and
\item 
a 2-isomorphism
$s(f)\otimes s(f') \longrightarrow  s(f\, f')$ in $\Sym_G(\Ga)$ for every $f,f'\in G(c)$,
\end{myitemize} 
such that $\varphi^*s(\varphi') \circ s(\varphi)= s(\varphi' \circ
\varphi)$ and the diagram
\begin{equation}
\begin{tikzcd}[row sep=1cm]
s(f) \otimes s(f')\otimes s(f'') \ar[rr] \ar[d] & & s(f) \otimes s(f'\, f'') \ar[d] \\
s(f\, f')\otimes s(f'') \ar[rr] & & s(f\, f' \, f'')
\end{tikzcd}
\end{equation}
commutes.
A morphism $\omega \colon s \longrightarrow s'$ of splittings consists
of 2-isomorphisms $\omega(f)\colon s(f)\longrightarrow s'(f)$ in
$\BGrb(c\times M)$ for all $f\in G(c)$ such that the diagrams
\begin{equation}
\begin{tikzcd}[column sep=1.5cm, row sep=1cm]
s(f) \ar[d, "s(\varphi)",swap]\ar[r, "\omega(f)"] & s'(f) \ar[d, "s'(\varphi)"] \\
\varphi^* s(f') \ar[r, "\varphi^*\omega(f')", swap] & \varphi^* s'(f')
\end{tikzcd}
\qandq
\begin{tikzcd}[row sep=1cm]
s(f)\otimes s(f') \ar[rr,"\omega(f)\otimes\omega(f')"] \ar[d] & & s'(f)\otimes s'(f') \ar[d] \\
s(f\, f') \ar[rr,swap,"\omega(f\,f')"] & & s'(f\, f')
\end{tikzcd}
\end{equation}
commute. 

In what follows we
construct an equivalence
\begin{align}
\Xi \colon \mathcal{E}(\Ga) \longrightarrow \mathcal{S} \big( \underline{G}\,;\Sym_G(\Ga) \big)
\end{align}
of categories.
Let $(A,\chi)$ be an equivariant structure on $\Ga$ and $(f\colon c\longrightarrow G) \in G(c)$. Pulling back $A$ along $\Phi_f$ gives rise to a 1-isomorphism $ \Phi_f^*A \colon \pr_M^* \Ga \longrightarrow \Phi_f^* \Ga$ over $c\times M$. We can define the section $\Xi(A,\chi) \colon \ul{G} \longrightarrow \Sym_G(\Ga)$ 
by $\Xi(A,\chi)(f)= \Phi_f^*A \colon \pr_M^*\Ga \longrightarrow \Phi_f^* \Ga $. The 2-isomorphisms $\Xi(A,\chi)(\varphi)$
are induced by general properties of pullbacks and the 2-isomorphism $\chi$ induces the 2-isomorphism encoding the 
compatibility with multiplication.  
The action of $\Xi$ on morphisms of equivariant structures is again by pullback along $\Phi_f$.  

\begin{theorem}
\label{Thm: Equivariant}
The functor $\Xi \colon \mathcal{E}(\Ga) \longrightarrow \mathcal{S}(\underline{G}\,;\Sym_G(\Ga))$ is an equivalence of categories.
\end{theorem}

\begin{proof}
We start by showing that $\Xi$ is essentially surjective.
Let $s\colon \ul{G} \longrightarrow \Sym_G(\Ga)$ be a splitting.
We pick a good open cover $\{c_i\}_{i \in \Lambda}$ of $G$. This
induces good open covers  of $G^{\times2}$ and $G^{\times3}$ given by $\{c_i\times c_j\}_{i,j \in \Lambda}$ and $\{ c_i \times c_j \times c_k \}_{i,j,k \in \Lambda}$, respectively. Every $c_i$ comes with an embedding $f_i\colon c_i \longrightarrow G$ and hence can be regarded as an 
object of $\underline{G}\,$. Applying the section $s\colon \underline{G}
\longrightarrow \Sym_G(\Ga)$ to all elements of the open cover
provides a collection of compatible 1-isomorphisms 
$$
s(c_i )\coloneqq s(f_i\colon c_i\longrightarrow G) \colon \pr_M^* \Ga
\longrightarrow \Phi_{f_i}^* \Ga \ .
$$
On double intersections $c_{ij}\coloneqq c_i \cap c_j$ we get coherent
2-isomorphisms\footnote{Here we interpret $s$ as a map of stacks, i.e.\
a natural transformation of 2-functors $\Cart \longrightarrow \Cat$, via the 
inverse Grothendieck construction.}  
\begin{align}
s(c_{ij})\colon s(c_i)_{|c_{ij}}\longrightarrow s(f_i{}_{|c_{ij}})= s(f_j{}_{|c_{ij}}) \longrightarrow s(c_j)_{|c_{ij}} \ ,
\end{align}
since 
$\Sym_G(\Ga)\longrightarrow \Cart$ is a Grothendieck fibration.
Hence the 1-isomorphisms $s(c_i)$
glue together to a 1-isomorphism $A_s \colon \pr_M^* \Ga
\longrightarrow \Phi^* \Ga$ over $G\times M$.

Let $\pr_1,\pr_2:G^{\times2}\to G$ be the projections to the first and
second factors, and $ m:G^{\times2}\to G$ the multiplication in $G$.
From $A_s$ we can construct three 1-morphisms $\pr_1^* A_s$, $\pr_2^*A_s$, and $ m^* A_s$  over $G^{\times2}\times M$. 
We would like to show that these 1-morphisms are canonically isomorphic to the 1-morphisms constructed from the
good open cover $\{ c_i\times c_j \}_{i,j\in \Lambda}$ by applying $s$
to the morphisms $\pr_{1|c_i \times c_j}$, $\pr_{2|c_i \times
  c_j} $, and $ m_{|c_i \times c_j}$ on $ {c_i \times c_j}
\longrightarrow G $, respectively. For this, consider the commutative diagram
\begin{equation}
\begin{tikzcd}[row sep=1cm]
c_i\times c_j \times M \ar[r] \ar[d, hookrightarrow] & c_i \times M \ar[d, hookrightarrow]  \\ 
G^{\times2}\times M \ar[r, "\pr_1", swap] & G\times M
\end{tikzcd}
\end{equation} 
which implies that 
$$
\pr_{1|c_i\times c_j}^*A_s= \pr_{1|c_i\times c_j}^*s(c_i
\hookrightarrow G) \xrightarrow{ \ \cong \ } s\big(c_i \times c_j 
\xrightarrow{\pr_1} c_i \hookrightarrow G\big) \ ,
$$
where the 2-isomorphism is part of the data of the section $s$. The same argument
shows the claim for $\pr_2$. To show the corresponding statement for $ m$ we need to pick a 
refinement $\{ \widetilde{c}_a\}_{a \in \widetilde{\Lambda}}$ of the cover $\{c_i\times c_j\}_{i,j \in \Lambda}$ 
such that the diagram
\begin{equation}
\begin{tikzcd}[row sep=1cm]
\coprod\limits_{a \in \widetilde{\Lambda}}\, \widetilde{c}_a  \ar[rd, " m"]  \ar[d, hookrightarrow] & \\
\coprod\limits_{i,j \in \Lambda}\, c_i\times c_j   \ar[d, hookrightarrow] & \coprod\limits_{i \in \Lambda}\, c_i \ar[d, hookrightarrow]  \\ 
G^{\times2} \ar[r, " m", swap] & G
\end{tikzcd}
\end{equation} 
commutes. The cover $\{ \widetilde{c}_a\}_{a \in \widetilde{\Lambda}}$ can be constructed by choosing a 
common refinement of the covers $\{c_i \times c_j\}_{i,j\in \Lambda}$
and $\{ m^{-1}(c_i)\}_{i\in \Lambda}$ of $G^{\times2}$. 

The multiplication of $c_i \times c_j 
\xrightarrow{\pr_1} G$ with $c_i \times c_j 
\xrightarrow{\pr_2} G$ in $\underline{G}$ is $c_i \times c_j 
\xrightarrow{ \  m \ }  G$. The structure of a smooth 2-group homomorphism on $s$ now provides natural 2-isomorphisms 
\begin{equation}\label{eq:chi2iso}
	s\big(c_i \times c_j \xrightarrow{ \ \pr_1 \ }  G\big) \otimes
        s\big(c_i \times c_j \xrightarrow{ \ \pr_2 \ }
        G\big)\longrightarrow s\big(c_i \times c_j \xrightarrow{ \  m
        \ }  G\big)
\end{equation} 
which glue together to a 2-isomorphism 
$\chi_s \colon \Phi^*(\pr_1^*A_s)\circ (\pr_2^* A_s)\longrightarrow
 m^* A_s$ over $G^{\times2}\times M$ because \eqref{eq:chi2iso} is a 2-isomorphism of smooth 1-isomorphisms. 
The coherence condition for $\chi_s$ over $G^{\times3}\times M$ follows from the observation that the various pullbacks to 
$G^{\times3}\times M$ can be constructed by applying $s$ to different functions from $c_i\times c_j \times c_k$ to $G$ and the coherence condition 
for $s$. 
This shows that $\Xi$ is essentially surjective.

We next show that the functor $\Xi$ is faithful: let $\vartheta,\vartheta' \colon (A,\chi) \longrightarrow (A',\chi')$ be isomorphisms
of equivariant structures on $\Ga$ such that
$\Xi(\vartheta')=\Xi(\vartheta)$, and let $g\in G$. 
We can take $c=\R^0$ and $f\colon c \longrightarrow G$ to be the constant map at $g$ to
conclude that $\varphi_{|\{g\}\times M}=\Xi(A,\chi)(\varphi)(f)$ and $\varphi_{|\{g\}\times M}=\Xi(A',\chi')(\varphi)(f)$
agree. Hence the two isomorphisms agree pointwise and the statement follows. 

Finally, we show that the functor $\Xi$ is full: let $(A,\chi) $ and $(A',\chi')$ be equivariant structures on $\Ga$ and 
$\omega \colon \Xi(A,\chi) \longrightarrow \Xi(A',\chi')$ a morphism of splittings. 
Evaluating $\omega$ on the good open cover $\{ c_i\}_{i \in \Lambda}$ from above provides
isomorphisms $\omega \colon A_{|c_i\times M} \longrightarrow A'_{|c_i\times M}$. 
Since $\omega$ is a morphism of splittings, these morphisms glue together to a
2-isomorphism $\vartheta_\omega\colon A\longrightarrow A'$. That this is an isomorphism 
of equivariant structures follows from the coherence conditions for $\omega$ and 
the observation that it suffices to check the conditions locally. 
\end{proof}

\begin{corollary}
A bundle gerbe $\Ga$ on a manifold $M$ admits an equivariant structure if and only if the 2-group extension 
$$
1\longrightarrow\HLBdl^M \longrightarrow \Sym_G(\Ga)\longrightarrow
\ul{G}\longrightarrow 1
$$ 
admits a splitting. 
\end{corollary}

\begin{definition}
Let $(\mathcal{G},A,\chi)$ and $(\mathcal{G}',A', \chi')$ be $G$-equivariant bundle gerbes on $M$. An \emph{equivariant 1-isomorphism} $\Ga \longrightarrow \Ga'$
consists of a 1-morphism of bundle gerbes $E \colon \Ga \longrightarrow \Ga'$
together with a 2-morphism $\gamma:\Phi^*E\to\pr_M^*E$ defined by the
diagram
\begin{equation}
\begin{tikzcd}[row sep=1cm]
\pr_M^* \Ga \ar[d, "\pr_M^* E",swap]  \ar[r,"A"] & \Phi^* \Ga \ar[d, "\Phi^*E"] \ar[ld, Rightarrow, "\gamma"] \\ 
\pr_M^* \Ga'  \ar[r,"A'",swap] & \Phi^* \Ga' 
\end{tikzcd}
\end{equation}
in $\BGrb(G\times M)$, such that for every $g,g' \in G$ there is an
equality of diagrams
\begin{equation}
\begin{tikzcd}[column sep= 1.5cm, row sep=1cm]
\pr_M^*\Ga \ar[d, " \pr_M^*E",swap]  \ar[r,"A"] & \Phi_{g}^* \Ga \ar[d, "\Phi_{g}^*E"] \ar[ld, Rightarrow, "\gamma"] \ar[r,"\Phi_{g}^*A"] &\Phi_{g'\,g}^* \Ga \ar[d, "\Phi_{g'\,g}^*E"] \ar[ld, Rightarrow, "\gamma"] \\ 
\pr_M^*\Ga' \ar[rr, bend right=40,"A'",swap]  \ar[r,"A'",swap] & \Phi_{g}^* \Ga' \ar[d,"\chi'", Rightarrow, shorten >= 5] \ar[r,"\Phi_{g}^*A'",swap] & \Phi_{g'\,g}^* \Ga' \\
& \ & 
\end{tikzcd}
\quad = \quad
\begin{tikzcd}[column sep= 1.5cm, row sep=1cm]
& \Phi_{g}^*\Ga \ar[rd,"\Phi_{g}^*A"] \ar[d,Rightarrow, "\chi"] & \\
\pr_M^*\Ga \ar[d, "\pr_M^* E",swap]  \ar[ru,"A"] \ar[rr,"A",swap] & \   & \Phi_{g'\,g}^* \Ga \ar[d, "\Phi_{g'\,g}^*E"] \ar[lld, Rightarrow, "\gamma"] \\ 
\pr_M^*\Ga'  \ar[rr,"A'",swap] & & \Phi_{g'\,g}^* \Ga' 
\end{tikzcd}
\end{equation}
\end{definition}

Being an equivariant 1-isomorphism is a structure and not a property: given a 1-morphism $E\colon\Ga\to\Ga'$ of bundle gerbes there is a
set $\mathcal{E}(E)$ of equivariant structures on $E$. 
According to Theorem~\ref{Thm: Equivariant} we can describe the
equivariant structures on $\Ga$ and $\Ga'$ by splittings
$s\colon \underline{G}\longrightarrow \Sym_G(\mathcal{G})$ and $s'\colon \underline{G} \longrightarrow \Sym_G(\mathcal{G}')$.
We shall now give a description of an equivariant structure on 
$E$ using these homomorphisms of 2-groups.
For this, recall from the end of Section~\ref{sec:Sym(Ga) global} that any 1-isomorphism $E \colon \Ga \to \Ga'$ in $\BGrb(M)$ gives rise to a morphism of smooth 2-groups $\widehat{E} \colon \Sym_G(\Ga) \to \Sym_G(\Ga')$: choose an adjoint inverse $E^{\vee}$ for $E$ and define $\widehat{E}$ via
\begin{equation}
	(f,A) \longmapsto \Phi_f^*E \circ (f,A) \circ E^{\vee} \ .
\end{equation}

\begin{proposition}
There is a natural bijection $\Xi_E$ between the set $\mathcal{E}(E)$ of equivariant structures 
on $E$ and the set of 2-isomorphisms $\widehat{\gamma}\colon \widehat{E} \circ s \longrightarrow s'$ of smooth morphisms of 2-groups. 
\end{proposition}

\begin{proof}
Let $\Ga$ and $\Ga'$ be $G$-equivariant bundle gerbes on a manifold
$M$ with a smooth $G$-action $\Phi\colon G\times M 
\longrightarrow M$, and let $(E,\gamma) \colon \Ga \longrightarrow
\Ga'$ be a 1-isomorphism of equivariant bundle gerbes. Fix an adjoint inverse $E^{\vee}$ to $E$. 
We construct the 2-isomorphism $\Xi_E(\gamma) \colon \widehat{E} \circ
s \longrightarrow s'$ as follows: let 
$f\colon c \longrightarrow G $ be an element of $\underline{G}\,$. The natural transformation $\Xi_E(\gamma)$ 
consists of a 2-isomorphism $ \Xi_E(\gamma)_f \colon  \widehat{E} \circ s(f) \longrightarrow s'(f)$
which we construct by the diagram
\begin{equation}
\begin{tikzcd}
& \pr_M^* \Ga \ar[dd,"\pr_M^* E"] \ar[ldd, Rightarrow, dashed, shorten >= 7.5ex] \ar[rr,"s(f)"] & & \Phi_f^* \Ga \ar[dd,"\Phi_f^* E"] \ar[lldd, "(1\times f)^*\gamma", Rightarrow] \\ 
\pr_M^* \Ga' \ar[ru, "\pr_M^*E^{\vee}"] \ar[rd, "1",swap] & \  & & \\
 \ & \pr_M^* \Ga' \ar[rr,"s'(f)"]  & & \Phi_f^* \Ga'
\end{tikzcd}
\end{equation}
Using the same methods as in the proof of Theorem~\ref{Thm: Equivariant} one can show that $\Xi_E$ is 
a bijection.
\end{proof}

\begin{definition}\label{def:eq2iso}
An \emph{equivariant 2-isomorphism} of $G$-equivariant bundle gerbes $(E,\gamma)\longrightarrow (E',\gamma')$ consists 
of a 2-morphism $\eta \colon E \longrightarrow E'$ such that there is
an equality of diagrams
\begin{equation}
\begin{tikzcd}[column sep= 1.5cm,  row sep =1.2cm]
\pr_M^* \Ga \ar[d, bend right=90,"\pr_M^* E'",swap, ""name=LL]  \ar[d, "\pr_M^* E",""name=LL2]  \ar[r,"A"] & \Phi^* \Ga \ar[d, "\Phi^*E"] \ar[ld, Rightarrow, "\gamma"] \\ 
\pr_M^* \Ga'  \ar[r,"A'",swap] & \Phi^* \Ga' 
\ar[from=LL2, to=LL,Rightarrow, "\pr_M^*\eta",shorten <= 3pt, shorten >= 3pt]
\end{tikzcd}
\quad = \quad
\begin{tikzcd}[column sep= 1.5cm, row sep =1.2cm]
\pr_M^* \Ga   \ar[d, "\pr_M^* E'"]  \ar[r,"A"] & \Phi^* \Ga \ar[d, "\Phi^*E'",""name=LL2,swap] \ar[ld, Rightarrow, "\gamma'"] \ar[d, bend left=90,"\Phi^* E", ""name=LL] \\ 
\pr_M^* \Ga'  \ar[r,"A'",swap] & \Phi^* \Ga' 
\ar[from=LL, to=LL2,Rightarrow, "\Phi^*\eta",shorten <= 3pt, shorten >= 3pt]
\end{tikzcd}
\end{equation} 
\end{definition}
Being an equivariant 2-morphism is a property. 

The 2-group extension $\Sym_G(\Ga)\longrightarrow \underline{G}$
can also be used to study the existence of equivariant structures on 1-morphisms.
A condition for 2-isomorphisms of bundle 
gerbes to be equivariant is
\begin{proposition}
	Let $(E,\gamma)$ and $(E',\gamma')$ be equivariant 1-isomorphisms. A 2-isomorphism $\eta \colon E 
	\longrightarrow E'$ is equivariant if and only if $
        \Xi_{E'}(\gamma') \circ (\widehat{\eta}\circ 1_\gamma) =\Xi_{E}(\gamma)$. 
\end{proposition}
\begin{proof}
This follows from Definition~\ref{def:eq2iso} using the fact that the inverses $E^{\vee}$
and $E'{}^{\vee}$ are adjoints to $E$ and $E'$. 
\end{proof}

\section{Application I: Nonassociative magnetic translations}
\label{Sec: Magnetic translations}

Nonassociativity in quantum mechanics has a long history dating back
to foundational work on the theory in the 1930's. Its interest was revived in the
1980's with the realisation that the magnetic translation operators
on the states of a charged particle moving in a magnetic monopole
background generally form a nonassociative
algebra~\cite{Jackiw:3-cocycles,Gunaydin:1985ur};
see~\cite{Szabo:Durham} for a mathematical introduction to the subject
together with a survey and comparison of the various approaches to the
quantisation of the pertinent twisted Poisson structures. The recent revived interest in these
models has come about from their conjectural relevance to the
low-energy dynamics
of closed strings in non-geometric
backgrounds,
which are based on arguments invoking T-duality applied to target spaces that are tori or more generally
total spaces of torus bundles~\cite{Lust:2010iy,MSS:NonGeo_Fluxes_and_Hopf_twist_Def,Bakas:2013jwa,Mylonas:2013jha}, and other compact Lie groups~\cite{Blumenhagen:2010hj}. See
e.g.~\cite{Szabo:Corfu} for a contemporary introduction to the
subject with further references.

As a first application of the general framework presented in this paper, we reformulate the 
well-known theory of magnetic translations for source-free magnetic fields in the language of Section~\ref{Sec: U(1)}.
We then use the results of Section~\ref{sect:extensions from gerbes} to describe nonassociative 
magnetic translations, which were first studied from a geometric
perspective in~\cite{BMS:NA_translations} on $\R^d$. They were subsequently  
studied from a quantum field theory perspective in~\cite{Mickelsson},
where generalisations from $\R^d$ to connected Lie groups are also considered. Here we will show that they are induced by a natural section of $\Sym_G(\Ga)\longrightarrow \underline{G}$
constructed using the parallel transport of Section~\ref{Sec: parallel transport}.

\subsection{Magnetic translations on $\bbT^d$}
\label{Ex: Line bundles on T}

Magnetic translations appear in the quantum mechanics of an
electrically charged particle moving 
on a manifold $M$ in the presence of a magnetic 
field, which is given by a 2-form $B\in \Omega^2(M)$. In the
semi-classical Maxwell theory of electromagnetism, the 2-form $B$ is
closed and has integer periods. The first requirement $H=\dd B=0$ is
the statement that there are no magnetic monopoles. The second
requirement is the Dirac charge quantisation condition which states that 
$B$ is the curvature of a connection on a hermitean line bundle $L$
over $M$. In Bloch theory (see e.g.~\cite{Gruber}), the line bundle $L$ is used in geometric
quantisation of the shift of the canonical symplectic structure on the
cotangent bundle $T^*M$ by the 2-form $B$, so that the quantum Hilbert 
space of wavefunctions for the particle is $\mathcal{H}= {\rm
  L}^2(M;L)$, the space of square-integrable
sections of $L$. The (global) symmetry group $G$ of the particle acts on $M$, and one would like to promote the $G$-action 
to an action on the Hilbert space by linear operators. In quantum
mechanics, this action on $\mathcal{H}$ is only required to define a
projective representation of $G$. If $G$ acts via translations the resulting operators are called 
\emph{magnetic translations}. The construction in Section~\ref{Sec: U(1)} provides 
a universal mechanism to construct magnetic translations,   
which we will illustrate on the example of a
$d$-dimensional torus $M=\bbT^d$. Magnetic translations on $\bbT^d$
have been studied in e.g.~\cite{Fiore,Davighi} (for constant magnetic
fields $B$), but our treatment is more general and also fits in with
expectations from string theory.

Instead of working on $\bbT^d$ directly, we work equivariantly on the
universal cover
$\R^d$ by viewing $\bbT^d=\R^d/\Z^d$ as the quotient of the
natural free action $\tau$ on $\R^d$ of the discrete subgroup
$\Z^d\subset\R^d$ by translations. The corresponding projection $\pi\colon \R^d\longrightarrow \bbT^d$ 
is a surjective submersion. To describe line bundles on $\bbT^d$ we
consider the diagram of manifolds 
\begin{equation}
\begin{tikzcd}[row sep=1cm]
\U(1) &  \\
\R^d\times \Z^d \cong \R^d \times_{\bbT^d} \R^d \ar[r, shift left,
"{\pi_1}"] \ar[r, shift right, "{\pi_0}", swap] \ar[u,"f"] & \R^d \ar[d,"\pi"] \\
 &  \bbT^d 
\end{tikzcd}
\end{equation}  
where we use the identification $\R^d\times \Z^d \ni (x,i) \mapsto
(x,x+i)\in  \R^d \times_{\bbT^d} \R^d $; under this identification,
$\pi_0=\pr_{\R^d}$ is the projection and $\pi_1=\tau$ is the
$\Z^d$-action on $\R^d$.
Any line bundle on $\bbT^d$ can be described by a smooth function $f\colon \R^d\times \Z^d\longrightarrow \U(1)$ satisfying 
$$
f(x+i,j)\, f(x,i)=f(x,i+j)
$$ 
for all $x\in\R^d$ and $i,j\in \Z^d$.
This means that $f$ is a 1-cocycle on the group $\Z^d$ with values in the $\Z^d$-module 
$C^ \infty(\R^d;\U(1))$.
We will use the notation $f_i(\,\boldsymbol\cdot\,)\coloneqq f(\,\boldsymbol\cdot\,,i)\in
C^\infty(\R^d;\U(1))$ for $i\in\Z^d$. 
Concretely, the $\U(1)$-bundle described by $f$ is the quotient
\begin{align}
P_f \coloneqq \big( \R^d\times \U(1) \big)\big/{\sim} 
\end{align}
by the equivalence relation
\begin{align}
(x+i,1) \sim \big(x, f_i(x)\big) 
\end{align}  
for all $x\in\R^d$ and $i\in\Z^d$. Sections of the associated line
bundle $L_f \longrightarrow \bbT^d$ are in one-to-one correspondence
with equivariant functions on the universal covering space, which are
functions $\psi\in C^\infty(\R^d;\C)$ satisfying quasi-periodic boundary conditions 
\begin{align}
\psi(x+i)= f_i(x)\,\psi(x) \ . 
\end{align} 

The action of the translation group $G=\R^d_{\mathtt{t}}$ on $\R^d$,
$x\mapsto x+v$,  induces
an action $\tau$ of $\R^d_{\mathtt{t}}$ on $\bbT^d$. For $v\in\R^d_{\mathtt
  t}$, the 
bundle $\tau_v^*P_f$ is described by the functions $\tau_v^*f_i=f_i(\,\boldsymbol\cdot\, -v)$.\footnote{Note that the functions $f_i$ are not 
invariant under the subgroup $\Z_{\mathtt{t}}^d\subset
\R_{\mathtt{t}}^d$, whereas the bundles $P_f$ 
and $P_{\tau_v^*f}$ are canonically isomorphic. This is nothing but a concrete
implementation of the fact that pullbacks are only well defined up to canonical isomorphism.   \color{black}} This allows us to give a concrete 
description for the fibres of the principal bundle $\Sym_{\R^d_{\mathtt
    t}}(P_f)\to\R^d_{\mathtt{t}}$ as
\begin{equation}\label{Eq: Sym for MT}
\Sym_{\R^d_{\mathtt{t}}}(P_f)_{|v} = \Bun_{\U(1)}(M)(P_f, P_{\tau_v^*f})=
  \big\{g\in C^\infty(\R^d, \U(1)) \ \big| \ g(x+i)=f_i(x)\,f_i(x-v)^{-1}\, g(x) \big\} \ .
\end{equation}     
The group\footnote{For an 
action of a group $G$ on a set $X$, we denote by $X^G\subseteq X$ the subset of $G$-invariants.}  $C^\infty(\bbT^d;\U(1))=C^\infty(\R^d;\U(1))^{\Z^d}$
acts freely and transitively on 
$\Sym_{\R^d_{\mathtt{t}}}(P_f)_{|v} $ by pointwise multiplication. The multiplicative structure from
Proposition~\ref{Prop: multiplicative structure bundle} on
$\Sym_{\R^d_{\mathtt{t}}}(P_f)$ takes the concrete form
\begin{align}
\mu\big((g,v),(g',v')\big)\coloneqq \big((\tau_{v'}^*g) \, g',v+v'\big) \ .
\end{align} 

A smooth section of the short exact sequence   
\begin{align}
1\longrightarrow C^\infty(\bbT^d, \U(1))\longrightarrow \Sym_{\R^d_{\mathtt
  t}}(P_f)\longrightarrow \R^d_{\mathtt{t}} \longrightarrow 1
\end{align}
induces a twisted action of $\R^d_{\mathtt{t}}$ on the quantum
Hilbert space $\mathcal{H}={\rm L}^2(\bbT^d;L_f)$; here we do not
require this section to be a group homomorphism, and the 2-cocycle twisting this action takes values
in $C^\infty(\bbT^d; \U(1))$. We can construct such a section
from the choice of a connection on $P_f$, which reproduces the 
usual expression for magnetic translations. A connection on $P_f$ can be described 
by a 1-form $A\in \Omega^1(\R^d)$ satisfying  
\begin{align}
-\iu \, \dd \log f = \pi_0^*A - \pi_1^* A \ .
\end{align}
This condition implies that the closed 2-form $\dd A=\pi^*B$ descends
to a well-defined magnetic field $B$ on $\bbT^d$.
The section corresponding to $A$ is given by parallel transport:
\begin{align}
s_A \colon \R^d_{\mathtt{t}} & \longrightarrow \Sym_{\R^d_{\mathtt{t}}}(P_f)  \\
v & \longmapsto s_A(v)= \exp\Big( - \iu\, \int_{\Delta^1(\,\boldsymbol\cdot\,;v)}\, A \Big) \ , 
\end{align}
where 
$$
\Delta^1(x;v)= \{ x-v +t\, v \in \R^d \mid 0\leq t \leq 1 \} \ .
$$
We check that this is indeed an element of $\Sym_{\R^d_{\tt
    t}}(P_f)_{|v}$:
\begin{align}
\big(s_A(v)\big)(x+i) &=
  \exp\Big(-\iu\,\int_{\Delta^1(x+i;v)}\,A+\iu\,\int_{\Delta^1(x;v)}\,A\Big)
              \ \big(s_A(v) \big) (x)
                              \\[4pt]
&=\exp\Big(\int_{\Delta^1(x;v)}\,\dd\log
  f_i\Big) \ \big(s_A(v) \big) (x) \\[4pt]
&= f_i(x)\, f_i(x-v)^{-1} \ \big(s_A(v) \big) (x) \ .
\end{align}
The corresponding 2-cocycle describing the extension agrees with the 2-cocycle constructed 
in e.g.~\cite{Soloviev:Dirac_Monopole_and_Kontsevich,
  BMS:NA_translations}. 

By Proposition~\ref{Prop: Action line bundle} the extension 
$\Sym_{\R^d_{\mathtt{t}}}(P_f)$
acts on the total space of the line bundle $L_f$
and hence on the quantum state space $\mathcal{H}$. The section $s_A$ realises 
translations $v\in \R^d_{\mathtt{t}}$ as linear operators
$\mathcal{P}(v):\mathcal{H}\to\mathcal{H}$ on this Hilbert space via 
$$
\mathcal{P}(v)\psi=s_A(v)\,\tau_v^*\psi \ .
$$
One easily checks that $\mathcal{P}(v)\psi\in\mathcal{H}$ for $\psi\in\mathcal{H}$,
i.e. $\big(\mathcal{P}(v)\psi\big)(x+i) = f_i(x)\,
\big(\mathcal{P}(v)\psi\big)(x) $. 
However, they only provide a projective representation of the
translation group $\R^d_{\mathtt{t}}$
since $s_A$ is not a group homomorphism. Explicitly, using Stokes’
Theorem we find that the magnetic translations satisfy the relations
of the twisted group algebra
\begin{align}
\mathcal{P}(v)\,\mathcal{P}(v')= \exp \Big(
  -\iu\,\int_{\Delta^2(\,\boldsymbol\cdot\,; v',v) } \, \pi^*B\Big) \ \mathcal{P}(v+v') \ ,
\end{align} 
where 
$$
\Delta^2(x;v',v) = \{ x-v-v'+t_1 \, v' + t_2\, v\in\R^d \mid 0
\leq t_2 \leq t_1 \leq 1  \} \ ,
$$
and we used the relation
$$
\partial\Delta^2(x;v',v) =
\Delta^1(x;v)-\Delta^1(x;v+v')+\Delta^1(x+v;x+v+v')
$$
in the simplicial complex in $\R^d$.

\begin{remark}
By dropping the (quasi-)periodicity conditions everywhere one gets
back the description of magnetic translations
corresponding to (necessarily trivialisable) line bundles over $\R^d$ (cf.~\cite{BMS:NA_translations}).
\qen
\end{remark}

\subsection{Nonassociative magnetic translations from parallel transport}

Dirac's extension of the classical Maxwell theory assumes a singular magnetic
field $B$ whose 3-form curvature $H=\dd B$ is distributional, with
zero-dimensional support consisting of the locations of magnetic
monopoles on the configuration manifold $M$. However, in applications
to string theory the closed 3-form $H$ corresponds to an NS--NS flux and is typically
smooth, as we now assume. The framework described in Section~\ref{Ex: Line bundles on T} is not capable of encoding 
magnetic fields with non-vanishing magnetic charge $H=\dd B$, since in this 
case $B$ can never be realised as the curvature of a line bundle. The
quantisation problem now concerns an $H$-twisted Poisson structure
on the cotangent bundle $T^*M$~\cite{Szabo:Durham}, with twisting
of the canonical Poisson structure which spoils the Jacobi identity
for functions in $C^\infty(T^*M;\C)$ that vary along the vertical directions. The
corresponding quantum operators do not associate; it is not
possible to realise a nonassociative algebra by linear operators acting on a Hilbert
space. Different approaches to describing
the nonassociative quantum mechanics of charged particles moving in the background
of a magnetic field with smooth monopole sources are described in~\cite{Mylonas:2013jha, Bojowald:2014oea, KS:Symplectic_realisation}. 

Fluxes in string theory obey a generalised version of Dirac charge
quantisation (see e.g.~\cite{Szabo:2012hc}); in particular, the closed
3-form $H$ has integer periods and hence is the
curvature of a connection on a hermitean bundle gerbe over $M$. 
Based on this observation, in~\cite{BMS:NA_translations} we suggested
the following approach: geometrically the magnetic field $B$
can be interpreted as the curving on a trivial gerbe $\mathcal{I}_B$ with curvature $H$. 
We proposed to use the 2-Hilbert space of sections of $\mathcal{I}_B$ as a replacement 
for the usual Hilbert space of quantum mechanics. 
The category of sections of a gerbe $\Ga$ on a manifold $M$ is the morphism category 
$\Gamma(M;\Ga)\coloneqq \BGrb(\mathcal{I}, \Ga)$; for details on the 2-Hilbert space structure on this category we refer to~\cite{BSS--HGeoQuan,Bunk-Szabo--Fluxes_brbs_2Hspaces,Bunk--Thesis}.  
As evidence for our approach we constructed a projective 
action of nonassociative magnetic translation operators on this 2-Hilbert space, 
which naturally encodes the relations of the $H$-twisted Poisson
algebra on $T^*M$. However, the drawback of this approach is that it does not work for
topologically non-trivial fluxes $H$, or equivalently for
gerbes $\Ga$ with non-torsion Dixmier-Douady class. Extending our geometric approach to nonassociative
quantum mechanics along these lines was in fact one of our original motivations behind
the present paper.
 
Let us first explain how the action of nonassociative magnetic translations for magnetic 
fields with sources on $M=\R^d$, which was described
in~\cite{BMS:NA_translations}, can be constructed via the 2-group extensions 
from Section~\ref{sect:extensions from gerbes}. 
Every gerbe on $\R^d$ is isomorphic to a trivial gerbe 
$\mathcal{I}_B$ represented by the diagram
\begin{equation}
\begin{tikzcd}[row sep=1cm]
\R^d\times \C \ar[d,swap,"\pr"] & \\
\R^d\ar[r, shift left] \ar[r, shift right] & \R^d \ar[d,"1"] \\
&  \R^d 
\end{tikzcd} \ \ 
\end{equation} 
with trivial connection $A=0$ and curving $B\in \Omega^2(\R^d)$. 
The connection on $\mathcal{I}_B$ induces a section\footnote{To
  simplify the presentation, in the
  following we disregard smooth structures and work in the 2-category of
2-groups $2\Grp(\Grpd)$. The extensions to categories fibred
in groupoids over $\Cart$ is straightfoward.}
$$
s_{B}\colon \R^d_{\mathtt{t}}\longrightarrow \text{Sym}_{\R^d_{\mathtt{t}}}(\mathcal{I})
$$  
via parallel transport:
\begin{align}
s_{B}(v)\coloneqq \big(\pt_{1}^{\mathcal{I}_B}\big)_{|\Delta^1(\,\boldsymbol\cdot\,;v)}\colon \Ia_B \longrightarrow \tau_v^* \Ia_B \ .
\end{align}
Combining $s_B$ with the action of $\Sym_{\R^d_{\mathtt{t}}}(\Ia)$ on $\Ia$ induces a higher projective action of 
$\R^d_{\mathtt{t}}$ on $\Gamma(\R^d;\Ia)$. We refer to~\cite[Section 4]{BMS:NA_translations} for precise definitions.  

Since all line 
bundles over $\R^d$ are isomorphic to a trivial line bundle, the
category \smash{$\text{Sym}_{\R^d_{\mathtt{t}}}(\mathcal{I})_{|v}$} at $v\in
\R^d_{\mathtt{t}}$ is equivalent to
the category with one object and morphisms 
described by smooth functions $\R^d\longrightarrow \U(1)$.  Thus
\begin{align}
\Sym_{\R^d_{\mathtt{t}}}(\mathcal{I}) \cong \big(*\DS
  C^\infty(\R^d;\U(1))\big) \rtimes \R^d_{\mathtt{t}} \ .
\end{align} 
The 2-group structure on $\Sym_{\R^d_{\mathtt{t}}}(\mathcal{I})$ is given by 
\begin{align}
v\otimes v'&\coloneqq v+v' \ , \\[4pt]
(h,v)\otimes(h',v')&\coloneqq \big((\tau_{v'}^*h) \, h',v+v'\big) \ .
\end{align}
However, under this equivalence the action of the magnetic translation operators 
becomes elusive. In~\cite{BMS:NA_translations} we equipped the 1-morphisms 
with a connection to circumvent this problem. We did not require the 2-morphisms to preserve these connections
leading to equivalent categories. The parallel transport 1-morphisms $\pt_1^{\Ia_B}$ can be equipped with 
such a connection in a canonical way. This reproduces the constructions from~\cite{BMS:NA_translations}.  

For this, let $\mathcal{I}_B$ be a trivial bundle gerbe
on a smooth manifold $M$ with curving $B\in\Omega^2(M)$ and curvature $H=\dd B$.
We fix a base point $x\in M$. Via transgression and regression  
we get a bundle gerbe $\Ra \Ta(\mathcal{I}_B)$ defined over the
diffeological path fibration 
$P_0M\longrightarrow M$. 
The corresponding line bundle over $(P_0M)^{[2]}$ admits a canonical trivialisation.
It comes equipped with
a connection 1-form given as the pullback of the transgression of $B$ to
the loop space $LM$ along the embedding $l\colon (P_0M)^{[2]} \longrightarrow LM$. 
To describe the curving of $\Ra \Ta(\mathcal{I}_B)$ we note that a tangent vector
to a based path $\gamma \in P_0M$ is a smooth section $ \mathcal{V} \in \Gamma ([0,1];\gamma^*TM)$ 
which is zero in a neighbourhood of $0$ and constant in a neighbourhood of $1$. 
The 2-form $\Ra \Ta(B)\in \Omega^2(P_0M)$ is defined by the transgression formula
\begin{align}
\Ra \Ta(B)(\mathcal{V},\mathcal{V}') = \int_\gamma\,
  \imath_{\mathcal{V}'} \imath_\mathcal{V} H \ ,
\end{align}   
where $\imath$ denotes the interior multiplication between a vector
and a form.
There is a natural 1-isomorphism \cite[Section 6.1]{Waldorf:Transgression_II} 
$W  \colon \Ra \Ta(\mathcal{I}_B) \longrightarrow \mathcal{I}_B$ with
underlying diagram
\begin{equation}
\begin{tikzcd}
 & & P_0 M \ar[rd] \ar[ld] & & \\
 (P_0 {M})^{[2]} \ar[r, shift left] \ar[r, shift right] &P_0 M \ar[rd] & & \ar[ld] M & \ar[l, shift left] \ar[l, shift right] M \\
& & M & &
\end{tikzcd}
\end{equation}
of diffeological spaces. The line bundle which is part of this 
$1$-morphism has a canonical trivialisation for trivial bundle gerbes $\mathcal{I}_B$ 
and has the connection 1-form $A_W\in \Omega^1(P_0M)$ defined 
by the transgression formula
\begin{align}
A_W(\mathcal{V})= \int_\gamma\, \imath_\mathcal{V} B \ .
\end{align} 

In order to describe the parallel transport, we pull everything back to the path space
$PM$ along the two evaluation maps $\ev_0,\ev_1 \colon PM \longrightarrow M$. 
The parallel transport defined in Section~\ref{Sec: parallel transport} is a 1-morphism 
\begin{align}
\pt_1^{\Ra \Ta(\mathcal{I}_B)}\colon \ev_0^*\Ra \Ta(\mathcal{I}_B) 
\longrightarrow \ev_1^* \Ra \Ta(\mathcal{I}_B) 
\end{align} 
given by 
a line bundle with connection over the space $P_{\partial \Delta^2}M$. 
An element of $P_{\partial \Delta^2}M$ 
is a triple of paths $(\gamma_{xy},\gamma_{yz},\gamma_{xz})$, where $\gamma_{xy}$ is a path 
from the base point $x$ to some other point $y\in M$, $\gamma_{yz}$ is a path from $y$ to
a third point $z\in M$, and $\gamma_{xz}$ is a path from $x$ to $z$. 
Again, the line bundle for the parallel transport is trivial since we work
with a trivial bundle gerbe. Its connection is given by 
\begin{align}
 \int_{ \overline{\gamma_{xz}} \star \gamma_{yz} \star \gamma_{xy}} \,
\imath_\bullet B \ ,
\end{align} 
where the notation means that the evaluation on a tangent vector $\mathcal{V}$ is 
given by 
\begin{align}
 \int_{ \overline{\gamma_{xz}} \star \gamma_{yz} \star \gamma_{xy}} \,
\imath_\mathcal{V} B \ . 
\end{align} 
We obtain a 1-morphism 
\begin{align}
\pt_1^{\mathcal{I}_B}\colon \ev_0^* \mathcal{I}_B \xrightarrow{\ev_0^* W^{-1}} \ev_0^*
\Ra \Ta(\mathcal{I}_B) \xrightarrow{\pt_1^{\Ra \Ta(\mathcal{I}_B)}}\ev_1^*
\Ra \Ta(\mathcal{I}_B) \xrightarrow{\ev_1^*W} \ev_1^* \mathcal{I}_B \ ,
\end{align} 
representing the colimit from Definition~\ref{Def: parallel transport on Gerbe}.
Concretely, this is a trivial line bundle over $PM$ with connection 
\[ 
\int_{\gamma_{yz}} \, \imath_\bullet B \ . 
\] 
This is exactly the formula used for the magnetic translations
in \cite{BMS:NA_translations} in the case $M=\R^d$, hence it provides a conceptual underpinning 
of the constructions
in~\cite{BMS:NA_translations}, and moreover generalises them to trivial bundle gerbes on arbitrary 
manifolds $M$.     

\subsection{Nonassociative magnetic translations on $\bbT^d$}

Now we generalise the description of nonassociative magnetic translations to the $d$-dimensional
torus $\bbT^d$, see also~\cite{ Mickelsson} for a discussion from a
quantum field theory point of view. 
A problem in this context is that for topologically non-trivial gerbes on $\bbT^d$, there 
are no non-trivial sections. This makes the 2-Hilbert space of sections an uninteresting object 
to study.\footnote{We expect that there exists a better definition of 
sections circumventing this problem.} However, our 2-group extension still exists and should encode 
the geometry of nonassociative magnetic translations in this context,
regardless of whether or not sections exist. Non-trivial gerbes over $\bbT^d$ are similarly
  treated as coming from $\Z^d$-equivariant (topologically trivial) gerbes over
  $\R^d$, as in e.g.~\cite[Section~7.1]{MW16}.

Bundle gerbes on $\bbT^d$ can be described using the surjective submersion
$\pi \colon \R^d\longrightarrow \bbT^d$ and the corresponding
diagram
\begin{equation}
\begin{tikzcd}[column sep=1em, row sep=1cm]
\pi_{0,1}^*L\otimes \pi_{1,2}^*L \ar[rd] \ar[rr,"f"] & &  \pi_{0,2}^*L
\ar[ld] & L \ar[d] & & \\
& \R^d\times \Z^d \times \Z^d \ar[rr, shift left] \ar[rr, shift right]
\ar[rr] & & \R^d\times \Z^d \ar[rr, shift left, "\pi_1"] \ar[rr, shift
right, swap,"\pi_0"] & & \R^d \ar[d,"\pi"] \\
& & & & & \bbT^d 
\end{tikzcd} 
\end{equation} 
Here we used the identification $(x,i,j)\in \R^d\times \Z^d \times \Z^d \mapsto (x,x+i,x+i+j)\in (\R^d)^{[3]}$. 
Concretely, a bundle gerbe consists of a line bundle $L$ 
over $\R^d\times \Z^d$, which we can assume to be trivial without loss of generality, and an isomorphism
$f\colon \pi_{0,1}^*L\otimes \pi_{1,2}^*L \longrightarrow
\pi_{0,2}^*L$ of line bundles over $\R^d\times \Z^d\times \Z^d$
satisfying a coherence condition over $\R^d\times (\Z^d)^{\times3}$. We can describe this 
isomorphism by a collection of smooth maps $f_{i,j}\colon \R^d\longrightarrow \U(1)$ for all $i,j\in \Z^d$, and the coherence 
condition translates to 
$$
f_{i,j}(x)\,f_{i+j,k}(x) = f_{i,j+k}(x)\,f_{j,k}(x+i)
$$
for all $x\in\R^d$ and $i,j,k\in\Z^d$, which is
the 2-cocycle condition for
$$
f_{\,\boldsymbol\cdot\,,\,\boldsymbol\cdot\,}\in {\rm C}^2\big(\Z^d;
C^\infty(\R^d;\U(1))\big) \ .
$$
We denote the gerbe described by $f$ 
as $\mathcal{G}_f$.

For $v\in\R^d_{\mathtt{t}}$, the pullback of $\mathcal{G}_f$ along the translation 
$\tau_v$ can be described by the 
2-cocycle $\tau_v^*f_{i,j}= f_{i,j}(\,\boldsymbol\cdot\, - v)$. 
Using
\cite[Proposition A.31]{Bunk--Thesis} we can describe the category
$\Sym_{\R^d_{\mathtt{t}}}(\mathcal{G}_f)_{|v}$ at $v\in \R^d_{\tt t}$
up to equivalence in the following way: 
its objects are functions  
$g\colon \R^d\times \Z^d \longrightarrow \U(1)$ satisfying 
\begin{align}
f_{i,j}(x -v)\,g_{i}(x)\,g_{j}(x+i)=g_{i+j}(x)\, f_{i,j}(x) \ ,
\end{align} 
for all $x\in\R^d$ and $i,j \in \Z^d$.
It is straightforward to deduce the morphisms in $\Sym_{\R^d_{\mathtt{t}}}(\mathcal{G}_f)$
from \cite[Proposition A.31]{Bunk--Thesis}; we find that a morphism 
from $g$ to $g'$ is described by a function $h\colon \R^d \longrightarrow \U(1)$
satisfying 
\begin{align}
h(x)\,g_i(x)= g_i'(x)\,h(x+i)
\end{align} 
for all $x\in\R^d$ and $i\in\Z^d$. 
Note that for the trivial 2-cocycle $f_{i,j}=1$ this describes the 
category of line bundles over $\bbT^d$ with arbitrary gauge 
transformations as morphisms. 
The 2-group structure on
\begin{align}
\Sym_{\R^d_{\mathtt{t}}}(\mathcal{G}_f) = \coprod_{v\in \R^d_{\mathtt{t}}} \, \Sym_{\R^d_{\mathtt{t}}}(\mathcal{G}_f)_{|v}
\end{align}
from Theorem~\ref{Thm: Sym is a 2-group} takes the form 
\begin{align}
(g_i,v)\otimes(g'_i,v')&\coloneqq \big((\tau_{v'}^*g_i)\,
                         g_i',v+v'\big) \ , \\[4pt] 
(h,v)\otimes(h',v')&\coloneqq \big((\tau_{v'}^*h)\, h',v+v'\big) \ , 
\end{align}
fitting into the 2-group extension from Theorem~\ref{st: 2-group extension}:
\begin{align}\label{eq:2groupextTn}
1 \longrightarrow \HLBdl(\bbT^d) \longrightarrow \text{Sym}_{\R^d_{\mathtt{t}}}(\mathcal{G}_f) \longrightarrow \R^d_{\mathtt{t}} \longrightarrow 1 \ .
\end{align}

As in the case of line bundles from Section~\ref{Ex: Line bundles on
  T}, a connection $(A,B)$ on $\mathcal{G}_f$ induces a section of the
extension \eqref{eq:2groupextTn}.
A connection on $\mathcal{G}_f$ is described by a 2-form $B\in \Omega^2(\R^d)$ and a 1-form
$A\in \Omega^1(\R^d\times \Z^d)$ satisfying
\begin{align}
-\iu \, \dd \log f_{i,j}(x) &= A_{i+j}(x)-A_i(x) -A_j(x+i) \ , \\[4pt]
\dd A &= \pi_1^*B-\pi_0^*B \ ,
\end{align}  
for all $x\in\R^d$ and $i,j \in \Z^d$. The second condition implies that
the closed 3-form $\dd B=\pi^*H$ descends to a well-defined flux $H$ on $\bbT^d$.
Using the connection we can construct a section 
$$
s_{A,B}\colon \R^d_{\mathtt{t}}\longrightarrow \text{Sym}_{\R^d_{\mathtt
    t}}(\mathcal{G}_f)
$$ 
by
\begin{align}
s_{A,B}(v)\coloneqq(g_i, v) \qquad \mbox{with} \quad g_i=\exp\Big( \iu\,
  \int_{\Delta^1(\,\boldsymbol\cdot\,;v)}\, A_i \Big) \ .
\end{align}
We check that this is indeed an element 
of $\text{Sym}_{\R^d_{\mathtt{t}}}(\mathcal{G}_f)_{|v}$:
\begin{align}
g_i(x)\,g_j(x+i)\,g_{i+j}(x)^{-1}&= \exp\Big( \iu\, \int_{\Delta^1(x;v)}\, A_i +\iu\, \int_{\Delta^1(x+i;v)}\, A_j - \iu\, \int_{\Delta^1(x;v)}\, A_{i+j} \Big) \\[4pt]
&= \exp\Big(  -\int_{\Delta^1(x;v)}\, \dd \log f_{i,j} \Big) \\[4pt]
&= f_{i,j}(x-v)^{-1}\,f_{i,j}(x) \ . 
\end{align}

For the multiplication we find 
\begin{align}
 s_{A,B}(v)\otimes s_{A,B}(v')&= \bigg(\exp\Big( \iu\,
                                \int_{\Delta^1(\,\boldsymbol\cdot\,-v';v)}\,
                                A_i +\iu\,
                                \int_{\Delta^1(\,\boldsymbol\cdot\,;v')}\,
                                A_i \Big) \,,\, v+v'\bigg) \\[4pt]
&=\bigg(\exp\Big( \iu\, \int_{\Delta^1(\,\boldsymbol\cdot\,;v+v')}\,
   A_i-\iu\, \int_{\Delta^2(\,\boldsymbol\cdot\,;v',v)}\,
   \big(B-\tau_i^*B\big)\Big) \,,\,v+v'\bigg) \\[4pt] 
&=s_{A,B}(v+v')\otimes \bigg(\exp\Big(-\iu\,
  \int_{\Delta^2(\,\boldsymbol\cdot\,;v',v)}\big(B-\tau_i^*B\big)
  \Big)\,,\,0\bigg) \ .
\end{align}
This particular product is associative on the nose. However, the line
bundle on $\bbT^d$
described by the transition functions
$\exp\big(-\iu\,\int_{\Delta^2(\,\boldsymbol\cdot\,;v',v)}\, (B-\tau_i^*B)\big)$ is non-trivial. We can use the decomposition
 \begin{align}
\int_{\Delta^2(\,\boldsymbol\cdot\,;v',v)}\, \dd\imath_\bullet B =
   \pounds_{\bullet}\int_{\Delta^2(\,\boldsymbol\cdot\,;v',v)} \, B
   -\int_{\Delta^2(\,\boldsymbol\cdot\,;v',v)}\, \imath_{\bullet}
   \pi^*H \ ,
\end{align}
where $\pounds$ is the Lie derivative, 
to construct a 2-isomorphism 
\begin{align}
\Pi_{v,v'}\coloneqq \exp\Big(-\iu\, \int_{\Delta^2(\,\boldsymbol\cdot\,;v',v)}\, B \Big) \colon 
 s_{A,B}(v)\otimes s_{A,B}(v') \longrightarrow s_{A,B}(v+v') \ ,
\end{align}
which has the advantage that the line bundle on the right-hand side is 
trivial. 
\begin{remark}\label{Remark: Non-associativity}
In this last representation the nonassociativity of the higher magnetic translations is realised by the
composition property 
\begin{align}
\Pi_{u,v+w} \circ \tau_{-u}^*\Pi_{v,w} = \omega_{u,v,w} \ \Pi_{u+v,w} \circ \Pi_{u,v} \ ,
\end{align}
where
\begin{align}
\omega_{u,v,w} = \exp\Big(\iu\,
  \int_{\Delta^3(\,\boldsymbol\cdot\,;w,v,u)}\, \pi^*H \Big)
\end{align}
and
\begin{align}
\Delta^3(x;w,v,u) = \{x-u-v-w+t_1\,w+t_2\,v+t_3\,u \in\R^d \mid 0\leq
  t_3\leq t_2\leq t_1\leq 1\} \ .
\end{align}
Concretely this means that there are two different ways to go from the
triple product
$s_{A,B}(u)\otimes s_{A,B}(v)\otimes s_{A,B}(w)$ to $s_{A,B}(u+v+w)$.
Their difference is controlled by the 3-cocycle
$\omega_{\,\boldsymbol\cdot\,,\,\boldsymbol\cdot\,,\,\boldsymbol\cdot\,}$
on the translation group $\R^d_{\mathtt{t}}$ with values in
$C^\infty(\bbT^d;\U(1))$, as depicted in the
commutative diagram
\begin{equation}
\begin{tikzcd}[row sep=1.0cm, column sep=0cm]
 & s_{A,B}(u) \otimes s_{A,B}(v+w) \ar[rd,"\Pi_{u,v+w}"] &  \\
s_{A,B}(u)\otimes \big(s_{A,B}(v)\otimes s_{A,B}(w)\big)\ar[d,"1",swap] \ar[ru,"1\otimes\Pi_{v,w}"] &  & s_{A,B}(u+v+w) \ar[d, "{\omega_{u,v,w}^{-1}}"]
	\\
\big(s_{A,B}(u)\otimes s_{A,B}(v)\big)\otimes s_{A,B}(w) \ar[rd,swap,"\Pi_{u,v}\otimes1"] &  &  s_{A,B}(u+v+w)  \\
 &s_{A,B}(u+v) \otimes s_{A,B}(w) \ar[ru,swap,"\Pi_{u+v,w}"] & 
\end{tikzcd}
\end{equation}
This is the implementation of nonassociativity in the higher categorical framework. 
\qen 
\end{remark}

\section{Application II: Anomalies in quantum field theory}
\label{Sec: Anomalies}

In this section we begin by using the group extensions $\Sym_G(P)$ from Section~\ref{Sec: U(1)}
to study the existence of equivariant trivialisations of line bundles.  
This has direct applications to the path integral description of the 
chiral anomaly in quantum field theory {on even-dimensional spacetime manifolds.} 
Then using the 2-group extension $\Sym_G(\mathcal{G})$ from
Section~\ref{sect:extensions from gerbes}, we study 
the analogous questions
for gerbes and apply our findings to the Hamiltonian description 
of the chiral anomaly {on odd-dimensional time-slices}.   

\subsection{Even dimensions: Chiral anomalies}
\label{Sec:Path integral}

Let $G$ be a connected Lie group, $M$ a manifold with smooth $G$-action $\Phi \colon G\times M \longrightarrow M$, and 
$(P,\chi)$ a $G$-equivariant $\U(1)$-bundle on $M$.
The equivariant structure on $P$ can be described by a splitting 
$s_{P} \colon G \longrightarrow \Sym_G(P)$. Assume furthermore that $P$ is
trivial as a line bundle, i.e.\ there exists a 
1-isomorphism $\psi\colon I \longrightarrow P$. 
The trivial bundle carries a canonical equivariant structure with corresponding splitting $s_I \colon G \longrightarrow \Sym_G(I)$.

Rewriting the result of Section~\ref{Sec: Eq bundles} slightly, we see that $\psi$ is equivariant if and only if 
the smooth 1-cocycle
\begin{align}
G  &\longrightarrow C^\infty(M;\U(1))
\\
g &\longmapsto s_{P}(g)\, \widehat{\psi}\big(s_I(g)\big)^{-1}
\end{align}
is trivial. Every other isomorphism $I \longrightarrow P$
differs from $\psi$ by a uniquely determined element of $C^\infty(M;\U(1))$.  Their corresponding 1-cocycles 
differ by the coboundary defined by this element. Hence the obstruction for an equivariant bundle which is trivial
as a line bundle to be also trivial as an equivariant bundle is an
element of the degree one group cohomology ${\rm H}^1\big(G;C^\infty(M;\U(1))\big)$.
This has also been observed in~\cite{CM94} from a different perspective. 
 
The question of whether a bundle is equivariantly trivial 
is important in the path integral perspective on
chiral anomalies in quantum field theory. Let $M$ be a based
even-dimensional compact Riemannian spin manifold, $G$ a 
Lie group, 
$Q$ a principal $G$-bundle on $M$, and $\rho\colon G \longrightarrow \End(V)$
a unitary representation of $G$ on a finite-dimensional vector space
$V$ which encodes the matter content of the field theory. 
Denote by $S^+$ and $S^-$ the positive and negative chirality spinor
bundles on $M$,
respectively, by $\mit\Gamma$ the group of based gauge transformations of $Q$ and 
by $\mathscr{A}$ the affine space of connections on $Q$. 
The field content of the theory are chiral spinors, which are smooth sections of the vector bundle
$S^+\otimes V$, where here $V$ is the hermitean vector bundle associated to $Q$ via the 
representation $\rho$.
There is a 
family of (twisted) Dirac operators 
$$
\D_A \colon \Gamma(M;S^+\otimes V) \longrightarrow  \Gamma(M;S^-\otimes
V)
$$
parameterised by gauge fields $A \in \mathscr{A}$, which are first
order elliptic differential operators acting on chiral spinors. These data
together define the content of a \emph{chiral gauge theory}.

The formal path integral over the chiral spinor fields is the
determinant $\det(\D_A)$ of the Dirac operator $\D_A$. 
However, the determinant of $\D_A$ is in general not a number but an
element of a complex
line, and it
can be defined only after suitable 
regularisation as an element of the fibre of the determinant line bundle 
$\det\longrightarrow \mathscr{A}$~\cite{AS}. This defines the (exponentiated)
effective action functional 
which is a section
\begin{align}
Z=\exp(-S):\Ascr\to\det \ ,
\end{align}
with $S(A)=-\log\,\det(\D_A)$.
The action of $\mit\Gamma$ on $\mathscr{A}$ via gauge transformations
can be lifted to the determinant line bundle, which hence descends to a 
line bundle over the moduli space of gauge connections $\mathscr{A}/ {\mit\Gamma}$.

Being an affine space, $\Ascr$ is contractible, so over $\mathscr{A}$ we can trivialise 
the determinant line bundle and hence identify the effective action functional $Z$ with 
a complex function. However, this might not be possible over the orbit 
space $\mathscr{A}/ {\mit\Gamma}$: if the descended line bundle
is non-trivial then we cannot identify the effective action functional with a complex
function in a gauge-invariant way, i.e.\ the gauge symmetry is anomalous. 
The line bundle over $\mathscr{A}/ {\mit\Gamma}$ is trivial if 
and only if we can choose a \emph{${\mit\Gamma}$-equivariant} trivialisation 
of the line $Z(A)=\det(\D_A)$. By our general discussion above, the obstruction to this is an element of
${\rm H}^1\big({\mit\Gamma}; \U(1)^\Ascr\big)$, where $\U(1)^\Ascr$ is the diffeological space of maps from $\Ascr$ to $\U(1)$.
An explicit formula for this smooth 1-cocycle is obtained in \cite{CM94}.    

\subsection{Odd dimensions: The Faddeev-Mickelsson-Shatashvili anomaly}
\label{Sec: Hamiltonian}

We shall now generalise the construction from the Section~\ref{Sec:Path integral}
to bundle gerbes. For this, we need to introduce a categorification of 
the first group cohomology which takes values in a smooth abelian 
2-group. 
We use a definition along the lines of \cite{BMS:NA_translations},
adjusted to the smooth setting. 
\begin{definition}\label{Def: smooth 1}
Let $G$ be a Lie group and $\sfA$ a smooth abelian 2-group equipped
with a left action $\rho$ of $\ul{G}\,$.
A \emph{smooth higher 1-cocycle} on
$G$ with values in $\sfA$ consists of
\begin{myitemize}
\item a morphism $\lambda \colon \underline{G}\longrightarrow \sfA$,
  $g\mapsto\lambda_g$, 
in $\Hscr$, and
 
\item a smooth natural isomorphism $\chi_{g,g'}\colon \lambda_g \otimes \rho_g(\lambda_{g'})\longrightarrow \lambda_{g\,g'}$ of smooth functors $\underline{G}\times_\Cart\ul{G}\longrightarrow \sfA$,
\end{myitemize} 
such that for every $c \in \Cart$:
\begin{myitemize}
\item $\lambda_{e_c} = I_c$ where $I_c$ is the monoidal identity object of the fibre $\sfA_{|c}$ of $\sfA$ over $c \in \Cart$,
and where $e_c \colon c \to G$ is the constant map at the identity element of $G$,

\item $\chi_{e_c,\,\boldsymbol\cdot\,}$ and $\chi_{\,\boldsymbol\cdot\,,e_c}$ agree with the left and right unitor morphisms 
in $\sfA_{|c}$, and

\item the diagram 
\begin{equation}
\begin{tikzcd}[row sep=1.0cm]
\lambda_g \otimes \rho_g\big(\lambda_{g'} \otimes \rho_{g'}(\lambda_{g''})\big) \ar[d] \ar[rr,"1 \otimes \rho_g(\chi_{g,g'})"]& & \lambda_g \otimes \rho_g(\lambda_{g'\,g''}) \ar[r,"\chi_{g,g'\,g''}"] & \lambda_{g\,g'\,g''} \\ 
\lambda_g \otimes \rho_g(\lambda_{g'}) \otimes \rho_{g\,g'}(\lambda_{g''}) \ar[rr,swap,"\chi_{g,g'}\otimes 1"]  & & \lambda_{g\,g'} \otimes \rho_{g\,g'}(\lambda_{g''}) \ar[ru,swap,"\chi_{g\,g',g''}"] & 
\end{tikzcd}
\end{equation}
commutes for all $g,g',g'' \in G(c)$. 
\end{myitemize}
\end{definition}
We will also need the concept of a higher coboundary.
\begin{definition}
Let $(\lambda,\chi)$ and $(\lambda',\chi')$ be higher 1-cocycles 
on a Lie group $G$ valued in a smooth abelian 2-group $\sfA$. A \emph{higher coboundary} between $(\lambda,\chi)$ and $(\lambda',\chi')$
consists of 
\begin{myitemize}
\item a morphism $\theta\colon  * \longrightarrow  \sfA$, and

\item smooth isomorphisms $\omega_g \colon \lambda_g\otimes \rho_g(\theta)\longrightarrow 
\theta \otimes \lambda'_{g}$ for every $g \in \ul{G}\,$,
\end{myitemize}
such that $\omega_{e_c}$ agrees with the symmetry isomorphism
$\beta_\sfA$, and the diagram
\begin{equation}
\begin{tikzcd}[row sep=1.0cm]
\lambda_g \otimes \rho_g\big(\lambda_{g'}\otimes
\rho_{g'}(\theta)\big) \ar[r] \ar[d] & \lambda_g\otimes
\rho_g(\theta)\otimes\rho_g(\lambda'_{g'}) \ar[r,"\omega_g\otimes 1"] & \theta\otimes \lambda'_g \otimes\rho_g(\lambda'_{g'}) \ar[d,"1\otimes\chi'_{g,g'}"] \\ 
\lambda_g \otimes \rho_g(\lambda_{g'})\otimes
\rho_{g\,g'}(\theta)\ar[r,swap,"\chi_{g,g'}\otimes 1"] & \lambda_{g\,g'}\otimes \rho_{g\,g'}(\theta) \ar[r,swap,"\omega_{g\,g'}"] & \theta\otimes \lambda'_{g\,g'}
\end{tikzcd}
\end{equation} 
commutes for all $c \in \Cart$ and $g,g',g'' \in G(c)$.  
\end{definition}
\begin{remark}
There is a natural definition of morphisms between higher
coboundaries, but these are not relevant for our purposes. 
\qen
\end{remark}

Let $G$ be a connected Lie group, $M$ a manifold with smooth $G$-action $\Phi \colon G\times M \longrightarrow M$, and 
$(\Ga,A,\chi)$ a $G$-equivariant bundle gerbe on $M$.
The equivariant structure on $\Ga$ can be described by a splitting 
$s_{\Ga} \colon \underline{G} \longrightarrow \Sym_G(\Ga)$, as
explained in
Section~\ref{Sec: Equivar BGrbs}. Assume furthermore that $\Ga$ is trivial as bundle gerbe, i.e. there exists a 
1-isomorphism
$E \colon \Ia \longrightarrow \Ga$. 
From the results in Section~\ref{Sec: Equivar BGrbs} we can deduce that 
the obstruction to the existence of an equivariant structure on 
$E$ is the higher 1-cocycle 
\begin{align}
( f\colon c \longrightarrow G) \longmapsto s_\Ga(f)^{-1}\circ \widehat{E}\big(s_\Ia(f)\big)
\ \in \ \BGrb(c \times M)(\pr_M^*\Ga, \pr_M^*\Ga) \cong \HLBdl(c\times
  M) \ .
\end{align} 
This 1-cocycle is trivial precisely if there exists a natural isomorphism to 
the constant 1-cocycle at the trivial line bundle $I$. The choice of such an 
isomorphism corresponds to the equivariant structure on $E$. 
The isomorphisms $\chi$ in Definition~\ref{Def: smooth 1} use the smooth 2-group 
structure on $\HLBdl^M$, and the isomorphisms which are part of the
splittings $s_\Ga$ and $s_\Ia$, and $\widehat{E}$.
Generally, two 1-isomorphisms $\mathcal{G}\longrightarrow
\mathcal{G}'$ of bundle gerbes differ 
by the action of a 1-automorphism of $\mathcal{G'}$. Hence the obstruction to the existence
of an equivariant isomorphism is an element of the first higher group cohomology of $G$ with coefficients
in the smooth abelian 2-group $\HLBdl^M$. 

Let us now explain the relation to the Hamiltonian description of chiral anomalies
in terms of bundle gerbes, which was worked out in~\cite{Mickelsson:1983xi,CM94,CM96,Index, Carey1997}.
Let $M$ be a based odd-dimensional compact Riemannian spin manifold,\footnote{The type of Hamiltonian anomaly discussed here can only occur on odd-dimensional manifolds, since otherwise
the chirality operator could be used to define consistent polarisations.}
$P$ a principal $G$-bundle on $M$, and $\rho\colon G \longrightarrow \End(V)$
a representation of $G$ describing the matter content of the gauge theory.  Again
we denote by $\mathscr{A}$ the affine space of connections on $P$ and by $\mit\Gamma$
the pointed group of gauge transformations. For every $A\in \mathscr{A}$ we can construct a massless 
Dirac operator 
$$
\D_A\colon \Gamma(M;S\otimes V) \longrightarrow \Gamma(M;S\otimes V) \ ,
$$
where $S\to M$ is the spinor bundle. The Dirac operator is a first
order self-adjoint elliptic differential operator, which serves as the
first quantised Hamiltonian acting on the
one-particle Hilbert space $\Ha=\Gamma(M;S\otimes V)$.

To define the fermionic Fock space $\mathcal{F}_A(\Ha)$ of the quantum field theory in the presence of a gauge field 
$A \in \mathscr{A}$, one has to pick a polarisation $\Ha=\Ha_+(A)\oplus \Ha_-(A)$.  
In general there are gauge fields $A\in \mathscr{A}$ for which the
Dirac operator $\D_A$ has zero modes; for these fields there is no
universal and natural way of choosing such a polarisation. 
Denote by $\mathscr{A}_0\subset \mathscr{A}\times \R$ the subset of pairs $(A,r)$ such 
that the real number $r$ is not contained in the spectrum of $\D_A$. To equip this space with a diffeology
we use the discrete diffeology on $\R$. For every point $(A,r)\in \mathscr{A}_0$ 
we get a decomposition of the one-particle Hilbert space 
$$
\Ha=\Ha_+(A,r)\oplus \Ha_-(A,r)
$$
into the positive and negative eigenstates of the operator $\D_A-r\, 1_\Ha$.
The corresponding Fock bundle $\Fa(\Ha)\to\Ascr_0$ has fibres 
$$
\mathcal{F}(\Ha)_{|(A,r)} = \midwedge\, \Ha_+(A,r) \otimes \midwedge\,
\Ha_-(A,r)^\vee \ .
$$

It is shown in~\cite{CM96}
that the corresponding projective Hilbert bundle descends to a bundle over $\mathscr{A}$,
and hence it induces a bundle gerbe $\Ga$ on $\mathscr{A}$. 
{The bundle gerbe can be explicitly described in terms of determinant lines associated to families of Dirac operators, see~\cite[Section~5]{CM96} for details.} 
Since
$\Ascr$ is contractible, over $\Ascr$ the 
projective Hilbert bundle is trivial and hence is associated to a bundle of Hilbert
spaces.
Again the action of $\mit\Gamma$ on $\Ascr$ lifts to an equivariant
structure on $\Ga$. Therefore $\Ga$ as well as the projective
Hilbert bundle descends to the orbit space $\Ascr/{\mit\Gamma}$.  The Faddeev-Mickelsson-Shatashvili anomaly is the obstruction to the existence of a well-defined bundle of
Hilbert spaces over $\Ascr/{\mit\Gamma}$, i.e.~to the existence of a trivialisation of the descended projective Hilbert bundle.
Equivalently, the anomaly vanishes if and only if $\Ga$ descends to a
trivial bundle gerbe on $\Ascr/{\mit\Gamma}$.
This in turn is the case if and only if $\Ga$ is trivial as a
\emph{$\mit\Gamma$-equivariant} bundle gerbe on $\Ascr$.

From the general discussion above it follows that 
the obstruction to the equivariant triviality of $\Ga$ is a smooth
higher 1-cocycle on $\mit\Gamma$ with
values in $\HLBdl^\Ascr$. Because $\Ascr$ is contractible, there is an
equivalence 
$$
\HLBdl^\Ascr \cong *
\DS \U(1)^\Ascr
$$ 
with the smooth category with one object and the diffeological mapping space $\U(1)^\Ascr$ as morphisms. Since this is a smooth 2-group 
with one object, Definition~\ref{Def: smooth 1} in this instance is equivalent to the definition
of an ordinary group 2-cocycle on $\mit\Gamma$ with values in $\U(1)^\Ascr$. That the obstruction to
the vanishing of the anomaly is a 2-cocycle of this form is well-known,
see e.g. \cite{CM96}; this cocycle reproduces the usual Schwinger terms in the commutator anomaly for
the gauge group action. 
{What is new here is the construction of a
smooth higher 
1-cocycle, which only reduces to an ordinary 2-cocycle because the space $\mathscr{A}$ is 
contractible, as well as a rigorous incorporation of the smooth structures. Computing this cocycle explicitly 
and comparing it to the Faddeev-Mickelsson-Shatashvili cocycle is beyond the scope of this paper. We expect this to be
possible using index theory following~\cite{Index}.}

\section{Application III: The string group}
\label{Sec:String}

Any compact simple Lie group $G$ has homotopy groups $\pi_3(G) \cong \rmH^3(G;\Z) \cong \Z$ and $\pi_i(G) \cong \rmH^i(G;\Z) \cong 0$ for $i = 0,1,2$; that is, $G$ is 2-connected.
It is of interest in topology and geometry (see
e.g.~\cite{DHH:String_Orientations,Stolz:Ricci_and_Witten,ST:Elliptic_Objects}),
as well as in string theory (see e.g.~\cite{SS:Non-Ab_SD_String}), to study 3-connected approximations to $G$ that arise as group extensions of $G$.
We denote such approximating objects by $\String(G)$ and call them models
for the \emph{string group} of $G$.
There is a variety of interpretations of what this means, based on different choices of ambient higher categories in which one considers $G$ to be a group object.
The general theme, however, is that one needs a way to realise a generator of $\pi_3(G) \cong \Z$ geometrically in the chosen framework, and a string group model for $G$ will generally be a choice of such a generator.

In this final section we recall the definition and construction of a
topological string group model, and show that our extensions
$\Sym_G(P)$ from Section~\ref{Sec: U(1)} provide a smooth enhancement thereof. We then propose the smooth
2-groups $\Sym_G(\Ga)$ and $\Des_\sfL$ from Section~\ref{sect:extensions
  from gerbes} as new string group models, for the specific choices of
$M = G$ and of a gerbe $\Ga$ on $G$ whose Dixmier-Douady class
generates $\rmH^3(G;\Z) \cong \Z$. A model for $\String(G)$ which is
very similar in spirit to our model $\Sym_G(\Ga)$ was found in~\cite{FRS:Higher_U(1)-connections}. 
However, that construction relies on the choice of connection on $\Ga$ and considers connection-preserving symmetries of $\Ga$.
Here, in contrast, we exhibit a construction of string group models for $G$ from symmetries of gerbes on $G$ without connections, and thus from representatives of the third integer cohomology of $G$ rather than its differential cohomology.
We defer further details and comments to Section~\ref{Sec:Smooth String}.

\subsection{A smooth string group model}

The simplest and original framework for considering string group models is that of topological spaces.

\begin{definition}
\label{def:String_t(G)}
Let $G$ be a compact simple simply-connected Lie group.
A topological model for the \emph{string group} $\String(G)$ of $G$ is
a topological 3-connected group $\String_{\mathtt{t}}(G)$ along with a fibration
$\String_{\mathtt{t}}(G) \to G$ whose typical fibre is an Eilenberg-MacLane space $K(\Z;2)$.
\end{definition}

Using homotopy and cohomology groups one shows that $\String_{\mathtt t}(G)$ cannot be a finite-dimen{-}sional Lie group~\cite{NSW:Smooth_string_group}.
If a topological string group model can be enhanced to consist of
smooth spaces (such as Fr\'echet manifolds or diffeological spaces),
we denote it by $\String(G)$ and refer to it as a smooth model for the string group of $G$.

We recall Stolz' model as a topological group~\cite{Stolz:Ricci_and_Witten}:
let $\PU$ denote the projective unitary group of an infinite-dimensional separable Hilbert space.
As a consequence of Kuiper's Theorem, $\PU$ has homotopy type $K(\Z;2)$.
Hence the classifying space $\sfB \PU$ has homotopy type $K(\Z;3)$, while at the same time it classifies topological principal $\PU$-bundles.
In particular, isomorphism classes of $\PU$-bundles on a space $X$ are
in one-to-one correspondence with elements of the set $\rmH^3(X;\Z)$.

Let $P \to G$ be a principal $\PU$-bundle on $G$ such that $P$
corresponds to a generator of $\rmH^3(G;\Z) \cong \Z$; such $\P\U$-bundles on $G$ are called \emph{basic}.
Let $\widehat{G}$ denote the group of $\PU$-equivariant homeomorphisms of $P$ to itself which act on $G$ as left multiplication by some element of $G$.
We can topologise $\widehat{G}$ as a subgroup of the topological group of homeomorphisms $P \to P$.
Thus $\widehat{G}$ comes with a continuous surjective group homomorphism $\widehat{G} \to G$.
The gauge group $\Gau(P)$ is the subgroup of $\widehat{G}$ of those
elements whose projection to $G$ is the identity element $e \in G$.

\begin{theorem}[{\cite[Section~5]{Stolz:Ricci_and_Witten} and \cite{NSW:Smooth_string_group}}]
The extension of topological groups
\begin{equation}
\begin{tikzcd}
	1 \ar[r] & \Gau(P) \ar[r] & \widehat{G} \ar[r] & G \ar[r] & 1
\end{tikzcd}
\end{equation}
exhibits $\widehat{G}$ as a topological model for $\String(G)$.
\end{theorem}
The crux of the proof of this theorem is showing that $\Gau(P)$ is
homotopy equivalent to $\PU$, i.e.~that it is an Eilenberg-MacLane space $K(\Z;2)$.
Part of the content in~\cite{NSW:Smooth_string_group} is to enhance this topological string group model to a smooth model in the sense that the groups appearing are Fr\'echet Lie groups.

The group extension $\widehat{G}\longrightarrow G$ agrees with the extension $\Sym_G(P) \to G$
constructed in Section~\ref{Sec: U(1)} when we set $M=G$ and $H=\PU$,
and let $G$ act on itself by left multiplication.
Thus we immediately obtain

\begin{corollary}
Let $P \to G$ be a basic $\P\U$-bundle.
The extension of diffeological groups
\begin{equation}
\begin{tikzcd}
	1 \ar[r] & \Gau(P) \ar[r] & \Sym_G(P) \ar[r] & G \ar[r] & 1
\end{tikzcd}
\end{equation}
exhibits $\Sym_G(P)$ as a smooth model for $\String(G)$.
\end{corollary}

\subsection{Smooth string 2-group models}
\label{Sec:Smooth String}

Let $G$ be a compact simply-connected Lie group, and let $\Ga$ be a
bundle gerbe on $G$ whose Dixmier-Douady class generates
$\rmH^3(G;\Z)$; such a bundle gerbe is said to be \emph{basic}.
Let $\Phi \colon G \times G \to G$ denote the left action of $G$ on itself by left multiplication.
In the spirit of Section~\ref{sect:extensions from gerbes}, it is reasonable to expect that we should also be able to interpret the smooth 2-groups $\Sym_G(\Ga)$ and $\Des_\sfL$ as models for $\String(G)$.
The idea of constructing $\String(G)$ as a smooth 2-group has also
been considered in e.g.~\cite{BCSS07, SP:String_group, Waldorf:String_2-group, NSW:Smooth_string_group,FRS:Higher_U(1)-connections}.
In the remainder of this section we will describe how $\Sym_G(\Ga)$ can be seen as a string 2-group model.
By Theorem~\ref{st:Sym = Des} it then follows that $\Des_\sfL$ is also a model for $\String(G)$.

Smooth string 2-group models usually consist of extensions of $G$ by
the smooth 2-group $\sfB\U(1)$, the delooping of the smooth abelian group $\U(1)$.
However, recall that in Theorem~\ref{st: 2-group extension} we
established $\Sym_G(\Ga)$ as an extension of $\ul{G}$ by the smooth
abelian 2-group $\HLBdl^G$.
Our point here is that what matters for string group models is only the homotopy type of the fibre and the total space of the map $\String(G) \to G$, so that there is a lot of flexibility in choosing the smooth 2-group $\sfA$ that extends $G$.
Observe that this ambiguity is inherent already in Definition~\ref{def:String_t(G)}.
This forces us to state which smooth 2-groups $\sfA$ are admissible in order to obtain smooth 2-group extensions of $G$ that deserve to be called string group models.
Our proposed definition for smooth string 2-group models emphasises
the structure of $\sfA$ as a smooth analogue of an Eilenberg-MacLane space $K(\Z;2)$.
Note that for every smooth abelian 2-group $\sfA$ and any manifold $M$, we can
define \v{C}ech cohomology of $M$ with coefficients in $\sfA$ by evaluating (a delooping of $\sfA$) on the \v{C}ech nerve of good open coverings of~$M$.

The definition of a smooth string 2-group model is thus a two-step process:
we first fix the homotopy type of the extending 2-group $\sfA$ in a
geometric way, and then we have to make precise when an $\sfA$-extension of $G$ has the correct homotopy type.

\begin{definition}
\label{def:BH as smooth groupoid}
Let $H$ be a diffeological group.
The \emph{delooping} $\sfB \ul{H} \in \Hscr$ is the category fibred in groupoids
over $\Cart$ whose objects are the Cartesian spaces $c \in \Cart$, and whose morphisms $c_0 \to c_1$ are pairs $(f_{01},h_{01})$ of smooth maps $f_{01} \colon c_0 \to c_1$ and $h_{01} \colon c_0 \to H$.
Composition of morphisms is given by $(f_{12},h_{12}) \circ (f_{01},h_{01}) = \big(f_{12} \circ f_{01}, h_{01} \, (h_{12} \circ f_{01})\big)$.
\end{definition}

If $H$ is abelian, then $\sfB \ul{H}$ naturally becomes a smooth
abelian 2-group.

\begin{definition}
\label{def:admissible 2Grp}
A smooth 2-group $\sfA$ is \emph{string-admissible} if it is abelian
and equivalent (as a smooth 2-group) to the delooping $\sfB \ul{H}$ of a
diffeological abelian group $H$ whose underlying topological space is
an Eilenberg-MacLane space $K(\Z;2)$.
\end{definition}

From the equivalence $\sfA \simeq \sfB \ul{H}\,$ it follows that \v{C}ech cohomology with coefficients in $\sfA$ is equivalent to \v{C}ech cohomology with coefficients in $H$, shifted by one degree.
Then since $H$ has homotopy type $K(\Z;2)$, it follows that there are isomorphisms
\begin{align}
	\cH{}^k(M; \sfA)
	\cong \cH{}^k(M; \sfB \ul{H})
	\cong \big[M, \sfB^{k+1}H\big]
	\cong \big[M, K(\Z;k+2)\big]
	\cong \rmH{}^{k+2}(M;\Z)
\end{align}
for all $k \in \N$.
(For the notion of \v{C}ech cohomology with coefficients in higher smooth groups, we refer the reader to~\cite{Schreiber:DCCT, NSS:oo-bdls_I}.)

From any smooth principal 2-bundle $\sfP \to \ul{M}$ over a manifold $M$ with
structure 2-group $\sfA$, we can distil a \v{C}ech cohomology class as
follows:
let $\Ua = \{U_i\}_{i \in I}$ be a good open covering of $M$. Denote
intersections by $U_{i_1\cdots i_n}:=U_{i_1}\cap\cdots\cap U_{i_n}$. 
Viewing the (intersections of) open patches $U_{i_1 \cdots i_n} \hookrightarrow M$ as objects in $\ul{M}\,$, we denote by $\sfP_{|U_{i_1 \cdots i_n}}$ the fibres of $\sfP$ over these objects.
By Definition~\ref{def:sfH-bundle}, we can choose an object $\psi_i \in \sfP_{|U_i}$ for every $i \in I$.
We can further choose an object $\sfa_{ij} \in \sfA_{|U_{ij}}$ for
every $i,j \in I$ and an isomorphism $g_{ij} \colon \psi_{i|U_{ij}}
\otimes \sfa_{ij} \to \psi_{j|U_{ij}}$ in $\sfP_{|U_{ij}}$ (where we
have chosen pullbacks of $\psi_i$ and $\psi_j$ to $\sfP_{|U_{ij}}$).
Over the triple overlaps $U_{ijk}$ we obtain isomorphisms
\begin{equation}
	\beta_{ijk} \colon \psi_{i|U_{ijk}} \otimes \sfa_{ij|U_{ijk}}
        \otimes \sfa_{jk|U_{ijk}} \to \psi_{i|U_{ijk}} \otimes
        \sfa_{ik|U_{ijk}} \ ,
\end{equation}
which are uniquely determined by the properties of the Grothendieck
fibration $\sfP \to \ul{M}\,$ (as previously, since $\ul{M}$ has discrete fibres, it follows that $\sfP \times_{\ul{M}}^{\mathtt{h}} \sfP = \sfP \times_{\ul{M}} \sfP$).
Since the morphisms $(1_{\psi_{i|U_{ijk}}}, \beta_{ijk})$ lie in the image of the action functor $\sfP \times \sfA \to \sfP \times_{\ul{M}} \sfP$, there are unique isomorphisms
\begin{equation}
	\alpha_{ijk} \colon \sfa_{ij|U_{ijk}} \otimes_\sfA \sfa_{jk|U_{ijk}} \to \sfa_{ik|U_{ijk}}
\end{equation}
in $\sfA_{|U_{ijk}}$, which satisfy the required coherence condition over quadruple overlaps by the fact that they are constructed as Cartesian lifts of identity morphisms.
Hence these data assemble into an $\sfA$-valued \v{C}ech 1-cocycle on
$M$ with respect to the open covering $\Ua$.
One can check that other choices of coverings and sections lead to 1-cocycles that become equivalent to $(\sfa_{ij}, \alpha_{ijk})$ when passing to a common refinement of good open coverings.

\begin{definition}
\label{def:String as smooth 2Grp}
Let $G$ be a compact simply-connected Lie group, and let $\sfA$ be a string-admissible smooth 2-group.
A \emph{smooth 2-group model} for $\String(G)$ is a smooth 2-group extension
\begin{equation}
\begin{tikzcd}
	1 \ar[r] & \sfA \ar[r] & \String(G) \ar[r] & \ul{G} \ar[r] & 1
\end{tikzcd}
\end{equation}
such that the principal 2-bundle $\String(G) \to \ul{G}$ represents a generator of $\rmH^3(G;\Z) \cong \Z$ under the isomorphism \smash{$\cH{}^1(G;\sfA) \cong \rmH^3(G;\Z)$}.
\end{definition}

With these definitions we have

\begin{theorem}
\label{st:HLB^G is admissible}
For any 2-connected manifold $M$, the smooth 2-group $\HLBdl^M$ is string-admissible.
\end{theorem}

\begin{theorem}
\label{st:Sym and Des as String models}
Let $G$ be a compact simply-connected Lie group, and let $\Ga \in
\BGrb(G)$ be a basic bundle gerbe.
Let $\Sym_G(\Ga)$ and $\Des_\sfL$ be the smooth 2-group extensions of $\ul{G}$ by $\HLBdl^G$ constructed from $\Ga$ with respect to the left action of $G$ on itself by left multiplication.
Then both $\Sym_G(\Ga)$ and $\Des_\sfL$ are smooth 2-group models for $\String(G)$ in the sense of Definition~\ref{def:String as smooth 2Grp}.
\end{theorem}

\begin{remark}
Note that there is a model for $\String(G)$ based on regression and transgression of the basic gerbe on $G$~\cite{Waldorf:String_2-group}.
Similarly, our model for $\String(G)$ which we obtain from $\Des_\sfL$ also heavily relies on the transgression-regression formalism.
However, the resulting models are very different:
Waldorf's model in~\cite{Waldorf:String_2-group} \emph{is} the regression of the transgression of the basic gerbe.
Waldorf observes that this gerbe picks up a multiplicative structure upon application of $\Ra \circ \Ta$.
Since any gerbe on a manifold can be seen as an extension of $M$ by $\rmB \U(1)$ as Lie groupoids, the resulting gerbe provides a string group model in the sense of~\cite{SP:String_group}.
Our construction differs from this significantly not only in the fact that we work with smooth 2-groups instead of Lie 2-groups; most notably, we obtain an extension of $G$ by the much larger smooth 2-group $\rmB (\ul{\U(1)^H})$, which is not equivalent to $\rmB \ul{\U(1)}$ as a smooth 2-group, but only on the level of their underlying homotopy types (see Theorem~\ref{st:HLB^G is admissible} as well as~\cite{Bunk:String_Grp}).
\qen
\end{remark}

The rest of this section is devoted to the proofs of Theorems~\ref{st:HLB^G is admissible} and~\ref{st:Sym and Des as String models}.
We begin with a few results that will combine to prove that $\HLBdl^M$ is string-admissible.
Then we will prove Theorem~\ref{st:Sym and Des as String models} by
observing that the $\HLBdl^G$-valued \v{C}ech 1-cocycles we obtain
from the 2-bundle $\Sym_G(\Ga)$ agree with those obtained from local
trivialisations of the bundle gerbe $\Ga$.

\begin{lemma}
\label{st: homotopy equivalence}
Let $M$ be a simply-connected manifold and $x\in M$ a fixed base point. 
Then the evaluation map $\ev_x \colon \U(1)^M \longrightarrow \U(1)$, $g \mapsto g(x)$ and the inclusion $\sfc \colon \U(1) \longrightarrow \U(1)^M$ as constant maps form a homotopy equivalence of diffeological spaces.\footnote{See Example~\ref{eg:Dfg examples} for the definition of the diffeological mapping space.}
\end{lemma}

\begin{proof}
Since $\ev_x \circ \sfc = 1_{\U(1)}$, it is enough to construct a smooth homotopy $\sfc \circ \ev_x\cong 1_{\U(1)^M}$.
Let $\pi \colon \R \longrightarrow \U(1)$, $r \mapsto \exp(2\pi\, \iu\, r)$ be the universal cover of $\U(1)$.
The assumption that $M$ is simply-connected implies (see e.g. \cite[Theorem 4.1]{bredon}) that 
the diagram
\begin{equation}
\begin{tikzcd}[row sep=1cm]
 & \R \ar[d,"\pi"] \\ 
 M \ar[r,"f"] & \U(1)
\end{tikzcd}
\end{equation}  
admits a unique continuous lift $\widehat{f}\colon M \longrightarrow \R$ for arbitrary $f\in\U(1)^M$ after fixing the lift at one point. We verify that the map $\widehat{f}$ is smooth. Fix a point $y \in M$ and a sufficiently small
open neighbourhood $U_y$ of $y$ which we identify with an open subset
of $\R^d$, where $d=\dim(M)$. 
The restriction of $f$ to $U_y$ is a plot of $\U(1)$.
Hence by Proposition~\ref{Prop: quotients G and diffeology} it admits a smooth lift \smash{$\widehat{f}_y\colon U_y\longrightarrow \R$} for sufficiently small 
$U_y$. From the uniqueness statement for lifts we obtain $\widehat{f}_{|U_y} = \widehat{f}_y +r_y$ for a fixed integer 
$r_y\in \Z$. This implies that $\widehat{f}$ is smooth and hence 
shows that the map \smash{$\pi^M:\R^M \longrightarrow \U(1)^M$}
is surjective. 

Consider the commutative diagram 
\begin{equation}
\begin{tikzcd}
\R \ar[rr,shift left=1, "\widehat{\sfc}"] \ar[dd,swap,"\pi"] & & \R^{M} \ar[ll,shift left=1, "\widehat{\ev_x}"] \ar[dd,"\pi^M"] \\  
 & & \\
\U(1) \ar[rr,shift left=1, "{\sfc}"] & & \U(1)^M \ar[ll,shift left=1, "{\ev}_x"]
\end{tikzcd}
\end{equation}
where $\widehat{\sfc}\,(r)(y) = r$ for all $y \in M$, and where
$\widehat{\ev_x}(g') = g'(x)$ for all $g'\in\R^M$.
We define a homotopy $\widehat{h} \colon 1_{\R^M}\longrightarrow \widehat{\sfc} 
\circ \widehat{\ev_x}$ by setting $\widehat{h}_t(g')(y) = g'(y) \,(1-t) + g'(x)\, t$. 
This homotopy descends to the desired homotopy $h\colon 1_{\U(1)^M} \longrightarrow 
\sfc \circ \ev_x$.

We verify that the homotopy $h$ is smooth: let $n$ be a natural number, $c \cong \R^n$ a Cartesian space and $f\colon c \longrightarrow \U(1)^M\times [0,1]$ a plot; that is, the maps
$\pr_{[0,1]} \circ f\colon c \longrightarrow [0,1]$ and $(\pr_{\U(1)^M}\circ f)^{\dashv}\colon 
c \times M \longrightarrow \U(1)$ are smooth.
We have to show 
that $h\circ f \colon c \longrightarrow \U(1)^M$ is a plot. By the arguments above we can lift $(\pr_{\U(1)^M}\circ f)^{\dashv}$
to a smooth map $c \times M \longrightarrow \R$ because $c\times M $ is simply-connected. This implies that we can lift $f$ to a smooth map \smash{$\widehat{f}\colon c \longrightarrow\R^M\times [0,1]$} making the diagram
\begin{equation}
\begin{tikzcd}[row sep=1cm]
 & \R^M\times [0,1] \ar[r,"\widehat{h}"] \ar[d,"\pi^M\times 1"] & \R^M \ar[d,"\pi^M"] \\
 c \ar[ur, "\widehat{f}"] \ar[r,"f",swap] & \U(1)^M\times [0,1] \ar[r,"h",swap] & \U(1)^M
\end{tikzcd}
\end{equation}
commute. The result now follows since $\widehat{h}$ and the projection $\pi^M:\R^M\longrightarrow \U(1)^M$ are smooth. 
\end{proof}

\begin{lemma}
\label{st:BU(1)^M and HLB^M}
There is an inclusion $\sfB (\ul{\U(1)^M}) \hookrightarrow \HLBdl^M$ of smooth 2-groups which is given by sending $c \in \sfB (\ul{\U(1)^M})$ (cf. Definition~\ref{def:BH as smooth groupoid}) to the trivial line bundle on $c \times M$.
If $\rmH^2(M;\Z) = 0$, then this inclusion is an equivalence.
\end{lemma}

\begin{proof}
We readily see that the inclusion respects the group structures, and that it is fully faithful.
If $\rmH^2(M;\Z) = 0$, then $\HLBdl^M{}_{|c} \cong \HLBdl(c \times M)$ is connected, so that in this case the inclusion is also fully faithful on all fibres.
Thus it is an equivalence on every fibre and hence an equivalence in
the 2-category $\Hscr$ by~\cite[Proposition~3.36]{Vistoli:Fib_Cats}.
\end{proof}

Combining Lemmas~\ref{st: homotopy equivalence} and~\ref{st:BU(1)^M and HLB^M}, we conclude that $\HLBdl^M$ is string-admissible for any 2-connected manifold $M$; that is, we have proven Theorem~\ref{st:HLB^G is admissible}.

For the diffeological group $H = \U(1)^M$ and a 2-connected manifold
$M$, there is an explicit isomorphism $\cH{}^k(X;\U(1)^M) \cong
\cH{}^k(X;\U(1))$ for $k>0$, for any manifold $X$, given by

\begin{proposition}
\label{st: Iso in cohomology}
Let $M$ be a 2-connected manifold with a fixed base point $x \in M$.
For any manifold $X$, evaluation at $x \in M$ induces an isomorphism
$\cH{}^k(X; \U(1)^M) \cong \cH{}^k(X; \U(1))$ for $k>0$ of {\v C}ech cohomology
groups with coefficients in the sheaves of smooth $\U(1)^M$-valued and $\U(1)$-valued functions, respectively.
\end{proposition}

\begin{proof}
Consider the sequence of diffeological groups 
\begin{align}
1 \longrightarrow \Z \longrightarrow \R^M \longrightarrow \U(1)^M
  \longrightarrow 1 \ ,
\end{align}
which is exact by the argument from the proof of Lemma~\ref{st: homotopy equivalence}.
The sheaf $\R^M$ admits a partition of unity by picking a partition of unity for the
sheaf of smooth $\R$-valued functions and a constant extension to $\R^M$-valued functions; hence $\cH{}^k(X; \R^M)=0$ for any manifold $X$ and for any $k\geq 1$. 
Now the statement follows from applying the five lemma to the diagram 
\begin{equation}
\begin{tikzcd}[row sep=1cm]
\cH{}^{k+1}(X; \R^M) \ar[r]\ar[d] & \cH{}^{k+1}(X;\Z) \ar[r]\ar[d] & \cH{}^{k}(X; \U(1)^M) \ar[r] \ar[d] & \cH{}^k(X; \R^M) \ar[r]\ar[d]& \cH{}^k(X; \Z) \ar[d] \\
\cH{}^{k+1}(X; \R) \ar[r] & \cH{}^{k+1}(X;\Z) \ar[r] & \cH{}^{k}(X; \U(1)) \ar[r] & \cH{}^k(X; \R) \ar[r] & \cH{}^k(X; \Z)
\end{tikzcd}
\end{equation}
induced by the long exact sequence in sheaf cohomology and the evaluation at $x \in M$. 
\end{proof}

It remains to determine the \v{C}ech cohomology class in $\cH{}^1(G;\HLBdl^G) \cong \rmH^3(G;\Z) \cong \Z$ determined by the extension $\Sym_G(\Ga) \to \ul{G}\,$.
The isomorphisms from Lemma~\ref{st:BU(1)^M and HLB^M} and
Proposition~\ref{st: Iso in cohomology} are useful in achieving
this. 
From the smooth 2-group extension $\Sym_G(\Ga) \longrightarrow \ul{G}$ we can 
extract a {\v C}ech 2-cocycle on $G$ with values in the sheaf of smooth $\U(1)^G$-valued 
functions.
To construct it, we first follow the procedure of the paragraph
preceding Definition~\ref{def:String as smooth 2Grp} to extract
$\HLBdl^G$-valued cocycle data and then choose local trivialisations of the line bundles which comprise it (which amounts to choosing an inverse for the equivalence from Lemma~\ref{st:BU(1)^M and HLB^M}).

Let $\Ua = \{ U_i\}_{i \in I}$ be a good open cover 
of $G$, let $\pi_i \colon U_i \times G \longrightarrow G$ denote the projection 
onto the second factor, and let $m_i\colon U_i \times G \longrightarrow G$ be the
multiplication map restricted to $U_i \times G$. We choose and fix 1-isomorphisms 
$\psi_i \colon m_i^* \mathcal{G} \longrightarrow \pi_i^* \mathcal{G}$
along with 
adjoint inverses
$\psi_i^{-1}$, which induce
equivalences $\Sym_G (\mathcal{G})_{|U_i}\cong \HLBdl(U_i\times G)$ of groupoids.

On double intersections $U_{ij}$ we can form the automorphism \smash{$\psi_{ij}\coloneqq \psi^{-1}_{j|U_{ij}} \circ \psi_{i|U_{ij}}$} of $m_{ij}^*\mathcal{G}$.
The isomorphism $\psi_{ij}$ can be identified with a line bundle $L_{ij}$ on 
$U_{ij}\times G$.
Since \smash{$\rmH^2(G;\Z)=0$}, we can choose and fix a trivialisation of $L_{ij}$ for each $i,j \in I$.

On triple intersections $U_{ijk}$ we get a 2-isomorphism
$\psi_{ijk}\colon \psi_{jk|U_{ijk}} \circ \psi_{ij|U_{ijk}}
\longrightarrow \psi_{ik|U_{ijk}} $ inducing an isomorphism 
$L_{jk} \otimes L_{ij}\longrightarrow L_{ik}$ of line bundles over
${U_{ijk}}\times G$.
(In contrast to the construction in the paragraph above Definition~\ref{def:String as smooth 2Grp}, here the isomorphisms $\psi_{ijk}$ can be obtained directly from the choice of $\psi_{ij}$.)
Using the trivialisations of these line bundles 
we obtain a smooth map ${U_{ijk}}\times G \longrightarrow \U(1)$ or equivalently
a map $c_{ijk}\colon U_{ijk} \longrightarrow \U(1)^G$. The collection $c_{ijk}$ form 
a smooth $\U(1)^G$-valued {\v C}ech 2-cocycle.

The corresponding cohomology class is independent of all choices involved:  
let $\psi_i' \colon m_i^* \mathcal{G} \longrightarrow \pi_i^* \mathcal{G}$ be a different
set of 1-isomorphisms. The automorphism $\psi_i^{-1} \circ \psi_i'$ of 
$ m_i^* \mathcal{G}$ can be identified 
with a line bundle ${\mit\Lambda}_i$ over $U_i \times G$. The definition of $L_{ij}$
implies $L_{ij}\otimes {\mit\Lambda}_i \cong {\mit\Lambda}_j \otimes L'_{ij}$. 
We can pick once and for all trivialisations of all bundles involved to identify this 
morphism with a function $A_{ij}\colon U_{ij}\times G \longrightarrow \U(1)$. 
The diagram 
\begin{equation}
\begin{tikzcd}[row sep=1cm]
L_{ij}\otimes {\mit\Lambda}_i \otimes L_{jk} \otimes {\mit\Lambda}_j \ar[d] \ar[r] &  {\mit\Lambda}_j \otimes L_{ik} \otimes {\mit\Lambda}_i \ar[d]  \\ 
{\mit\Lambda}_j\otimes {\mit\Lambda}_k \otimes L'_{ij} \otimes L'_{jk} \ar[r] & {\mit\Lambda}_j\otimes {\mit\Lambda}_k \otimes L'_{ik}
\end{tikzcd}
\end{equation}
commutes over $U_{ijk}$, which follows from the fact that all inverses were chosen 
to be adjoint so that the corresponding diagram involving
$\psi_i$ and $\psi_i'$ commutes.
Applying the trivialisations we get 
\begin{equation}
c_{ijk}\, A_{ik}= c'_{ijk}\, A_{ij}\, A_{jk} \ .
\end{equation}
This argument also shows that the cocycles define the same cohomology class if $\psi_i'=\psi_i$ 
and only the trivialisations of $L_{ij}$ differ.

The image of $c_{ijk}$ in $\cH{}^2(G; \U(1)^G)$ under the isomorphism 
$\cH{}^2(G; \U(1)^G) \cong \cH{}^2(G; \U(1))$ of Proposition~\ref{st: Iso in cohomology}
can be computed by restricting each $\psi_i$ to $U_i\times \{ e\} \subset U_i\times G $; the restriction $\psi_{|U_i\times \{ e\}}$ is a 1-isomorphism $\mathcal{G}_{|U_i} \longrightarrow \mathcal{G}|_e $. After fixing once and for all a trivialisation of $\mathcal{G}_{|e}$, this is just a trivialisation of $\mathcal{G}_{|U_i}$. This shows that the 
image of $c_{ijk}$ in $ \cH{}^2(G; \U(1))\cong\rmH^3(G;\Z)$ agrees with the cocycle $c_\mathcal{G}$ 
classifying the bundle gerbe $\mathcal{G}$, which proves Theorem~\ref{st:Sym and Des as String models}.

\begin{remark}
The arguments involving cocycles can be adjusted to the simpler case
of principal bundles over the Lie group $G$. In that case, starting from a principal $\U(1)$-bundle $P$ on $G$ we 
get a principal $\U(1)^G$-bundle $\Sym_G(P)$ on $G$ which is homotopy equivalent
to $P$.
The homotopy equivalence is induced by the maps $\ev_x$ and $\sfc$ from Lemma~\ref{st: homotopy equivalence}.
We can iterate the procedure to get larger and larger groups $\Sym_G( \cdots \Sym_G(\Sym_G(P))\cdots )$.
However, these groups are all topologically equivalent,
so that iterating the procedure does not produce anything that is topologically novel.
\qen  
\end{remark}

\subsection*{Acknowledgements}

We thank Jouko Mickelsson and Birgit Richter for helpful discussions
and correspondence.
This work was supported by the COST Action MP1405 ``Quantum Structure of Spacetime'', funded by the European Cooperation in Science and Technology (COST).
The work of S.B. was partially supported by the RTG~1670 ``Mathematics Inspired by String Theory and Quantum Field Theory''.
The work of L.M. was supported by the Doctoral Training Grant ST/N509099/1 from the UK Science and Technology Facilities Council (STFC).
The work of R.J.S. was supported in part by the STFC Consolidated
Grant ST/P000363/1 ``Particle Theory at the Higgs Centre''.

\bigskip

\begin{appendix}

\section{Properties of smooth principal 2-bundles}
\label{app:Principal 2-bundles}

\subsection{Surjectivity on objects and homotopy pullbacks}

Here we provide some technical background on smooth groupoids, as introduced in Definition~\ref{def:Hscr}.

\begin{lemma}
\label{st:surjectivity lemma}
Let $\pi \colon \sfX \to \Cart$ and $\pi' \colon \sfP \to \Cart$ be objects in $\Hscr$.
\begin{myenumerate}
\item Either $\sfX = \varnothing$ or $\pi$ is surjective on objects.

\item Let $p \colon \sfP \to \sfX$ be a morphism in $\Hscr$ whose
  underlying functor is an essentially surjective Grothendieck fibration.
Then $p$ is surjective on objects.
\end{myenumerate}
\end{lemma}

\begin{proof}
To see~(1), observe that $\Cart$ has a terminal object $* \in \Cart$.
Thus since $\pi$ is a Grothendieck fibration, if $\sfX_{|*} = \pi^{-1}(*)$ is non-empty then so is $\sfX_{|c}$ for any $c \in \Cart$.
For any $c \in \Cart$ there exists a morphism $x \colon * \to c$ in $\Cart$ given by choosing any point $x \in c$.
It follows that as soon as $\sfX \neq \varnothing$, it has only non-empty fibres over $\Cart$.

Claim~(2) follows from the general observation that a Grothendieck
fibration is essentially surjective if and only if it is surjective on objects.
\end{proof}

We now consider the setup of Definition~\ref{def:hopb of fibd Grpds}.

\begin{lemma}
\label{st:GrFib Lemma}
Let $\Cscr$ be a category, let $\pi_i \colon \sfD_i \to \Cscr$, for $i = 0,1$, and $\pi_\sfE \colon \sfE \to \Cscr$ be Grothendieck fibrations in groupoids, and let $F_i \colon \sfD_i \to \sfE$, for $i=0,1$, be morphisms of categories fibred in groupoids over $\Cscr$.
\begin{myenumerate}
\item $(\sfD_0 \times_{\sfE}^{\tt h} \sfD_1, \pi_{\tt h}) \in \Hscr$,
  i.e.~$\pi_{\tt h}$ is a Grothendieck fibration in groupoids.

\item Any morphism $G = (G_0, G_\sfE, G_1)$ of diagrams
\begin{equation}
\begin{tikzcd}[row sep=1cm]
       & & \sfE \ar[ddd, "G_\sfE"] & & \\
	\sfD_0 \ar[urr, "F_0"] \ar[d, swap,"G_0"] & & & & \sfD_1 \ar[ull, "F_1"'] \ar[d, "G_1"]
	\\
	\sfD'_0 \ar[drr, "F'_0"'] & & & & \sfD'_1 \ar[dll, "F'_1"] \\ 
        & & \sfE' & &
\end{tikzcd}
\end{equation}
in $\Hscr$, where all vertical morphisms are equivalences, induces an
equivalence $\sfD_0 \times_{\sfE}^{\tt h} \sfD_1 \to \sfD'_0
\times_{\sfE'}^{\tt h} \sfD'_1$.

\item If $F_1$ (resp.~$F_0$) is a Grothendieck fibration in groupoids,
  then the inclusion $\sfD_0 \times_{\sfE} \sfD_1 \hookrightarrow
  \sfD_0 \times_{\sfE}^{\tt h} \sfD_1$ is an equivalence, and the
  projections $\pr_0^{\tt h} \colon \sfD_0 \times_{\sfE}^{\tt h}
  \sfD_1 \to \sfD_0$ and $\pr_0 \colon \sfD_0 \times_{\sfE} \sfD_1 \to
  \sfD_0$ (resp.~the projections $\pr_1^{\tt h}$ and $\pr_1$ to $\sfD_1$) are Grothendieck fibrations in groupoids.
\end{myenumerate}
\end{lemma}

\begin{proof}
To prove (1), we first show that every morphism in $\sfD_0
\times_{\sfE}^{\tt h} \sfD_1$ is $\pi_{\tt h}$-Cartesian.
Consider morphisms $(\psi_0,\psi_1) \colon (d'_0,\eta',d'_1) \to
(d''_0,\eta'',d''_1)$ and $(\varphi_0, \varphi_1) \colon (d_0, \eta,
d_1) \to (d''_0, \eta'', d''_1)$ in $\sfD_0 \times_{\sfE}^{\tt h} \sfD_1$.
By assumption $\pi_0 (d_0) = \pi_1(d_1)$ in $\Cscr$ (and analogously for $d'_i, d''_i$), and $\pi_0(\psi_0) = \pi_1(\psi_1)$ (and analogously for $\varphi_i$).
Let $\chi \colon \pi_0(d_0) \to \pi_0(d'_0)$ be a morphism in $\Cscr$ such that $\pi_0(\psi_0) \circ \chi = \pi_0(\varphi_0)$.
Since $\pi_\sfE$ is a Grothendieck fibration in groupoids, there exists a unique lift $\chi_{\sfE,i} \colon F_i(d_i) \to F_i(d'_i)$ induced by each of the pairs $(F_i(\psi_i), F_i(\varphi_i))$ of morphisms in $\sfE$.
Similarly, the pairs $(\psi_i,\varphi_i)$ induce unique lifts
$\chi_i:d_i\to d_i'$ of $\chi$ to $\sfD_i$ along $\pi_i$.
The uniqueness of lifts (along $\pi_\sfE$) implies that $F_i(\chi_i) = \chi_{\sfE,i}$ for $i = 0,1$.
It remains to show that $(\chi_0, \chi_1)$ defines a morphism $(d_0, \eta, d_1) \to (d'_0, \eta', d'_1)$.
That is, we need to prove that $F_1(\chi_1) \circ \eta = \eta' \circ F_0(\chi_0)$ in $\sfE$.
So far we have a diagram
\begin{equation}
\begin{tikzcd}
	& F_0(d_0) \ar[dl, "F_0(\chi_0)"'] \ar[dr, "F_0(\varphi_0)"] \ar[dd, "\eta" pos=0.25] &
	\\
	F_0(d'_0) \ar[dd, "\eta'"'] \ar[rr, crossing over, "F_0(\psi_0)"' pos=0.25] & & F_0(d''_0) \ar[dd, "\eta''"]
	\\
	& F_1(d_1) \ar[dl, "F_1(\chi_1)"'] \ar[dr, "F_1(\varphi_1)"] &
	\\
	F_1(d'_1) \ar[rr, "F_1(\psi_1)"'] & & F_1(d''_1)
\end{tikzcd}
\end{equation}
in $\sfE$, where each face of this diagram commutes, apart from the back left square.
The commutativity of that square is what we need to prove.
For this, observe that both $\eta' \circ F_0(\chi_0)$ and $\eta \circ F_1(\chi_1)$ provide lifts of $\chi$ to $\sfE$ with respect to the morphisms $\eta'' \circ F_0(\varphi_0) \colon F_0(d_0) \to F_1(d''_1)$ and $F_1(\psi_1) \colon F_1(d'_1) \to F_1(d''_1)$.
The desired identity now follows from the uniqueness of such lifts along the functor $\pi_\sfE$.

Next we need to show that for any morphism $f \colon c \to c'$ in
$\Cscr$ and any object $(d'_0, \eta', d'_1)$ in $\sfD_0
\times_{\sfE}^{\tt h} \sfD_1$ over $c'$, there exists a lift
$\widehat{f} = (f_0,f_1)$ of $f$ to $\sfD_0 \times_{\sfE}^{\tt h} \sfD_1$ with codomain $(d'_0, \eta', d'_1)$.
Such a lift is obtained by lifting $f$ to morphisms $f_i \colon d_i
\to d'_i$ in $\sfD_i$ using the fact that $\pi_i$ is a Grothendieck fibration in groupoids, for $i = 0,1$.
An isomorphism $\eta \colon F_0(d_0) \to F_1(d_1)$ compatible with
$f_0, f_1$ is obtained by filling the horn given by the morphisms
$\eta'\circ F_0(f_0)$ and $F_1(f_1)$ over the identity morphism $1_c$
in $\Cscr$. The filler is an isomorphism since the fibre $\sfE_{|c}$
is a groupoid. 

To prove~(2), we note that
by~\cite[Proposition~3.36]{Vistoli:Fib_Cats} the induced morphism $G_0
\times_{G_\sfE}^{\tt h} G_1$ is an equivalence in $\Hscr$ if and only
if it restricts to an equivalence of groupoids between all fibres of
$\pi_{\tt h}$ and $\pi'_{\tt h}$.
A direct inspection on any $c\in\Cscr$ reveals that
\begin{equation}
	\pi_{\tt h}^{-1}(c) = \pi_0^{-1}(c)
        \times_{\pi_\sfE^{-1}(c)}^{\tt h} \pi_1^{-1}(c)
\end{equation}
as groupoids, and it is well-known that equivalences of spans of groupoids induce equivalences on homotopy pullbacks of groupoids.

To prove~(3), we first show that $\pr_0^{\tt h}$ is a Grothendieck fibration in groupoids.
Consider a span
\begin{equation}
\begin{tikzcd}
        & (d''_0, \eta'', d''_1) & \\
	(d'_0, \eta', d'_1) \ar[ur, "{(\psi_0, \psi_1)}"] & & (d_0, \eta, d_1) \ar[ul, "{(\varphi_0, \varphi_1)}"']
\end{tikzcd}
\end{equation}
in $\sfD_0 \times_{\sfE}^{\tt h} \sfD_1$, and a morphism $\chi_0 \colon d_0 \to d'_0$ in $\sfD_0$ such that $\psi_0 \circ \chi_0 = \varphi_0$.
We obtain a commutative triangle in $\sfE$ formed by the morphisms $F_1(\psi_1)$, $F_1(\varphi_1)$, and $\eta' \circ F_0(\chi_0) \circ \eta^{-1}$.
Since $F_1$ is a Grothendieck fibration in groupoids, this gives rise
to a unique lift $\chi_1:d_1\to d_1'$ of the latter morphism to $\sfD_1$.
The pair $(\chi_0, \chi_1)$ is automatically a morphism in $\sfD_0
\times_{\sfE}^{\tt h} \sfD_1$ which projects to $\chi_0$.
If $(\chi'_0, \chi'_1)$ were any other such filling over $\chi_0$ of the horn given by $(\psi_0, \psi_1)$ and $(\varphi_0, \varphi_1)$, it would immediately follow that $\chi_0' = \chi_0$, and the uniqueness of fillings of $\psi_1, \varphi_1$ over $\eta' \circ F_0(\chi_0) \circ \eta^{-1}$ would imply that $\chi_1' = \chi_1$.

Let $\varphi_0 \colon d_0 \to d'_0$ be a morphism in $\sfD_0$, and let
$(d'_0, \eta', d'_1)$ be an object in $\sfD_0 \times_{\sfE}^{\tt h} \sfD_1$ which projects to $d'_0$.
Let $\varphi_1$ be a lift along $\pi_1$ of $f \coloneqq \pi_0(\varphi_0)$ with codomain $d'_1$.
The pair of morphisms $(\eta' \circ F_0(\psi_0), F_1(\psi_1))$ then gives rise to a cospan in $\sfE$.
Both morphisms project to $f$ in $\Cscr$ and hence, since $\pi_\sfE$
is a Grothendieck fibration in groupoids, there exists a unique
isomorphism $\eta \colon F_0(d_0) \to F_1(d_1)$ such that $(d_0, \eta,
d_1) \in \sfD_0 \times_{\sfE}^{\tt h} \sfD_1$ and such that $(\psi_0,
\psi_1)$ is a lift of $\psi_0$ to $\sfD_0 \times_{\sfE}^{\tt h} \sfD_1$ with codomain $(d'_0, \eta', d'_1)$.
The claim for $\pr_0$ is proven in an entirely analogous way by
restricting $\eta$, $\eta'$ and $\eta''$ to be identity morphisms.

Finally, consider the inclusion functor $\sfD_0 \times_{\sfE} \sfD_1
\hookrightarrow \sfD_0 \times_{\sfE}^{\tt h} \sfD_1$.
Since $\pr_0^{\tt h}$ is a Grothendieck fibration in groupoids, so is its restriction to each fibre over $\Cscr$.
It is well-known that the inclusion of a pullback of groupoids into the homotopy pullback is an equivalence in case one of the functors in the diagram is a Grothendieck fibration.
Thus our inclusion functor is an equivalence on each fibre over $\Cscr$, whence the result follows by~\cite[Proposition~3.36]{Vistoli:Fib_Cats}.
\end{proof}

\subsection{Relation to principal $\infty$-bundles}
\label{app:2-buns are oo-buns}

Our notion of smooth principal 2-bundle does not have any notion of `local triviality' built into it.
This differs from the version of a principal 2-bundle defined in~\cite{SP:String_group}, but is very much in the spirit of the definition of a principal $\infty$-bundle from~\cite{NSS:oo-bdls_I}.
The fact that we require essential surjectivity is our version of saying that the (homotopy) fibres of the bundle should be non-empty.
In contrast to~\cite{NSS:oo-bdls_I} we have to require fibration properties because we do not work purely in an $\infty$-categorical framework.
We shall now show that an $\sfH$-principal 2-bundle in $\Hscr$ in the sense of Definition~\ref{def:sfH-bundle} gives rise to a principal 2-bundle in the sense of~\cite[Definition~3.4]{NSS:oo-bdls_I}, adapted from a general $\infty$-topos (described e.g.~by presheaves of $\infty$-groupoids) to our situation involving presheaves of groupoids.
Let $p \colon \sfP \to \sfX$ be a morphism in $\Hscr$, and let $\sfP^{[\bullet]}$ be the \v{C}ech nerve of $p$.
We write $\hocolim^\Cscr$ (resp.~$\holim^\Cscr$) for a homotopy
colimit (resp. limit) taken in a simplicial model category $\Cscr$.

\begin{proposition}
\label{st:effective epi from surj GrFib}
Every morphism $p \colon \sfP \to \sfX$ in $\Hscr$ whose underlying functor is an essentially surjective Grothendieck fibration in groupoids gives rise to an effective epimorphism: the morphism
\begin{equation}
	\underset{\Delta^{\rm op}}{\hocolim}^\Hscr \, \sfP^{[\bullet]} \to \sfX
\end{equation}
from its \v{C}ech nerve to $\sfX$ is an equivalence.
\end{proposition}

Because of Lemma~\ref{st:GrFib Lemma} and the assumption that $p$ is a
Grothendieck fibration in groupoids, it does not matter here if one uses
the coherent \v{C}ech nerve, formed using $\sfP \times_\sfX^{\tt h}
\cdots \times_\sfX^{\tt h} \sfP$, or the strict \v{C}ech nerve, formed using $\sfP \times_\sfX \cdots \times_\sfX \sfP$.

\begin{proof}
We work with Hollander's model structure on $\Hscr$~\cite{Hollander:HoThy_for_stacks}.
In this picture, $\Hscr$ is a model category enriched, tensored and cotensored in the model category $\Grpd$ (seen as a strict category).
In both $\Hscr$ and $\Grpd$ all objects are fibrant, and the functor $\ul{\Hscr} \colon \Hscr^\opp \times \Hscr \to \Grpd$ is homotopical by~\cite[Proposition~3.35]{Vistoli:Fib_Cats}, i.e.~it preserves weak equivalences in each argument.
The enrichment of $\Hscr$ in $\Grpd$ even enhances to an enrichment over $\sSet$, the category of simplicial sets with the Kan-Quillen model structure.
Thus homotopy (co)limits in $\Hscr$ can be computed using (co)bar constructions~\cite{riehl}.
Let $Q$ denote a cofibrant replacement functor in $\Hscr$, and let $\sfZ \in \Hscr$ be an arbitrary object.
Then
\begin{align}
	\ul{\Hscr} \Big( \underset{\Delta^{\rm op}}{\hocolim}^\Hscr \,
  \sfP^{[\bullet]} \,,\, \sfZ \Big)
\cong \underset{\Delta}{\holim}^\Grpd\ \ul{\Hscr} \big( Q(\sfP^{[\bullet]}),\, \sfZ \big)
	\cong \underset{\Delta}{\holim}^\Grpd\ \ul{\Hscr} \big(
  \sfP^{[\bullet]},\, \sfZ \big) \ ,
\end{align}
where the first equivalence stems from the fact that $\sfZ$ is fibrant and $\Hscr$ is a $\Grpd$-enriched model category, and the second equivalence stems from the fact that $\ul{\Hscr}$ is homotopical.
It thus suffices to prove that the functor
\begin{equation}
	p^* \colon \ul{\Hscr}(\sfX, \sfZ) \to \underset{\Delta}{\holim}^\Grpd\ \ul{\Hscr} \big( \sfP^{[\bullet]},\, \sfZ \big)
	\eqqcolon \Des_p(\sfZ)
\end{equation}
is an equivalence of groupoids.

An object in $\Des_p(\sfZ)$ is a pair $(G,\eta)$ of a functor $G
\colon \sfP \to \sfZ$ of categories fibred in groupoids over $\Cart$, together with a natural isomorphism $\eta_{|(p_0,p_1)} \colon G(p_0) \to G(p_1)$ from $d_1^*G$ to $d_0^*G$ of functors over $\Cart$, where $d_i$ are the face maps in the simplicial object \smash{$\sfP^{[\bullet]}$}.
This natural isomorphism is subject to the conditions $d_2^*\eta \circ
d_0^*\eta = d_1^*\eta$ over $\sfP^{[3]}$ and ${\mit\Delta}^*\eta=1_G$
over $\sfP$, where ${\mit\Delta}:\sfP\to\sfP^{[2]}$ is the diagonal map.
A morphism $(G,\eta) \to (G',\eta')$ in $\Des_p(\sfZ)$ is a natural isomorphism $\gamma \colon G \to G'$ in $\Hscr$ such that $\eta' \circ d_1^*\gamma = d_0^*\gamma \circ \eta$.

We first show that $p^*$ is essentially surjective:
let $(G,\eta) \in \Des_p(\sfZ)$ be any object.
We define a functor $F \colon \sfX \to \sfZ$ as follows:
first, recalling that $p$ is surjective on objects by
Lemma~\ref{st:surjectivity lemma}, we choose a section $s \colon {\tt
  ob}(\sfX) \to {\tt ob}(\sfP)$ of the map of objects defined by $p$.
Then we set $F(x) \coloneqq G(s(x)) \in \sfZ$ for $x\in\sfX$.
Now consider a morphism $\psi \colon x \to y$ in $\sfX$.
Since $p$ is a Grothendieck fibration, $\psi$ has a lift
$\widehat{\psi} \colon \widehat{x} \to s(y)$ to a morphism in $\sfP$
with codomain $s(y)$, where $p(\widehat{x})=x$.
Define $F(\psi) \colon F(x) \to F(y)$ via the diagram
\begin{equation}
\begin{tikzcd}[column sep=1.5cm, row sep=1cm]
	F(x) = G\big(s(x)\big) \ar[d, "\eta_{|(s(x),\widehat{x})}"', "\cong"] \ar[r, dashed, "F(\psi)"]
	& G\big(s(y)\big) =F(y)
	\\
	G(\widehat{x}) \ar[ur, "G(\widehat{\psi}\,)"'] &
\end{tikzcd}
\end{equation}
The naturality of $\eta$, together with the two conditions it
satisfies and the fact that $p$ is a Grothendieck fibration in groupoids, imply that $F$ is a well-defined functor.
Furthermore, $\eta$ establishes an isomorphism $p^*F = (F,1_F) \to (G,\eta)$ in $\Des_p(\sfZ)$.
Thus $p^*$ is essentially surjective.

That $p^*$ is fully faithful follows from its explicit construction and the fact that $p$ is essentially surjective.
\end{proof}

\end{appendix}

\bigskip

\newcommand{\etalchar}[1]{$^{#1}$}


\begin{thebibliography}{BDL{\etalchar{+}}11}
      
\bibitem[Alf20]{Alfonsi:DFT_and_KK}
L. Alfonsi.
\newblock{Global double field theory is higher Kaluza-Klein
  theory}. 
\newblock{\em Fortsch. Phys.}, 68: 2000010, 42~pp., 2020. 
\newblock \href {https://arxiv.org/abs/1912.07089} {\path{arXiv:1912.07089}}.

\bibitem[AS84]{AS}
M.{\,}F.~Atiyah and I.{\,}M.~Singer.
\newblock Dirac operators coupled to vector potentials.
\newblock{\em Proc. Natl. Acad. Sci. USA}, 81:2597--2600, 1984.

\bibitem[BCSS07]{BCSS07}
J.{\,}C. Baez, A. Crans, U. Schreiber, and D. Stevenson.
\newblock{From loop groups to 2-groups.}
\newblock{\em Homol. Homot. Appl.}, 9(2):101--135, 2007.
\newblock \href {https://arxiv.org/abs/math/0504123} {\path{arXiv:math/0504123}}.

\bibitem[BL04]{BL:2-Groups}
J.{\,}C.~Baez and A.{\,}D.~Lauda.
\newblock{Higher-dimensional algebra {V}. 2-groups.}
\newblock{\em Theory Appl. Categ.}, 12:423--491, 2004.
\newblock \href {http://arxiv.org/abs/math/0307200} {\path{arXiv:math/0307200}}.

\bibitem[BH11]{BH:Diffelogical spaces}
J.{\,}C.~Baez and A.{\,}E. Hoffnung.
\newblock{Convenient categories of smooth spaces.}
\newblock{\em Trans. Amer. Math. Soc.}, 363(11): 5789--5825, 2011.
\newblock \href {https://arxiv.org/abs/0807.1704v4} {\path{ arXiv:0807.1704}}.

\bibitem[BL14]{Bakas:2013jwa}
I.~Bakas and D.~L{\"u}st.
\newblock {3-cocycles, nonassociative star products and the magnetic paradigm
 of $R$-flux string vacua}.
\newblock {\em J. High Energy Phys.}, 01:171, 49~pp., 2014.
\newblock \href {http://arxiv.org/abs/1309.3172} {\path{arXiv:1309.3172}}.

\bibitem[BP11]{Blumenhagen:2010hj}
R.~Blumenhagen and E.~Plauschinn.
\newblock {Nonassociative gravity in string theory?}
\newblock {\em J. Phys. A}, 44:015401, 19~pp., 2011.
\newblock \href {http://arxiv.org/abs/1010.1263} {\path{arXiv:1010.1263}}.

\bibitem[BBBS15]{Bojowald:2014oea}
M.~Bojowald, S.~Brahma, U.~B{\"u}y{\"u}k\c{c}am, and T.~Strobl.
\newblock {States in nonassociative quantum mechanics: Uncertainty relations
 and semiclassical evolution}.
\newblock {\em J. High Energy Phys.}, 03:093, 22~pp., 2015.
\newblock \href {http://arxiv.org/abs/1411.3710} {\path{arXiv:1411.3710}}.

\bibitem[BMS19]{BMS:NA_translations}
S.~Bunk, L.~Müller, and R.{\,}J.~Szabo.
\newblock {Geometry and 2-Hilbert space for nonassociative magnetic translations.}
\newblock{\em Lett. Math. Phys.}, 109(8):1827--1866, 2019.
\newblock \href {http://arxiv.org/abs/1804.08953} {\path{arXiv:1804.08953}}.

\bibitem[BSS18]{BSS--HGeoQuan}
S.~Bunk, C.~Saemann, and R.{\,}J. Szabo.
\newblock The 2-{H}ilbert space of a prequantum bundle gerbe.
\newblock {\em Rev. Math. Phys.}, 30(1):1850001, 101 pp., 2018.
\newblock \href {http://arxiv.org/abs/1608.08455} {\path{arXiv:1608.08455}}.

\bibitem[BS17]{Bunk-Szabo--Fluxes_brbs_2Hspaces}
S.~Bunk and R.{\,}J. Szabo.
\newblock Fluxes, bundle gerbes and 2-{H}ilbert spaces.
\newblock {\em Lett. Math. Phys.}, 107(10):1877--1918, 2017.
\newblock \href {http://arxiv.org/abs/1612.01878} {\path{arXiv:1612.01878}}.

\bibitem[Bun17]{Bunk--Thesis}
S.~Bunk.
\newblock {{Categorical structures on bundle gerbes and higher geometric
  prequantisation}}.
\newblock PhD Thesis, Heriot-Watt University, Edinburgh, 139~pp., 2017.
\newblock \href {http://arxiv.org/abs/1709.06174} {\path{arXiv:1709.06174}}.

\bibitem[Bun20a]{Bunk:2020ifw}
S.~Bunk.
\newblock Sheaves of higher categories and presentations of smooth
field theories. Preprint, 50~pp., 2020.
\newblock \href {http://arxiv.org/abs/2003.00592} {\path{arXiv:2003.00592}}.

\bibitem[Bun20b]{Bunk:String_Grp}
S.~Bunk.
\newblock {Principal $\infty$-bundles and smooth string group models.}
Preprint, 44~pp., 2020.
\newblock \href {http://arxiv.org/abs/2008.12263} {\path{arXiv:2008.12263}}.

\bibitem[BW18]{BW:Transgression_of_D-branes}
S.~Bunk and K.~Waldorf.
\newblock{Transgression of D-branes.} Preprint, 69~pp., 2018.
\newblock \href {http://arxiv.org/abs/1808.04894} {\path{arXiv:1808.04894}}.

\bibitem[BW19]{BW:OCFFTs}
S.~Bunk and K.~Waldorf.
\newblock{Smooth functorial field theories from $B$-fields and
  D-branes.} Preprint, 64~pp., 2019.
\newblock \href {http://arxiv.org/abs/1911.09990} {\path{arXiv:1911.09990}}.

\bibitem[Bre93]{bredon}
G.~E. Bredon.
\emph{Topology and Geometry}.
Springer Graduate Texts in Mathematics, 1993.

\bibitem[CMM97]{Index}
A.{\,}L.~Carey, J.~Mickelsson, and M.{\,}K.~Murray.
\newblock {Index Theory, Gerbes, and Hamiltonian Quantization.}
\newblock {\em Comm. Math. Phys.}, 183:707--722, 1997.
\newblock \href {https://arxiv.org/abs/hep-th/9511151} {\path{arXiv:hep-th/9511151}}.
	
\bibitem[CMM00]{Carey1997}
A.{\,}L.~Carey, J.~Mickelsson, and M.{\,}K.~Murray.
\newblock {Bundle gerbes applied to quantum field theory.}
\newblock {\em Rev. Math. Phys.}, 12(1):65--90, 2000.
\newblock \href {https://arxiv.org/abs/hep-th/9711133} {\path{arXiv:hep-th/9711133}}.

\bibitem[CM95]{CM94}
A.{\,}L.~Carey and M.{\,}K.~Murray.
\newblock {Mathematical remarks on the cohomology of gauge groups and anomalies.}
\newblock In \emph{Confronting the Infinite}, pages 136–148. World Scientific Publishing, River Edge, NJ, 1995.
\newblock \href {https://arxiv.org/abs/hep-th/9408141} {\path{arXiv:hep-th/9408141}}.

\bibitem[CM96]{CM96}
A.{\,}L.~Carey and M.{\,}K.~Murray.
\newblock {Faddeev's anomaly and bundle gerbes.}
\newblock{\em Lett. Math. Phys.}, 37(1):29–36, 1996.

\bibitem[Col11]{Collier:Inft_Symmetries_of_gerbes}
B. L. Collier.
\newblock{Infinitesimal symmetries of Dixmier-Douady gerbes.}
Preprint, 2011.
\newblock \href {https://arxiv.org/abs/1108.1525} {\path{arXiv:1108.1525}}.

\bibitem[DGTS20]{Davighi}
J.~Davighi, B.~Gripaios and J.~Tooby-Smith,
\newblock{Quantum mechanics in magnetic backgrounds with manifest symmetry and locality}.
\newblock{\em J. Phys. A}, 53:145302, 35~pp., 2020
\newblock \href{https://arxiv.org/abs/1905.11999} {\path{arXiv:1905.11999}}

\bibitem[DHH11]{DHH:String_Orientations}
C.{\,}L.~Douglas, A.{\,}G.~Henriques, and M.{\,}A. Hill.
\newblock {Homological obstructions to string orientations},
{\em Int. Math. Res. Not.}, 18:4074--4088, 2011.
\newblock \href {https://arxiv.org/abs/0810.2131} {\path{arXiv:0810.2131}}.

\bibitem[Fad84]{Faddeev:Op_Anomaly}
L.{\,}D. Faddeev.
\newblock {Operator anomaly for the Gauss law}.
\newblock {\em Phys. Lett. B}, 145:81--84, 1984.

\bibitem[FS85]{FM:Nonabelian_Anomalies}
L.{\,}D. Faddeev and S.{\,}L. Shatashvili.
\newblock{Algebraic and Hamiltonian methods in the theory of nonabelian anomalies}.
\newblock {\em Theor. Math. Phys.}, 60:770--778, 1985.

\bibitem[Fio13]{Fiore}
G.~Fiore,
\newblock{On quantum mechanics with a magnetic field on $\R^n$ and on a torus $\bbT^n$.}
\newblock{\em Int.\ J.\ Theor.\ Phys.},  52:877--896, 2013.
\newblock \href {https://arxiv.org/abs/1103.0034} {\path{arXiv:1103.0034}}.

\bibitem[FRS16]{FRS:Higher_U(1)-connections}
D.~Fiorenza, C.{\,}L.~Rogers, and U.~Schreiber.
\newblock{Higher {$\U(1)$}-gerbe connections in geometric prequantisation.}
\newblock {\em Rev. Math. Phys.}, 28(6):1650012, 72~pp., 2016.
\newblock \href {https://arxiv.org/abs/1304.0236} {\path{arXiv:1304.0236}}.

\bibitem[GSW11]{GSW:Global_gauge_anomalies}
K.~Gaw\c{e}dzki, R.~Suszek, and K.~Waldorf.
\newblock{Global gauge anomalies in two-dimensional bosonic sigma models.}
\newblock {\em Comm. Math. Phys.}, 302(2):513--580, 2011.
\newblock \href {https://arxiv.org/abs/1003.4154} {\path{arXiv:1003.4154}}.

\bibitem[Gru00]{Gruber}
M.{\,}J.~Gruber.
\newblock{Bloch theory and quantisation of magnetic systems.}
\newblock{\em J. Geom. Phys.}, 34:137--154, 2000.
\newblock {\href {http://arxiv.org/abs/math-ph/9903048} {\path{arXiv:math-ph/9903048}}}.

\bibitem[GZ86]{Gunaydin:1985ur}
M.~G{\"u}naydin and B.~Zumino.
\newblock {Magnetic charge and nonassociative algebras}.
\newblock In {\emph{Old and New Problems in Fundamental Physics: Meeting in
 Honour of G.{\,}C. Wick}}, pages 43--53. Quaderni, Pisa: Scuola Normale
 Superiore, 1986.

\bibitem[Hol08]{Hollander:HoThy_for_stacks}
S.~Hollander.
\newblock{A homotopy theory for stacks.}
\newblock{\em Israel J. Math.}, 163:93--124, 2008.
\newblock {\href {http://arxiv.org/abs/math/0110247} {\path{arXiv:math/0110247}}}.

\bibitem[IZ13]{book Diffeology}
P. Iglesias-Zemmour.
\emph{Diffeology}.
 Mathematical Surveys and Monographs, 185. American Mathematical Society, 2013.
 
\bibitem[Jac85]{Jackiw:3-cocycles}
R.~Jackiw.
\newblock {3-cocycle in mathematics and physics}.
\newblock {\em Phys. Rev. Lett.}, 54:159--162, 1985.

\bibitem[Jo85]{Jo}
S.-G.~Jo.
\newblock{Commutators in an anomalous nonabelian chiral gauge theory}.
\newblock{\em Phys. Lett. B}, 163:353--359, 1985.

\bibitem[KS18]{KS:Symplectic_realisation}
V.{\,}G. Kupriyanov and R.{\,}J. Szabo.
\newblock {Symplectic realisation of electric charge in fields of monopole
 distributions}.
\newblock {\em Phys.\ Rev.\ D}, 98(4):045005, 25~pp., 2018.
\newblock \href {http://arxiv.org/abs/1803.00405} {\path{arXiv:1803.00405}}.

\bibitem[Lee13]{Lee}
J.{\,}M. Lee.
\emph{Introduction to Smooth Manifolds}, Second Edition.
Springer Graduate Texts in Mathematics,  2013.
	
\bibitem[Lur09]{Lurie:HTT}
J.~Lurie.
\newblock{Higher topos theory.}
\newblock{\em Ann. Math. Studies}, 170:1--925, 2009.
\newblock \href {http://arxiv.org/abs/math/0608040} {\path{arXiv:math/0806040}}.

\bibitem[L{\"u}s10]{Lust:2010iy}
D.~L{\"u}st.
\newblock {T-duality and closed string noncommutative (doubled) geometry}.
\newblock {\em J. High Energy Phys.}, 12:084, 28~pp., 2010.
\newblock \href {http://arxiv.org/abs/1010.1361} {\path{arXiv:1010.1361}}.

\bibitem[Mic85]{Mickelsson:1983xi}
J.~Mickelsson.
\newblock {Chiral anomalies in even and odd dimensions}.
\newblock {\em Commun. Math. Phys.}, 97:361--370, 1985.

\bibitem[Mic19]{Mickelsson}
J. Mickelsson.
\newblock{Nonassociative magnetic translations: A QFT construction.}
Preprint, 2019.
\newblock \href{https://arxiv.org/abs/1905.01944} {\path{arXiv:1905.01944}}.

\bibitem[MW16]{MW16}
J.~Mickelsson and S.~Wagner.
\newblock{Third group cohomology and gerbes over Lie groups.}
\newblock{\em J. Geom. Phys.}, 108:49--70, 2016.
\newblock \href{https://arxiv.org/abs/1602.02565} {\path{arXiv:1602.02565}}.

\bibitem[Moe02]{Moerdijk:Stacks_and_gerbes}
I.~Moerdijk.
\newblock{Introduction to the language of stacks and gerbes.}
Preprint, 2002.
\newblock \href {http://arxiv.org/abs/math/0212266} {\path{arXiv:math/0212266}}.

\bibitem[Mur96]{Murray--Bundle_gerbes}
M.{\,}K. Murray.
\newblock {Bundle gerbes}.
\newblock {\em J. London Math. Soc.}, 54:403--416, 1996.
\newblock \href {http://arxiv.org/abs/dg-ga/9407015}
{\path{arXiv:dg-ga/9407015}}.

\bibitem[M$^+$17]{MRSV:Equivariant_BGrbs}
M.{\,}K.~Murray, D.{\,}M.~Roberts, D. Stevenson, and R.{\,}F.~Vozzo.
\newblock {Equivariant bundle gerbes},
\newblock {\em Adv. Theor. Math. Phys.}, 21(4):921--975, 2017.
\newblock \href {http://arxiv.org/abs/1506.07931} {\path{arXiv:1506.07931}}.

\bibitem[MSS12]{MSS:NonGeo_Fluxes_and_Hopf_twist_Def}
D.~Mylonas, P.~Schupp, and R.{\,}J. Szabo.
\newblock {Membrane sigma-models and quantisation of non-geometric flux
  backgrounds}.
\newblock {\em J. High Energy Phys.}, 09:012, 55pp., 2012.
\newblock \href {http://arxiv.org/abs/1207.0926} {\path{arXiv:1207.0926}}.

\bibitem[MSS14]{Mylonas:2013jha}
D.~Mylonas, P.~Schupp, and R.{\,}J. Szabo.
\newblock {Non-geometric fluxes, quasi-Hopf twist deformations and
 nonassociative quantum mechanics}.
\newblock {\em J. Math. Phys.}, 55:122301, 30~pp., 2014.
\newblock \href {http://arxiv.org/abs/1312.1621} {\path{arXiv:1312.1621}}.

\bibitem[NSW13]{NSW:Smooth_string_group}
T.~Nikolaus, C.~Sachse, and C.~Wockel.
\newblock {A smooth model for the string group.}
\newblock {\em Int. Math. Res. Not.}, 16:3678--3721, 2013.
\newblock \href {http://arxiv.org/abs/1104.4288} {\path{arXiv:1104.4288}}.

\bibitem[NSS15]{NSS:oo-bdls_I}
T.~Nikolaus, U.~Schreiber, and D.~Stevenson.
\newblock{Principal {$\infty$}-bundles: General theory.}
\newblock{\em J. Homot. Relat. Struct.}, 10(4):749--801, 2015.
\newblock \href {http://arxiv.org/abs/1207.0248} {\path{arXiv:1207.0248}}.

\bibitem[NS11]{NS:Equivar}
T.~Nikolaus and C.~Schweigert.
\newblock{Equivariance in higher geometry.}
\newblock {\em Adv. Math.} 226(4):3367--3408, 2011.
\newblock \href {http://arxiv.org/abs/1004.4558}
{\path{arXiv:1004.4558}}.

\bibitem[Rie14]{riehl}
E. Riehl.
\emph{Categorical Homotopy Theory}.
New Mathematical Monographs, 24. Cambridge University Press, 2014.

\bibitem[SS20]{SS:Non-Ab_SD_String}
C.~Saemann and L.~Schmidt.
\newblock {The nonabelian self-dual string.}
\newblock {\em Lett. Math. Phys.}, 110:1001–1042, 2020.
\newblock \href {http://arxiv.org/abs/1705.02353} {\path{arXiv:1705.02353}}.

\bibitem[SP11]{SP:String_group}
C.{\,}J.~Schommer-Pries.
\newblock{Central extensions of smooth 2-groups and a finite-dimensional string 2-group.}
\newblock{\em Geom. Topol.}, 15(2):609--676, 2011.
\newblock \href {http://arxiv.org/abs/0911.2483} {\path{arXiv:0911.2483}}.

\bibitem[Sch13]{Schreiber:DCCT}
U.~Schreiber.
\newblock{Differential cohomology in a cohesive $\infty$-topos.}
\newblock{Preprint, }
\newblock \href {http://arxiv.org/abs/1310.7930} {\path{arXiv:1310.7930}}.

\bibitem[SW09]{SW:PT_and_Fctrs}
U. Schreiber and K. Waldorf.
\newblock{Parallel transport and functors}.
\newblock {\em J. Homot. Relat. Struct.}, 4(1):187--244, 2009.
\newblock \href {http://arxiv.org/abs/0705.0452} {\path{arXiv:0705.0452}}.

\bibitem[SW11]{SW:Fctrs_v_forms}
U.~Schreiber and K.~Waldorf.
\newblock {Smooth functors vs. differential forms.}
\newblock {\em Homol. Homot. Appl.}, 13(1):143--203, 2011.
\newblock \href {http://arxiv.org/abs/0802.0663} {\path{arXiv:0802.0663}}.

\bibitem[SW17]{SW:Local_2-fctrs}
U.~Schreiber and K.~Waldorf.
\newblock {Local theory for 2-functors on path 2-groupoids.}
\newblock {\em J. Homot. Relat. Struct.}, 12(3):617--658, 2017.
\newblock \href {http://arxiv.org/abs/1303.4663} {\path{arXiv:1303.4663}}.

\bibitem[Sol18]{Soloviev:Dirac_Monopole_and_Kontsevich}
M.{\,}A. Soloviev.
\newblock {Dirac’s magnetic monopole and the Kontsevich star product}.
\newblock {\em J. Phys. A}, 51(9):095205, 21 pp., 2018.
\newblock \href {http://arxiv.org/abs/1708.05030} {\path{arXiv:1708.05030}}.

\bibitem[Sto96]{Stolz:Ricci_and_Witten}
S.~Stolz.
\newblock {A conjecture concerning positive Ricci curvature and the Witten genus.}
\newblock {\em Math. Ann.}, 304(4):785--800, 1996.

\bibitem[ST04]{ST:Elliptic_Objects}
S.~Stolz and P.~Teichner.
\newblock {What is an elliptic object?}
\newblock {\em London Math. Soc. Lect. Note Ser.}, 308:247--343, 2004.

\bibitem[Sza13]{Szabo:2012hc}
R.{\,}J.~Szabo.
\newblock {Quantisation of higher abelian gauge theory in generalised differential cohomology.}
\newblock {\em Proc. Science}, 175:009, 65~pp., 2013.
\newblock \href {http://arxiv.org/abs/1209.2530} {\path{arXiv:1209.2530}}.

\bibitem[Sza19a]{Szabo:Durham}
R.{\,}J.~Szabo.
\newblock {Quantisation of magnetic Poisson structures.}
\newblock {\em Fortsch. Phys.}, 67(8--9):1910022, 13~pp., 2019.
\newblock \href {http://arxiv.org/abs/1903.02845} {\path{arXiv:1903.02845}}.

\bibitem[Sza19b]{Szabo:Corfu}
R.{\,}J.~Szabo.
\newblock {An introduction to nonassociative physics.}
\newblock {\em Proc. Science}, 347:100, 41~pp., 2019.
\newblock \href {http://arxiv.org/abs/1903.05673} {\path{arXiv:1903.05673}}.

\bibitem[Vis05]{Vistoli:Fib_Cats}
A.~Vistoli.
\newblock{Grothendieck topologies, fibred categories and descent theory.}
\newblock{\em Math. Surv. Monogr.}, 123:1--104, 2005.
\newblock \href {http://arxiv.org/abs/math/0412512} {\path{arXiv:math/0412512}}.

\bibitem[Wal07a]{waldorf}
K. Waldorf.	
\newblock {Algebraic structures for bundle gerbes and the Wess-Zumino term in conformal field theory.}
PhD Thesis, Universität Hamburg, 179~pp., 2007.

\bibitem[Wal07b]{Waldorf--More_morphisms}
K.~Waldorf.
\newblock {More morphisms between bundle gerbes}.
\newblock {\em Theory Appl. Categ.}, 18(9):240--273, 2007.
\newblock \href {http://arxiv.org/abs/math/0702652}
  {\path{arXiv:math/0702652}}.

\bibitem[Wal12a]{Waldorf:String_2-group}
K.~Waldorf.
\newblock {A construction of string 2-group models using a transgression-regression technique.}
\newblock {\em Contemp. Math.}, 584:99--115, 2012.
\newblock \href {http://arxiv.org/abs/1201.5052}
{\path{arXiv:1201.5052}}.

\bibitem[Wal12b]{WaldorfI}
K.~Waldorf.
\newblock {Transgression to loop spaces and its inverse I: Diffeological bundles and fusion maps.}
\newblock {\em Cah. Topol. Géom. Différ. Catég.}, LIII:162--210, 2012.
\newblock \href {https://arxiv.org/abs/0911.3212}
{\path{arXiv:0911.3212}}.

\bibitem[Wal16]{Waldorf:Transgression_II}
K.~Waldorf.
\newblock {Transgression to loop spaces and its inverse {II}: {G}erbes and fusion bundles with connection.}
\newblock {\em Asian J. Math.}, 20(1):59--115, 2016.
\newblock \href {http://arxiv.org/abs/1004.0031}
{\path{arXiv:1004.0031}}.

\bibitem[Wal18]{Waldorf:PT_in_2Bundles}
K.~Waldorf.
\newblock {Parallel transport in principal 2-bundles.}
\newblock {\em High. Struct.}, 2(1):57--115, 2018.
\newblock \href {http://arxiv.org/abs/1704.08542}
{\path{arXiv:1704.08542}}.

\end{thebibliography}
\end{document}